\documentclass{amsart}
\usepackage{amssymb, graphics, color, enumitem, mathrsfs}
\usepackage{appendix}

\usepackage[all,cmtip]{xy} 
\usepackage{tikz}
\usepackage{tikz-cd}
\usepackage[cyr]{aeguill}
\usepackage[linguistics]{forest}
\usepackage[margin = 1in]{geometry}
\usepackage{stmaryrd}
\usepackage{braket}
\usepackage{mathabx}
\usepackage{siunitx}
\newtheorem{lemma}{Lemma}[section]
\newtheorem{proposition}[lemma]{Proposition}
\newtheorem{corollary}[lemma]{Corollary}
\newtheorem{theorem}[lemma]{Theorem}

\newtheorem{example}[lemma]{Example}
\newtheorem{definition}[lemma]{Definition}
\newtheorem{remark}[lemma]{Remark}

\newtheorem*{Acknowledgement}{Acknowledgements}

\newtheorem{thmy}{Theorem}
\newenvironment{thmx}{\begin{thmy}}{\end{thmy}}
 \newtheorem{corox}[thmy]{Corollary} 
 \newtheorem{propy}{Proposition}
   
\newcommand*\bR{\mathbb R}
\newcommand*\bC{\mathbb C}

\newcommand{\dd}{\partial\bar{\partial}}

\DeclareMathOperator{\dvol}{dvol}
\DeclareMathOperator{\Id}{Id}

\newcommand{\kahler}{K\"ahler }
\DeclareMathOperator{\Bl}{Bl}
\DeclareMathOperator{\diag}{diag}
\DeclareMathOperator{\pt}{pt}
\DeclareMathOperator{\pr}{Pr}
\DeclareMathOperator{\DR}{dR}
\DeclareMathOperator{\Stab}{Stab}
\DeclareMathOperator{\id}{id}
\DeclareMathOperator{\QAC}{QAC}

\DeclareMathOperator{\ALE}{ALE}
\DeclareMathOperator{\Lie}{Lie}

\DeclareMathOperator{\Aut}{Aut}
\DeclareMathOperator{\Iso}{Iso}

\DeclareMathOperator{\sca}{SC}

\DeclareMathOperator{\inter}{int}

\DeclareMathOperator{\op}{op}

\DeclareMathOperator{\Rc}{Ric}

\DeclareMathOperator{\SC}{SC}
\DeclareMathOperator{\Vol}{Vol}
\DeclareMathOperator{\depth}{depth}
\DeclareMathOperator{\Euc}{Euc}
\DeclareMathOperator{\FS}{FS}
\DeclareMathOperator{\PGL}{\mathbb{P}GL}

\DeclareMathOperator{\SP}{SP}

\usepackage{color}

\newcommand\GL{\operatorname{GL}}

\begin{document}
\title[]
{Constant Scalar Curvature \kahler metrics on resolutions of an orbifold singularity of depth 1}

\author{Mehrdad Najafpour}
\address{D\'epartement de Math\'ematiques, Universit\'e du Qu\'ebec \`a Montr\'eal}
\email{najafpour\textunderscore ghazvini.mehrdad@uqam.ca}

\maketitle

\begin{abstract}
We construct new examples of constant scalar curvature Kähler metrics on suitable resolutions of certain constant scalar curvature Kähler orbifolds with singularities of type $\mathcal{I}$ in the sense of Apostolov-Rollin, along a suborbifold of complex codimension $k > 2$.\\
\end{abstract}

\tableofcontents

\numberwithin{equation}{section}

\section{Introduction}
In the 1950s, Eugenio Calabi in \cite{calabi1958improper, calabi2015kahler} proposed a natural notion of canonical Kähler metrics, namely extremal metrics. This involves fixing a Kähler class $\Omega$ and minimizing the Calabi functional:
$$\text{Cal}(\omega)=\displaystyle \int_M S(\omega)^2\omega^n,$$
 where $S(\omega)$ is the scalar curvature, within the space of Kähler metrics whose Kähler form $\omega$ belongs to $\Omega$. Constant scalar curvature Kähler (cscK) metrics are examples of extremal metrics, and Kähler-Einstein metrics are examples of cscK metrics.

The existence of Kähler-Einstein metrics for compact Kähler manifolds depends on the sign of the first Chern class of the Kähler manifold. When the first Chern class is negative, there is always a Kähler-Einstein metric, as independently proved by Thierry Aubin \cite{aubin1976equations} and Shing-Tung Yau \cite{yau1978ricci, yau1977calabi}. When the first Chern class is zero, there is always a Kähler-Einstein metric, as was shown by Shing-Tung Yau in \cite{yau1978ricci, yau1977calabi}. However, when the first Chern class is positive (also called Fano), the existence of Kähler-Einstein metric remained a well-known open problem for many years. In 2012, Xiuxiong Chen, Simon Donaldson, and Song Sun \cite{chen2015kahler1, chen2015kahler2, chen2015kahler3}, as well as independently Gang Tian \cite{MR3352459}, proved that for the Fano case, an algebraic-geometric criterion called K-stability implies the existence of a Kähler-Einstein metric. Additionally, the converse was proved by Robert Berman \cite{berman2016k}. 
Recently, other developments have arisen, such as \cite{lahdili2023einstein} and \cite{ammar2024delta}.

Sixty years after it was proposed, Calabi's program continues to represent the forefront of most active current research in complex geometry, yielding spectacular results. Yau-Tian-Donaldson \cite{donaldson2002scalar} conjectured more generally that there is an equivalence between the existence of a cscK metric on a polarized projective manifold and the K-polystability of that polarized manifold. Beyond the Kähler-Einstein Fano case, the conjecture was established for toric Kähler surfaces by Donaldson \cite{donaldson2002scalar} and for general toric varieties by Chen-Cheng. This conjecture was recently proven in 2021 by Chen-Cheng \cite{chen2021constant1, chen2021constant2, chen2018constant} in the toric case. In fact, it provides a necessary and sufficient condition, expressed in terms of the corresponding Delzant polytope, for a compact smooth toric manifold to admit a compatible Riemannian metric of constant scalar curvature.

In this paper, we focus on constant scalar curvature \kahler (cscK) metrics. In 2006, Arezzo and Pacard \cite{arezzo2006blowing} proved that if a compact manifold or compact orbifold $M$ with isolated singularities and no non-trivial holomorphic vector fields vanishing somewhere, admits a cscK metric, then the blow-up of $M$ at finitely many points also admit a cscK metric. In 2009 \cite{arezzo2009blowing}, they generalized the statement to situations where there are non-trivial holomorphic vector fields with zeros. In 2011, Arezzo, Pacard, and Singer \cite{arezzo2011extremal} proved the existence of an extremal metric on the blow-ups of a manifold at certain points, subject to assumptions on the position of the points, such as balancing and genericity conditions. Recently, in 2020, Seyyedali and Székelyhidi \cite{seyyedali2020extremal} extended the results of Arezzo and Pacard to obtain a cscK metric on blow-ups of a manifold along a submanifold. When the extremal metric is cscK and the automorphisms group is trivial, their result can be formulated as follows.
\begin{thmx}[\cite{seyyedali2020extremal}]\label{SS}
Let $(X,\omega_X)$ be a compact cscK complex manifold with discrete group of automorphisms (in particular, there are no non-trivial holomorphic vector fields on $X$) and $Y\subset X$ be a
submanifold of codimension $k$ greater than 2. Then $\Bl_Y^X$ admits a cscK metric in the class $[\omega_X]-\varepsilon^2[E]$ for sufficiently small $\varepsilon>0$, where $E$ is the exceptional divisor of the blow-up.
\end{thmx}
In this paper, we generalize Arezzo-Pacard-Singer and Seyyedali-Sz\'ekelyhidi results by constructing cscK metrics on the resolution of a certain orbifolds as follows.

\begin{thmx}[Theorem \ref{themain}]\label{B}
Suppose that $(X,\omega_X)$ is a compact cscK orbifold with no holomorphic vector fields, and such that the set of singular points $Y$ of $X$ is of complex co-dimension $>$ 2. Suppose, furthermore, that any point $p\in Y$ has a local orbifold uniformization chart of the form $\bC^{n-k}\times \left(\bC^k/\Gamma\right)$ where $\Gamma$ is a finite linear group of type $\mathcal{I}$ in the sense of Apostolov-Rollin of the form $(-w_0,w)$. If $\pi: \widehat{X} \to X$ is the partial resolution of $X$ obtained by a $(-w_0, w)$-weighted blow-up of $X$ along $Y$, then the class $[\omega_X] - \varepsilon^2 [E]$ admits a cscK metric for $\varepsilon > 0$ sufficiently small, where $E = \pi^{-1}(Y)$ is the exceptional divisor of the resolution $\pi: \widehat{X} \to X$.
\end{thmx}
Unless the singularity is of type $\mathcal{I}$ and of the form $(-w_0, 1, \ldots, 1)$, the resolution $\widehat{X}$ is not smooth. However, there is a possibly non-unique sequence of resolutions
$$\widehat{X}_l\to \widehat{X}_{l-1}\to \ldots \to \widehat{X}_1\to X$$
obtained by a sequence of weighted blow-ups with $\widehat{X}_l$. For such a sequence of resolutions, we show that Theorem \ref{B} can be applied iteratively to each $\widehat{X}_i$, yielding the following result.
\begin{corox}\label{coroc}
For $(X, \omega_X)$ as in Theorem \ref{B}, let $\widehat{X}_l\to \widehat{X}_{l-1}\to \ldots \to \widehat{X}_1\to X$ be a sequence of resolutions obtained through a sequence of weighted blow-ups with $\widehat{X}_l$ smooth. Then $\widehat{X}_l$ admits a cscK metric in a suitable Kähler class.
\end{corox}
Our strategy to prove this result consists in adapting the approach of \cite{seyyedali2020extremal} to the singular setting, using a coordinate-free description involving manifolds with corners.\\

The paper is organized as follows. In \S 2, we revisit concepts about weighted projective spaces and their singularities. In \S 3, we describe the analytical tools needed for the main result. We begin by defining manifolds with corners and blow-ups in the Melrose sense. We then define the Lie structure at infinity and Riemannian metrics from them. In \S 4, we focus on constant scalar curvature Kähler (cscK) metrics and give examples of cscK orbifolds with singularities of type $\mathcal{I}$ and having discrete automorphism group (Theorem \ref{example}). In \S 5, we construct a family of Kähler metrics $\widehat{\omega}_\varepsilon$ on the partial resolution $\widehat{X}$ of $X$ by gluing technique and using manifolds with corners. In \S 6, we focus on linear analysis through the linearization of constant scalar curvature, which requires considering the Lichnerowicz operator on weighted Hölder spaces using techniques introduced by Mazzeo in his study of conical metrics. The triviality of the kernel of the Lichnerowicz operator is due to the assumption on holomorphic vector fields on \(X\) (lemma \ref{green}). Based on techniques developed by Seyyedali and Székelyhidi, we proved that the twisted Lichnerowicz operator \( \widetilde{L}_\varepsilon \) is indeed uniformly boundedly invertible for sufficiently small \(\varepsilon\) (Proposition \ref{inv}). Finally, in \S 7, we use nonlinear analysis to find a potential \( u \) for obtaining a cscK metric \( \widetilde{\omega}_\varepsilon = \widehat{\omega}_\varepsilon + \sqrt{-1}\,\partial\bar{\partial}\phi_\varepsilon \) on the resolution. We begin by expressing the first Chern class of the resolution \( \widehat{X} \) in terms of that of \( X \) (Proposition \ref{chern}). This allows for the explicit calculation of the topological constant in the cscK equation on the resolution that we aim to solve (Proposition \ref{lambda}). The final step is to apply Banach's fixed-point theorem by carefully controlling the error term across four different regions, which depend on the distance to \( Y \), following a strategy implemented in Székelyhidi's book \cite{szekelyhidi2014introduction}. 

\begin{Acknowledgement}
I am grateful to Vestislav Apostolov and Frédéric Rochon, my Ph.D. supervisors, for introducing this problem and for their help.
\end{Acknowledgement}

\section{Weighted projective spaces and their singularities}
Since weighted projective spaces will play an important role in our construction, we will first review this notion and explain how it can be used to resolve some orbifold singularities.
\begin{definition}[Compact weighted projective space]
Let $w_0\in\mathbb{N}$ and $w=(w_1,\ldots,w_n)\in\mathbb{N}^n$. The compact weighted projective space corresponding to the weight vector $\overrightarrow{w}=(w_0,w)$ is the quotient 
$$\mathbb{CP}^{n}_{\overrightarrow{w}}=(\mathbb{C}^{n+1}\setminus\{0\})/\mathbb{C}^*,$$
with $\mathbb{C}^*$-action on $\mathbb{C}^{n+1}$ is given by 
$$t.(z_0,z_1,\ldots,z_n)=(t^{w_0}z_0,t^{w_1}z_1,\ldots,t^{w_n}z_n),  \forall  t\in \mathbb{C}^*.$$
\end{definition}
\begin{definition}[Non-compact weighted projective space]
Let $w_0\in\mathbb{N}$ and $w=(w_1,\ldots,w_n)\in\mathbb{N}^n$. The non-compact weighted projective space corresponding to the weight vector $\overrightarrow{w}=(-w_0,w)$
 is the quotient 
$$\mathbb{CP}^{n}_{\overrightarrow{w}}=(\mathbb{C}\times \mathbb{C}^{n}\setminus \mathbb{C}\times\{0\})/\mathbb{C}^*,$$
with $\mathbb{C}^*$-action given by 
$$t.(z_0,z_1,\ldots,z_n)=(t^{-w_0}z_0,t^{w_1}z_1,\ldots,t^{w_n}z_n),  \forall  t\in \mathbb{C}^*.$$
\end{definition}
\begin{remark}
The weighted projective space $\mathbb{CP}^{n}_{\overrightarrow{w}}$ (compact or non-compact) has the structure of a complex orbifold, since the $\mathbb{C}^*$-action is holomorphic, faithful and orientation preserving.
\end{remark}
\begin{example} 
The weighted projective space $\mathbb{CP}^{n}_{(1,1,\ldots,1)}$ is the usual complex projective space $\mathbb{CP}^{n}$.
\end{example}
\begin{example}
Let $r\in \mathbb{N}$, the non-compact weighted projective space $\mathbb{CP}^{n}_{(-r,1,\ldots,1)}$ is the total space of the line bundle $\mathcal{O}_{\mathbb{CP}^{n-1}}(-r)$.
\end{example}
To describe singularities of the weighted projective space, we begin with the following observation. Assume that $w_0> 1$, and let $p=(1,0,\ldots,0)\in \mathbb{C}^{n+1}\setminus \{0\}$. Under the action of $t\in \mathbb{C}^*$, $(1,0,\ldots,0)$ is taken to $t.p=(t^{w_0},0,\ldots,0)$. So the stablizer of $p$ is given by
$$\Stab(p)=\{t\in \mathbb{C}^*:t^{w_0}=1 \}\cong \mathbb{Z}_{w_0}.$$
With the same idea we can see that for a point $p=[z_0,z_1,\ldots,z_n]\in \mathbb{CP}^{n}_{(w_0,w)}$, if we set
$$d=\gcd\{w_i:z_i\neq 0,0\leq i\leq n\},$$
then we get two cases:
\begin{enumerate}
\item Case 1: If $d=1$, then $p$ is a non-singular point.
\item Case 2: If $d\neq1$, then $p$ is a singular point. Near the point $p$, the weighted projective space is locally like $\mathbb{C}^{k-1}\times (\mathbb{C}^{n-k+1}/\mathbb{Z}_d)$ where $k=\text{card}\{i:z_i\neq 0,0\leq i\leq n\}$ and $\mathbb{Z}_d$ acts on $ \mathbb{C}^{n-k+1}$ by
$$e^{\frac{2\pi i}{d}}.(\xi_{i_1},\xi_{i_2},\ldots,\xi_{i_{n-k+1}})=(e^{\frac{2\pi i}{d}w_{i_1}}\xi_{i_1},e^{\frac{2\pi i}{d}w_{i_2}}\xi_{i_2},\ldots,e^{\frac{2\pi i}{d}w_{i_{n-k+1}}}\xi_{i_{n-k+1}}),$$
and $\{i_1,\ldots,i_{n-k+1}\}=\{i:z_i= 0\}$. See pages 133-134 \cite{joyce2000compact} for more details.
\end{enumerate}
For the non-compact weighted projective space $\mathbb{CP}^{n}_{(-w_0,w)}$, the singular points correspond to the singular points of $\mathbb{CP}^{n-1}_{w}$ in $\mathbb{CP}^{n}_{(-w_0,w)}$ given by
$$\mathbb{CP}^{n-1}_{w}=\{[z_0,z_1,\ldots,z_n]\in \mathbb{CP}^{n}_{(-w_0,w)} :z_0=0\}\subset \mathbb{CP}^{n}_{(-w_0,w)}.$$

\begin{remark}\label{singularities of wiethted}
The above discussion shows that
\begin{enumerate}
\item The compact weighted projective space $\mathbb{CP}^{n}_{(w_0,w)}$ is smooth if and only if $w_0=1$ and $w=(1,\ldots,1)$, i.e, it is the usual weighted projective space $\mathbb{CP}^{n}$.
\item The non-compact weighted projective space $\mathbb{CP}^{n}_{(-w_0,w)}$ is smooth if and only if $w=(1,\ldots,1)$ and $\mathbb{CP}^{n}_{(-w_0,1,\ldots,1)}\cong \mathcal{O}_{\mathbb{CP}^{n-1}}(-w_0)$. 
\end{enumerate}
\end{remark}

\begin{remark}\label{singularities of wiethted2}
A non-compact weighted projective space $\mathbb{CP}^{n}_{(-w_0,w)}$ has only isolated singularities, if and only if 
\begin{equation}\label{isocond2}
\gcd(w_i,w_j)=1, \forall i\neq j\in\{0,\ldots,n\}.
\end{equation}
In this case, the number of singularities is equal to the number of values $w_i$ that are not equal to 1 for $i\in {1,\ldots,n}$. If $w_i\neq 1$, the singularity corresponding to the point $[0:\ldots:0:1:\ldots:0]\in \mathbb{CP}^{n}_{(-w_0,w)}$ and is modeled by the orbifolds $\mathbb{C}^{n}/\mathbb{Z}_{w_i}$ obtained by the action of the cyclic group of $w_i$-th roots of unity given by
$$e^{\frac{2\pi i}{w_i}}.(\xi_{0},\xi_{i_1},\ldots,\xi_{i_{n-1}})=(e^{-\frac{2\pi i}{w_i}w_{0}}\xi_{0},e^{\frac{2\pi i}{w_i}w_{i_1}}\xi_{i_1},\ldots,e^{\frac{2\pi i}{w_i}w_{i_{n-1}}}\xi_{i_{n-1}}),$$
with $\{i_1,\ldots,i_{n-1}\}=\{j\in \{1,\ldots,n\}:z_j\neq i\}$.
\end{remark}
\begin{example}\label{aaa}
Let $\Gamma_{(-w_0,w)}$ be the cyclic group of $w_0$-roots unity defined by 
 $$\Gamma_{(-w_0,w)}=\braket{\diag(\xi^{w_1},\ldots,\xi^{w_n})}\cong\mathbb{Z}_{w_0},$$
 where $\xi=e^{\frac{2\pi i}{w_0}}$. Consider the action of $\Gamma_{(-w_0,w)}$ on $\mathbb{C}^n$ given by
$$(z_1,\ldots,z_n)\mapsto (\xi^{w_1}z_1,\ldots,\xi^{w_n}z_n).$$
The group $\Gamma_{(-w_0,w)}$ is a  finite subgroup of $U(n)$, so $X=\mathbb{C}^n/\Gamma_{(-w_0,w)}$ has the structure of an orbifold. It has an isolated singularity at the origin if and only if
\begin{equation}\label{isocond}
\gcd(w_0,w_i)=1, \forall i\in\{1,\ldots,n\}.
\end{equation}
\end{example}

\begin{definition}\label{projectivelinebundle}
Assume that the complex orbifold  $X=\mathbb{C}^n/\Gamma_{(-w_0,w)}$ has an isolated singularity at the origin. A blow-up of the origin is given by the non-compact weighted projective space $\mathbb{CP}^{n}_{(-w_0,w)}$ with the blow-down map $\beta:\mathbb{CP}^{n}_{(-w_0,w)}\to \mathbb{C}^n/\Gamma_{(-w_0,w)}$ given by
$$[z_0,z_1,\ldots,z_n]\mapsto (z_0^{\frac{w_1}{w_0}}z_1,\ldots, z_0^{\frac{w_n}{w_0}}z_n).$$
The above map is well defined because for any $t\in\mathbb{C}^*$, 
$$\beta([t^{-w_0}z_0,t^{w_1}z_1,\ldots,t^{w_n}z_n])=((t^{-w_0}z_0)^{\frac{w_1}{w_0}}t^{w_1}z_1,\ldots, (t^{-w_0}z_0)^{\frac{w_n}{w_0}}t^{w_n}z_n)=(z_0^{\frac{w_1}{w_0}}z_1,\ldots, z_0^{\frac{w_n}{w_0}}z_n).$$
The exceptional divisor $E=\beta^{-1}(0)$ of this weighted blow-up is naturally identified with the compact weighted projective space $E=\mathbb{CP}^{n-1}_{w}$.
\end{definition}

\begin{remark}
Note that $\mathbb{C}^n/\Gamma_{(-1,1,\ldots,1)}=\mathbb{C}^n$, so
$$\Bl^{\mathbb{C}^n}_0=\mathbb{CP}^{n}_{(-1,1,\ldots,1)}=\mathcal{O}_{\mathbb{CP}^{n-1}}(-1).$$
\end{remark}
Definition \ref{projectivelinebundle} shows that for any $w_0\in \mathbb{N}$, the non compact weighted projective space $\mathbb{CP}^{n}_{(-w_0,w)}$ is a holomorphic line bundle over the compact weighted projective space $\mathbb{CP}^{n-1}_{w}$.
\begin{definition}\label{weightedtautological}
For $w\in \mathbb{N}^n$, the tautological line bundle of the compact weighted projective space $\mathbb{CP}^{n-1}_{w}$ is defined by
$$\mathcal{O}_{\mathbb{CP}^{n-1}_{w}}(-1):=\mathbb{CP}^{n}_{(-1,w)},$$
Similarly for $w_0\in \mathbb{N}$ we define
$$\mathcal{O}_{\mathbb{CP}^{n-1}_{w}}(-w_0):=(\mathcal{O}_{\mathbb{CP}^{n-1}_{w}}(-1))^{\otimes w_0}=\mathbb{CP}^{n}_{(-w_0,w)}.$$
\end{definition}
\begin{proposition}\label{holomorphic sections of weighted projective}
For any $w_0\in \mathbb{N}$ and $w\in \mathbb{N}^n$, the holomorphic line bundle $\mathcal{O}_{\mathbb{CP}^{n-1}_{w}}(-w_0)$ has no non-trivial global holomorphic section.
\end{proposition}

\begin{proof}
Let $s\in H^0(\mathbb{CP}^{n-1}_{w},\mathcal{O}_{\mathbb{CP}^{n-1}_{w}}(-w_0))$ be a global holomorphic section. Then the composition with the blow-down map $\beta\circ s:\mathbb{CP}^{n-1}_{w}\to \mathbb{C}^n/\Gamma_{(-w_0,w)}$ is a holomorphic function defined on the compact weighted projective space $\mathbb{CP}^{n-1}_{w}$, so it is a constant $\beta\circ s \equiv c$. Clearly, this can only happen if $s=0$ as an elemmaent of $H^0(\mathbb{CP}^{n-1}_{w},\mathcal{O}_{\mathbb{CP}^{n-1}_{w}}(-w_0))$.
\end{proof}
On non-compact weighted projective spaces, an important class of singularities is given by those of type $\mathcal{I}$. Singularities of type $\mathcal{I}$ were introduced by Vestislav Apostolov and Yann Rollin in 2016 in \cite{apostolov2017ale}. Consider the congruence relation $\sim$ on $\mathbb{N}^{n+1}$ defined by
$$(a_0,a_1,\ldots, a_n)\sim(b_0,b_1,\ldots, b_n)\Longleftrightarrow a_0=b_0, a_i\equiv b_i \mod a_0.$$ 
We can check that if $(a_0,a)\sim(b_0,b)$ then $\mathbb{C}^n/\Gamma_{(-a_0,a)}\cong \mathbb{C}^n/\Gamma_{(-b_0,b)}$.

\begin{example}\label{firstI}
Consider $\mathbb{C}^{3}/\Gamma_{(-5,3,2,1)}$, and blow-up the origin by replacing $\mathbb{C}^{3}/\Gamma_{(-5,3,2,1)}$ with \linebreak $\mathbb{CP}^{3}_{(-5,3,2,1)}$. By remark \ref{singularities of wiethted2}, this new space is still singular with the isolated singularities at two points $[0:1:0:0]$ and $[0:0:1:0]$. These singularities are locally of the forms $\mathbb{C}^{3}/\Gamma_{(-3,1,2,1)}$ and $\mathbb{C}^{3}/\Gamma_{(-2,3,1,1)}$ so we can still blow them up by replacing these by $\mathbb{CP}^{3}_{(-3,1,2,1)}$ and $\mathbb{CP}^{3}_{(-2,1,1,1)}$. Since $(2,3,1,1)\sim (2,1,1,1)$, $\mathbb{CP}^{3}_{(-2,1,1,1)}$ is smooth. However, $\mathbb{CP}^{3}_{(-3,1,2,1)}$ still has a singularity locally in the form of $\mathbb{C}^{3}/\Gamma_{(2,1,1,1)}$. By blowing it up and replacing it with $\mathbb{CP}^{3}_{(-2,1,1,1)}$, we finally obtain a smooth complex manifold. There is a corresponding tree of singularities for $\mathbb{C}^{3}/\Gamma_{(-5,3,2,1)}$:
\begin{center}
\begin{forest}
  [\textsc{(5,3,2,1)}
    [\textsc{(3,1,2,1)}
     [\textsc{(2,1,1,1)}]
    ]
    [\textsc{(2,3,1,1)$\sim$(2,1,1,1)}]
  ]
\end{forest}

\end{center}
\end{example}
Now we define singularities of type $\mathcal{I}$.
\begin{definition}[Singularities of type $\mathcal{I}$]\label{typeI}
A singularity of an orbifold $\mathbb{C}^{n}/\Gamma_{(-w_0,w)}$ with $(-w_0,w)$ as in \eqref{isocond2} is a singularity of type $\mathcal{I}$ if either
\begin{enumerate}
\item $(w_0,w)\sim (w_0,1,\ldots, 1)$.\\
or
\item $(w_0,w)\sim (a_0,a)$ such that $\mathbb{CP}^{n}_{(-a_0,a)}$ has only isolated singularities of the forms $\mathbb{C}^n/\Gamma_{(-b_0,b)}$ of type $\mathcal{I}$. That is, after finitely many weighted blow-up we can end with a smooth manifold.
\end{enumerate}
\end{definition}
For singularities of type $\mathcal{I}$, we can represent a tree of singularities as follows. Start with $(a_0, a_1, \ldots, a_n)$, and inductively construct each branch corresponding to $a_i \neq 1$. At each step, a new singularity is obtained by:
$$(a_0, a_1, \ldots, a_{i-1}, a_i, a_{i+1}, \ldots, a_n) \to (a_i, a_1, \ldots, a_{i-1}, x, a_{i+1}, \ldots, a_n),$$
where $x \in \mathbb{N}$ is such that $x \equiv -a_0 \pmod{a_i}$. According to the definition of singularities of type $\mathcal{I}$, each branch is considered complete when it end up  to $(w_0, 1, \ldots, 1)$ with corresponding smooth weighted blow-up.
\begin{example}
The singularity of the orbifold $\mathbb{C}^{3}/\Gamma_{(-5,3,2,1)}$ in Example \ref{firstI} is a singularity of type $\mathcal{I}$ because we end with $(2,1,1,1)$ in each branch.
\end{example}

\begin{example}
Let $(w_0, w) = (p, q, 1, \ldots, 1)$, where $p$ and $q$ are two positive coprime integers such that $p > q$. Then, $(p, q, 1, \ldots, 1)$ is of type $\mathcal{I}$. A similar inductive procedure shows that by starting with $(p_0, q_0) = (p, q)$ and blowing-up at the stage $k$, by performing the Euclidean algorithm, we get $p_k = q_{k-1} < p_{k-1}$ and $0 < q_k < p_{k-1}$ such that $q_k \equiv -p_{k-1} \pmod{q_{k-1}}$. Clearly, in each stage, $p_k$ and $q_k$ are coprime, so we have an isolated singularity. Since $q_k < q_{k-1}$, we will eventually obtain a weight vector of the form $(p_N,1, 1, \ldots, 1)$.
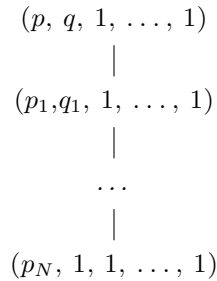
\begin{center}
\begin{forest}
  [\textsc{($p$, $q$, $1$, \ldots, $1$)}
    [\textsc{($p_1$,$q_1$, $1$, \ldots, $1$)}
     [\ldots
          [\textsc{($p_N$, $1$, $1$, \ldots, $1$)}]
    ]]].
\end{forest}
\end{center}
\end{example}
Definition \ref{aaa} shows that locally we can glue non-compact weighted projective spaces to resolve partially the isolated singularities of $\mathbb{C}^n/\Gamma_{(-w_0,w)}$. The fact that the singularity is of type $\mathcal{I}$ means that this type of partial resolution can be iterated finitely many times to obtain a smooth manifold. Globally, a complex orbifold with isolated singularities of type $\mathcal{I}$ admits a resolution (which is not necessarily unique) denoted as $\widehat{X}$ of type $\mathcal{I}$. More generally, let $X$ be a compact complex orbifold of depth 1 (i.e, for each connected component \(\Sigma\) of \(X_{\text{sing}}\), the isotropy groups of the points of \(\Sigma\) are all isomorphic) with singularities of type $\mathcal{I}$ along a connected subset $Y$ with codimension $k$ greater than 2. We can define a type $\mathcal{I}$ resolution of $X$ along $Y$ as follows. Since $X$ has singularities of type $\mathcal{I}$ along $Y$ of codimension $k$, the normal bundle of $Y$ in $X$ is a fiber bundle over $Y$ with fibers of the form $\mathbb{C}^k/\Gamma_{(-w_0,w)}$, where $\Gamma_{(-w_0,w)}$ is a discrete finite subgroup of $U(k)$ as in Definition \ref{typeI}. Now, in a local chart 
$$\phi :U\to V_1\times V_2\subset \mathbb{C}^{n-k}\times (\mathbb{C}^k/\Gamma_{(-w_0,w)}),$$
with $\phi(U\cap Y)=V_1\times \{0\}$, we can consider the resolution $V_1\times \widehat{V}_2$ with $\widehat{V}_2=\beta^{-1}(V_2)$, where
$$\beta:\mathbb{CP}^k_{(-w_0,w)}\to \mathbb{C}^k/\Gamma_{(-w_0,w)},$$
is the natural blow-down map of Definition \ref{weightedtautological}. That is, we can consider the resolution $\beta_U: \widehat{U} \to U$, inducing a commutative diagram
$$\xymatrix{\widehat{U}\ar[rr]^{\widehat{\phi}}\ar[d]^{\beta_U}&&V_1\times \widehat{V}_2\ar[d]^{\Id\times \beta}\\U\ar[rr]^{\phi}&&V_1\times V_2,}$$
with $\widehat{\phi}$ a biholomorphism. This resolution does not depend on the choice of coordinates. Indeed, if $f:V_1\times V_2\to V_1\times V_2$ is a biholomorphism sending $V_1\times \{0\}$ onto $V_1\times \{0\}$, then it lifts to a $\Gamma_{(-w_0,w)}$-equivariant biholomorphism $\widetilde{f}:V_1\times \widetilde{V}_2\to V_1\times \widetilde{V}_2$ with $\widetilde{V}_2$ the lift of $V_2$ to $\mathbb{C}^k$ under the quotient map $q:\mathbb{C}^k\to \mathbb{C}^k/\Gamma_{(-w_0,w)}$.\\

The differential of $\widetilde{f}$ in the $\widetilde{V}_2$ factor induces, when restricted to $V_1 \times \{0\}$, a biholomorphism
$$d\widetilde{f}_2:V_1\times\mathbb{C}^k\to V_1\times\mathbb{C}^k, $$
which is linear in the $\mathbb{C}^k$ factor and $\Gamma_{(-w_0,w)}$-equivariant. In paticular, it has a weighted projectivization
$$\mathbb{P}_w(d\widetilde{f}_2):V_1\times\mathbb{P}_w(\mathbb{C}^k)\to V_1\times\mathbb{P}_w(\mathbb{C}^k).$$
One can then easily check that the biholomorphism $f:V_1\times V_2\to V_1\times V_2$ cab be lifted to a biholomorphism
$$\widehat{f}:V_1\times \widehat{V}_2\to V_1\times \widehat{V_2},$$
given by $\mathbb{P}_w(d\widetilde{f}_2)$ on $V_1\times\mathbb{P}_w(\mathbb{C}^k)$ and by $f$ on $V_1\times(\widehat{V}_2\setminus \mathbb{P}_w(\mathbb{C}^k))=V_1\times(V_2\setminus \{0\})$.

Clearly, this biholomorphism induces the commutative diagram
$$\xymatrix{V_1\times \widehat{V}_2\ar[rr]^{\widehat{f}}\ar[d]^{\Id\times \beta}&&V_1\times \widehat{V}_2\ar[d]^{\Id\times \beta}\\V_1\times V_2\ar[rr]^{f}&&V_1\times V_2,}$$
confirming that the resolution $\widehat{U}$ does not depend on the choice of coordinates. This means that we can consider a partial resolution $\pi: \widehat{X} \to X$ along $Y$ in which an open set $U$ as described above corresponds to the local resolution $\beta_U: \widehat{U} \to U$, and away from $Y$ is simply the identity map.\\
We say that the partial resolution $\widehat{X}$ is the $(-w_0,w)$-weighted blow-up of $X$ along $Y$. We denote by $E=\pi^{-1}(Y)$ the exceptional divisor of this weighted blow-up. Notice that $\pi:E\to Y$ is a fiber bundle with fibers $\mathbb{CP}_w^{k-1}$. In fact, there is a rank $k$ complex vector bundle $W\to Y$ and a fiberwise $\Gamma_{(-w_0,w)}$-action on $W$ such that $N_X(Y)=W/\Gamma_{(-w_0,w)}$ and $E=\mathbb{P}_w(W)\label{wbb}$ is the fiberwise weighted projectivization of $W$. In general, $\widehat{X}$ is not smooth and has orbifold singularities of depth one along suborbifolds corresponding to the isolated singularities of the fibers of $E \to Y$. In particular, these suborbifolds are covers of $Y$. Assuming the initial singularity along $Y$ is of type $\mathcal{I}$, we can perform weighted blow-ups along these suborbifolds. These weighted blow-ups can still have suborbifold singularities of depth one, but by performing additional weighted blow-ups, we can eventually obtain a smooth resolution after finitely many steps.\\
In other words, when the singularity along $Y$ is of type $\mathcal{I}$, we can find a finite sequence of weighted blow-ups
$$\widehat{X}_l\to \widehat{X}_{l-1}\to \ldots \to \widehat{X}_1\to X$$
with $\widehat{X}_1=\widehat{X}$ and $\widehat{X}_l$ smooth.

\begin{proposition}\label{bihologrroupblow}
Let $X$ is a compact complex orbifold of complex dimension $n$. Suppose that $X$ has only depth one singularities of type $\mathcal{I}$ and we ‌denote by $Y$ the‌ singular part of $X$. 
Assume that the complex codimension $k$ of $Y$ is greater than 2. For any partial resolution $\widehat{X}$ of $X$ of type $\mathcal{I}$, $\mathfrak{h}(\widehat{X})$ is naturally realized as the Lie subalgebra of $\mathfrak{h}(X)$ consisting of holomorphic vector fields on $X$ tangent to $Y$.
\end{proposition}
\begin{proof}
We proceed with a proof similar to Proposition 6.4.1 in \cite{Gauduchon}. Firstly, we demonstrate that any (real) holomorphic vector field, denoted as $\widehat{V}$, on $\widehat{X}$ descends to a (real) holomorphic vector field, denoted as $V$, on $X$ which is tangent to $Y$. Indeed, since $\pi:\widehat{X}\setminus E\to X\setminus Y$ is a biholomorphism, $\widehat{V}$ naturally descends to a holomorphic vector field on $X\setminus Y$. Via the short exact sequence of vector bundles
\begin{equation*}\label{firstblo}
\xymatrix{0\ar[r]&TE\ar[r]& T\widehat{X}\ar[d]\ar[r]&T\widehat{X}/TE \displaystyle
\ar[r] &0,\\&&E,}
\end{equation*}
the restriction $\widehat{V}{\big|}_E$ determines a holomorphic section of the normal bundle 
$$T\widehat{X}/TE \cong N_{\widehat{X}}(E).$$
On the other hand, on $E$, the restriction of $\pi$ induces a fiber bundle
\begin{equation*}\label{firstblo}
\xymatrix{E \cong \mathbb{P}_{w}(W)\ar[d]\\ Y,}
\end{equation*}
where $W$ is such that $N_{X}(Y)=W/\Gamma_{(-w_0,w)}$ for a $\Gamma_{(-w_0,w)}$ of type $\mathcal{I}$ and $\mathbb{P}_{w}(W)$ is the weighted fiberwise projectivization of $W$. Thus, each fiber of $N_{X}(Y)$ corresponds to $\mathbb{CP}^{k-1}_{w}$ with the restriction of $N_{\widehat{X}}(E)$ corresponding to $\mathcal{O}_{\mathbb{CP}^{k-1}_{w}}(-w_0)$. Using Proposition \ref{holomorphic sections of weighted projective}, it has no non-trivial global holomorphic section. Consequently, $\widehat{V}{\big|}_E$ is tangent to $E$. Now via $\pi_*:TE\to \pi^*TY$, $\widehat{V}{\big|}_E$ induces a section $\pi_*(\widehat{V}{\big|}_E)\in H^{0}(E,\pi^*TY)$. On each fiber of $\pi$, $\pi^*TY$ is trivial, so its only holomorphic sections are constant sections. This implies that $V=\pi_*(\widehat{V})$ is a well-defined continuous vector field on $X$, which is tangent to $Y$ and holomorphic on $X\setminus Y$. By Hartogs' theorem, $V$ is holomorphic everywhere on $X$, hence it belongs to the Lie subalgebra,
$$\mathfrak{h}_Y(X)=\{\xi\in \mathfrak{h}(X):\xi{\big|}_Y\in H^{0}(Y,T^{1,0}Y)\}.$$
Furthermore, the resulting map from $\mathfrak{h}(\widehat{X})$ to $\mathfrak{h}_Y(X)$ is a Lie algebra morphism. This map is clearly injective, and we shall now show that it is surjective, establishing that it is a Lie algebra isomorphism.

Indeed, any element $V$ of $\mathfrak{h}_Y(X)$ lifts to a (real) holomorphic vector field, say $\widehat{V}$, on $\widehat{X}\setminus E$, via the isomorphism $\pi:\widehat{X}\setminus E\to X\setminus Y$. We need to check that $\widehat{V}$ extends to all of $\widehat{X}$ as an element of $\mathfrak{h}(\widehat{X})$. Since $\widehat{V}$ is holomorphic on $\widehat{X}\setminus E$, we only have to worry about the behavior of $\widehat{V}$ near points $p$ of the exceptional divisor $E$. Since $p$ has a neighborhood in $\widehat{X}$ isomorphic to $\mathbb{C}^{n-k}\times\mathbb{CP}^{k}_{(-w_0,w)}$, we can use local holomorphic coordinates
$$(y_1,\ldots,y_{n-k},Z_1,\ldots, Z_{k})\in \mathbb{C}^{n-k}\times \mathbb{C}^{k}/\Gamma_{(-w_0,w)},$$
near $y=\pi(p)$. Relative to these coordinates, since $V{\big|}_Y$ is tangent to $Y$, $V$ will be here conveniently regarded as a holomorphic (complex) vector field of type $(1,0)$ in the form of:
$$V=\displaystyle\sum_{i=1}^{n-k}a^i\dfrac{\partial}{\partial y_i}+\displaystyle\sum_{i,j=1}^{k}b^{ij}Z_j\dfrac{\partial}{\partial Z_i},$$
where the $a^{i}$ and $b^{ij}$ are holomorphic functions of $y_1,\ldots,y_{n-k}, Z_1,\ldots, Z_{k}$. Now, since this is an orbifold chart, 
$V$ must be $\Gamma_{(-w_0,w)}$-invariant as well, which means that $b^{ij}=0$ for $i\neq j$ with $w_i\neq w_j$ and $b^{ii}$ does not depend on $(Z_1,\ldots, Z_{k})$ whenever $w_i\neq 1$. Correspondingly, near $p$ we can use coordinates $(y_1,\ldots,y_{n-k},z_1,\ldots, z_{k})$ such that 
$$\pi(y_1,\ldots,y_{n-k},z_1,\ldots, z_{k})=(y_1,\ldots,y_{n-k},z_1^{\frac{w_1}{w_0}},z_1^{\frac{w_2}{w_0}}z_2,\ldots, z_1^{\frac{w_k}{w_0}}z_{k})=(y_1,\ldots,y_{n-k},Z_1,\ldots, Z_{k}),$$
where we assume without loss of generality that $z_1\neq 0$ for the weighted projective class corresponding to $p$. Then
$$\pi^*(Z_i\dfrac{\partial}{\partial Z_i})=\left\{
    \begin{array}{lr}
        \dfrac{w_0}{w_1}z_1\dfrac{\partial}{\partial z_1}-\displaystyle\sum_{i=2}^{k}\dfrac{w_i}{w_1}z_i\dfrac{\partial}{\partial z_i} &: i=1\\
        z_i\dfrac{\partial}{\partial z_i} &: i>1
    \end{array}\right.$$
and $\pi^*(Z_i\dfrac{\partial}{\partial Z_j})=z_i\dfrac{\partial}{\partial z_j}$ for $i\neq j\neq 1$ with $w_i=w_j=1$. Moreover, if $w_1=1$, then for $i\neq 1$ with $w_i=1$, $\pi^*(Z_i\dfrac{\partial}{\partial Z_1})=w_0z_1z_i\dfrac{\partial}{\partial z_1}-\displaystyle\sum_{j=2}^{k}\dfrac{w_j}{w_1}z_jz_i\dfrac{\partial}{\partial z_j}$. This demonstrates that $\widehat{V}$ is well-defined on the whole of $\widehat{X}$ as an element of $\mathfrak{h}(\widehat{X})$.
\end{proof}

\section{Manifolds with corners and Lie structures at infinity}
\label{sec: Manifolds with corners and Lie structures at infinity}
Many problems in differential geometry and partial differential equations often involve manifolds with boundaries, such as boundary value problems. The category of smooth manifolds alone presents challenges, and even the category of manifolds with boundaries is not sufficiently convenient, as the product of two manifolds with boundaries does not yield a manifold with a boundary. This complexity prompts the introduction of the category of manifolds with corners. Manifolds with corners arise in various ways, as will be shown later. Constructions leading to manifolds with corners include the desingularization of singular varieties (blow-up) and the compactification of non-compact spaces.\\
Melrose calculus, also known as pseudodifferential operator calculus or boundary value calculus, is a framework that extends the theory of pseudodifferential operators to manifolds with boundaries and corners. Developed by Richard Melrose in the 1980s, it has found applications in microlocal analysis, geometric analysis, etc. The key idea in Melrose calculus is to study operators that behave like pseudodifferential operators near the boundary or corners of a manifold. Pseudodifferential operators are a class of linear operators with symbols that have asymptotic expansions, playing a fundamental role in harmonic analysis and partial differential equations. In Melrose calculus, these operators are generalized to handle boundary value problems and singularities.\\
We begin this chapter with a brief introduction to manifolds with corners and blow-ups in the sense of Melrose. The definition of manifolds with corners is not universally agreed upon. In this chapter, we follow the approach outlined by Richard Melrose in his book 'Differential Analysis on Manifolds with Corners' \cite{melrose1996differential}. In the second part of this chapter, we introduce Lie structures at infinity based on the series of papers by Ammann-Lauter-Nistor in \cite{ammann2004geometry}.\\

The definition of a manifold with corners below is based on the model spaces 
$$\mathbb{R}^n_k = [0,\infty)^k \times \mathbb{R}^{n-k} = \{ x \in \mathbb{R}^n \, | \, x_i \geq 0, 1 \leq i \leq k \},$$
which are products of half-lines and lines. The topology on $ \mathbb{R}^n_k$ is inherited from $ \mathbb{R}^n $. In particular, a subset $ \Omega \subset \mathbb{R}^n_k $ is open if there exists an open set $ \Omega_0 \subset \mathbb{R}^n $ such that $ \Omega = \Omega_0 \cap \mathbb{R}^n_k $. If $\Omega \subset \mathbb{R}^n_k$ is an open subset, we define $ C^\infty(\Omega)$ as the set of functions $ u: \Omega \to \mathbb{C}$ such that $ u$ is smooth in $\Omega^{\circ}$ with all derivatives bounded on $K \cap \Omega^{\circ}$ for all subsets $K \Subset \Omega$. Here, $\Omega^{\circ}=\Omega \cap \inter(\mathbb{R}^n_k)$ and $K\Subset \Omega$ means that the closure of $K$ is a compact subset of $\Omega$. Similarly, smooth structures, diffeomorphisms, and partitions of unity are defined in a natural way on $\mathbb{R}^n_k$.

 \begin{definition}[Smooth structure with corners]
Let $X$ be a Hausdorff topological space. A chart with corners on $X$ is a map $\phi: U \rightarrow \mathbb{R}^n_k$, which is a homeomorphism from an open set $U \subseteq X$ onto an open subset of $\mathbb{R}^n_k$, for some $k$. Two charts $(\phi_1, U_1)$ and $(\phi_2, U_2)$ are said to be compatible if either $U_1 \cap U_2 = \emptyset$, or $\phi_2 \circ \phi_1^{-1}: \phi_1(U_1 \cap U_2) \rightarrow \phi_2(U_1 \cap U_2)$ is a diffeomorphism between open subsets of $\mathbb{R}^n_{k_1}$ and $\mathbb{R}^n_{k_2}$. An atlas on $X$ is a system of charts $\{(\phi_a, U_a)\}$ for $a \in A$, which are compatible in pairs and cover $X$, i.e., $X = \displaystyle\bigcup_{a \in A} U_a$. A smooth structure with corners on $X$ is a maximal atlas, i.e., an atlas that contains any chart compatible with each element of the atlas.

\end{definition}

\begin{definition}[$t$-manifold]
A $t$-manifold is a paracompact Hausdorff space $X$ endowed with some smooth structure with corners on it.

\end{definition}

\begin{definition}[Submanifold]\label{subcorner}
If $X$ is a $t$-manifold, then a submanifold $Y \subseteq X$ is a connected subset with the property that for each $y \in Y$, there exists a coordinate system $(\phi, U)$ around $y$, a linear transformation $G \in \GL(n,\mathbb{R})$, and an open neighborhood $\Omega '\subset \mathbb{R}^n$ of $0$ in terms of which 
$$\phi{\big|}_U : Y \cap U \to G .(\mathbb{R}^{n'}_{k'}\times \{0\})\cap \Omega',$$
 for some integers $n'$ and $k' = k'(y)$.
\end{definition}

\begin{definition}[\(p\)-submanifold]\label{p-submanifold}
A submanifold \(Y\) in a \(t\)-manifold \(X\) is called a \(p\)-submanifold if, for each \(y \in Y\), there exist local coordinates \(\phi\) at \(y\) within a coordinate neighborhood \(\Omega \subset X\), such that 
$$\phi(\Omega \cap Y) = L \cap \phi(\Omega),$$
where
$$L=\{x\in \mathbb{R}^n_k: x_{k-j+1}=\ldots=x_k= 0,  x_{k+1} = \ldots = x_{k+r} = 0\},$$
and \(j + r\) is the codimension of the submanifold. 
This implies that \(X\) and \(Y\) have a common local product decomposition. The 'p' in \(p\)-submanifold stands for 'product'. 
\end{definition}
\begin{example} 
The $\mathbb{S}^{n-1}_k := \{ x \in \mathbb{R}^n_k : \|x\| = 1 \}=\mathbb{S}^{n-1}\cap \mathbb{R}^n_k$ is a \( p \)-submanifold of the manifold with corners \( \mathbb{R}^n_k \).
\end{example}
We now study the notion of the boundary $\partial X$ for $t$-manifolds.
\begin{definition}[Boundary hypersurface]
For a general $t$-manifold set 
$$\partial_l X=\{p\in X: \text{there is a chart $\phi$ near $p$ with $\phi(p)\in\partial_l\mathbb{R}_k^n$}\},$$
where
$$\partial_l\mathbb{R}_k^n=\{x\in\mathbb{R}_k^n: x_i = 0 \text{  for exactly $l$ of the first $k$ indices}\}.$$
Then $X^\circ =\partial_0 X$. More generally, we shall set
$$\partial^l X = \overline{\partial_{l}X}=\displaystyle\bigcup_{r\geq l}\partial_r X.$$
Thus $\partial_l X$ consists precisely of the points in the boundary of $X$ laying in the interior of a corner of codimension $l$, while $\partial^l X$ consists of the points at which the boundary has codimension at least $l$. We also use the notation $\partial X = \partial^1 X$, so $X^\circ = X \setminus \partial X$. A boundary hypersurface of a $t$-manifold $X$ is the closure of a component of $\partial_1 X$; the collection of boundary hypersurfaces will be denoted $M_1(X)$.
\end{definition}
\begin{definition}[Manifold with corners]
A manifold with corners is a Hausdorff space with a $C^\infty$ structure with corners (a $t$-manifold) such that each boundary hypersurface is a submanifold in the sense of Definition \ref{subcorner}.
\end{definition}
\begin{definition}
The cotangent space of a manifold with corners $X$ at $p \in X$ is defined by 
$$T_p^*X = I_pX /( I_pX)^2,$$
 where $I_pX$ is the ideal of smooth functions on $X$ vanishing at $p$:
\[I_pX = \{f \in C^\infty(X) : f(p) = 0\}.\]
Therefore, the tangent space at $p$ is defined by the dual of the cotangent bundle, i.e., 
$$T_pX = (I_pX / ( I_pX)^2)^*,$$
where
\[( I_pX)^2 = \{f \in C^\infty(X) : \exists k\in \mathbb{N}, g_1,h_1,\ldots,g_k,h_k\in I_pX \quad\text{s.t.}\quad f=\displaystyle \sum_{i=1}^kg_ih_i\}.\]
\end{definition}
\begin{example}
Examples and non-examples of manifolds with corners
\begin{enumerate}
\item The tetrahedron is a 3-manifold with corners but the square pyramid is not.
\item The tear drop $T = \{(x, y) \in \mathbb{R}^2 \mid x > 0, y^2 \leq x^2 - x^4\}$ is not a 2-manifold with corners, because its unique boundary hypersurface is not a submanifold.
\end{enumerate}
\end{example}
\begin{lemma}
In a manifold with corners, each boundary hypersurface $H$ has a global defining function in the sense that there exists $\rho_H \in C^\infty(X)$ such that $\rho_H\geq 0$, $H = \rho_H^{-1}(0)$ is the boundary hypersurface and the differential $d\rho_H$ is nowhere zero on $H$. Near each point of $H$, there are local coordinates with $\rho_H$ as the first element.
\end{lemma}

\begin{example}
Consider a square in $\mathbb{R}^2$ defined by $[0,1] \times [0,1]$. This is a manifold with corners. Let us denote its boundary hypersurfaces as $H_1, H_2, H_3, H_4$, corresponding to the left, right, bottom, and top sides of the square, respectively. The boundary defining functions can be denoted as $\rho_{H_1}, \rho_{H_2}, \rho_{H_3}, \rho_{H_4}$. Specifically:
\begin{align*}
    \rho_{H_1}(x, y) &= x \text{ for } H_1, \\
    \rho_{H_2}(x, y) &= 1 - x \text{ for } H_2, \\
    \rho_{H_3}(x, y) &= y \text{ for } H_3, \\
    \rho_{H_4}(x, y) &= 1 - y \text{ for } H_4.
\end{align*}
These functions are $C^\infty$ and considered as defining functions for the respective boundary hypersurfaces.
\end{example}
\begin{definition} Let $U \subseteq \mathbb{R}^n_k$ be open. For each $u = (u_1, \ldots, u_n)$ in $U$, define the boundary depth of $u$ in $U$ denoted by $\depth_U (u)$ as the number of $u_1, \ldots, u_k$ which are zero. In other words, $\depth_U (u)$ is the number of boundary hypersurfaces of $U$ containing $u$.
\end{definition}
\begin{definition}[Boundary depth]
Let $X$ be an $n$-manifold with corners. For $p \in X$, choose a local chart $(U, \phi)$ around $p$ on the manifold $X$ with $\phi(p) = u$ for $u \in U\subseteq \mathbb{R}^n_k$, and define the depth $\depth_X (x)$ of $x$ in $X$ as $\depth_X (x) = \depth_U (u)$. This is independent of the choice of $(U, \phi)$.
\end{definition}
\begin{example}
Let $X$ be a manifold with boundary and $p\in X$. If $p\in \inter (X)$, then the boundary depth of $p$ is equal to 0 and if $p\in \partial X$, then the boundary depth of $p$ is equal to 1.
\end{example}
Now, let us define the concept of stratified spaces. These spaces are generalization of manifolds with corners.

\begin{definition}[Stratified space]\label{Stratified space}
A stratified space of dimension $n$ is a pair $(X,S)$, where $X$ is a locally compact, separable, metrizable space and $S$ is a stratification, that is, $S=\{S_i\}_{i\in I}$ is a locally finite collection of disjoint locally closed subsets of $X$ and $I$ is a poset such that:
\begin{enumerate}
\item $\displaystyle\bigcup_{i\in I} S_i=X$.
\item $S_i\cap \overline{S_j}$ is nonempty if and only if $S_i\subset\overline{S_j}$ , and this happens if and only if $i = j$ or $i < j$.
\item Each $S_i$ is a locally closed smooth submanifold of $\mathbb{R}^n$.
\end{enumerate}
The pieces $S_i$ are called strata. The set of strata is itself a poset, with the relation induced from inclusion.
\end{definition}
\begin{definition}[Depth of a stratified space]
The depth of a stratified space $(X, S)$ is the largest $k$ such that one can find $k + 1$ different strata with $S_1 < S_2 < \ldots < S_k < S_{k+1}$.
\end{definition}
\begin{example}
Any algebraic variety is naturally a stratified space.
\end{example}

\begin{example}
Consider the closed unit disk $\mathbb{D}^2$ in $\mathbb{R}^2$ and its boundary $\partial \mathbb{D}^2$, which is the unit circle. We can stratify this space as two stratas. The top stratum is the interior of the disk $\mathbb{D}^2$. It is an open subset of $\mathbb{R}^2$ and has dimension 2. And the bottom stratum is the boundary of the disk $\partial \mathbb{D}^2$, which is the unit circle. It is a closed subset of $\mathbb{R}^2$ and has dimension 1. This example illustrates a simple case of a stratified space where each stratum is a subset of the whole space, and they have different dimensions.
\end{example}

\begin{example}
Any manifold with corners $X$ of dimension $n$ is a stratified space.  For each $k = 0, \ldots, n$, define the $k$-th depth stratum of $X$ to be:
\[ S_k(X) = \{x \in X : \depth_X (x) = k\}=\partial_kX. \]
\end{example}

\begin{example}
An orbifold is naturally a stratified space. In fact, it is not hard to see that an orbifold is naturally a smoothly stratified space, with the corresponding manifold with fibered corners obtained by blowing-up the strata in an order compatible with the partial order. Notice that, in particular, an orbifold \(X\) is of depth 1 if and only if it is of depth 1 as a smoothly stratified space, i.e, for each connected component \(\Sigma\) of \(X_{\text{sing}}\), the isotropy groups of the points of \(\Sigma\) are all isomorphic.
\end{example}

\begin{definition}[Manifold with fibered corners]
We say that $(M,\phi)$ is a manifold with fibered corners if there is a partial order on the boundary hypersurfaces such that:
\begin{enumerate}
\item Any subset $I$ of boundary hypersurfaces such that $\displaystyle \bigcap _{H_i\in I}H_i\neq \emptyset$ is totally ordered.
\item If $H_i<H_j$ , then $H_i\cap H_j\neq \emptyset$, $\phi_i\vert_{H_i\cap H_j}:H_i\cap H_j\to S_i$ is a surjective submersion and $S_{ji}:=\phi_j(H_i\cap H_j)$ is one of the boundary hypersurfaces of the manifold with corners $S_j$ . Moreover, there is a surjective submersion $\phi_{ji}:S_{ji}\to S_i$ such that $\phi_{ji} \circ \phi_j=\phi_i$ on $H_i\cap H_j$.
\item The boundary hypersurfaces of $S_j$ are given by the $S_{ji}$ for $H_i<H_j$ .
\end{enumerate}
\end{definition}

\subsection{Blow-up in Melrose Sense}
Why consider blow-up at all? If we work on category of smooth manifolds, there isn't a compelling rationale for initiating any form of blow-up. Nonetheless, there are three interconnected scenarios where the process of blow-up can prove highly beneficial. These instances involve attempting to 'resolve' the following:
\begin{enumerate}
    \item A singular function, e.g., $f(x, y, z) = \sqrt{x^2 + y^2 + z^2}$.
    
    \item A singular space, e.g., $C = \{(t, x, y) \mid t^2 = x^2 + y^2, t \geq 0\}$.
    
    \item Degenerate vector fields, e.g., the span of $z_i \dfrac{\partial}{\partial z_i}$, $i = 1, 2, 3$, on $\mathbb{R}^3$.
\end{enumerate}
In Melrose blow-up (real blow-up) as introduced in \cite{hassell1995analytic}, the idea is simply to work in polar coordinates around the singular point. That is, we lift everything up to a manifold with a boundary by using the polar map. Before defining the Melrose blow-up, we need to recall the definition of sphere bundle.
\begin{definition}[Sphere bundle]
Let $E\to X$ be a smooth vector bundle. The sphere bundle of $E$, denoted by $S(E)$ is a fiber bundle whose fiber is an \(n\)-sphere and is defined as the set of (positive) rays in the bundle $E$, that is, 
$$S(E) = (E \setminus X)/\mathbb{R}^+.$$
If we fix a smooth metric on $E$, then the fiber of $S(E)$ over a point \(p\) is the set of all unit vectors in \(E_p\), the fiber over $p$ in $E$. When the vector bundle is the tangent bundle \(TX\), the unit sphere bundle is known as the unit tangent bundle.
\end{definition}
Blowing-up the origin in $\mathbb{R}^2$ simply amounts to the introduction of polar coordinates. We define $\mathbb{R}^2$ blown-up at $\{0\}$ to be
$$[\mathbb{R}^2, \{0\}] = \mathbb{S}^{1} \times [0, \infty)_r,$$
together with the associated blow-down map $\beta: \mathbb{S}^{1} \times [0, \infty)_r\to \mathbb{R}^2$ defined by $\beta(\omega,r)=r\omega$. This is a diffeomorphism from $[\mathbb{R}^2, \{0\}]$ onto $\mathbb{R}^2 \setminus \{0\}$ and has rank 1 at the boundary $\partial [\mathbb{R}^2, \{0\}] = \mathbb{S}^{1} \times \{0\}$, which projects to $\{0\}$.

\begin{figure}
\centering
\begin{tikzpicture}[x=0.75pt,y=0.75pt,yscale=-1,xscale=1]

\draw   (281.55,112.38) .. controls (281.55,87.7) and (301.56,67.69) .. (326.24,67.69) .. controls (350.92,67.69) and (370.93,87.7) .. (370.93,112.38) .. controls (370.93,137.06) and (350.92,157.07) .. (326.24,157.07) .. controls (301.56,157.07) and (281.55,137.06) .. (281.55,112.38) -- cycle ;
\draw    (370.93,112.38) -- (434.72,112.6) ;
\draw [shift={(436.72,112.6)}, rotate = 180.19] [color={rgb, 255:red, 0; green, 0; blue, 0 }  ][line width=0.75]    (10.93,-3.29) .. controls (6.95,-1.4) and (3.31,-0.3) .. (0,0) .. controls (3.31,0.3) and (6.95,1.4) .. (10.93,3.29)   ;
\draw    (211.71,111.99) -- (281.55,112.38) ;
\draw    (326.24,67.69) -- (326.7,4.43) ;
\draw [shift={(326.71,2.43)}, rotate = 90.41] [color={rgb, 255:red, 0; green, 0; blue, 0 }  ][line width=0.75]    (10.93,-3.29) .. controls (6.95,-1.4) and (3.31,-0.3) .. (0,0) .. controls (3.31,0.3) and (6.95,1.4) .. (10.93,3.29)   ;
\draw    (326.24,157.07) -- (325.71,228.43) ;
\end{tikzpicture}
\caption{\small Blowing-up the origin in $\mathbb{R}^2$}
\end{figure}
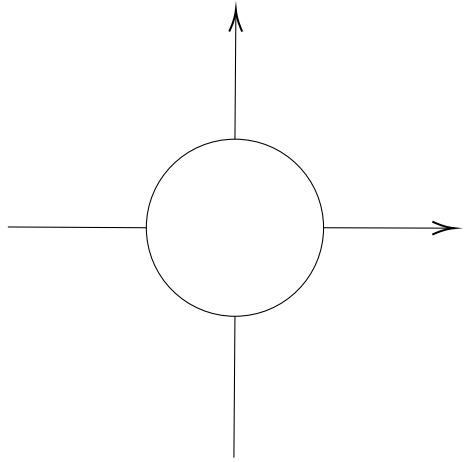

We can generalize the above idea and get that 
\begin{align*}
[\bR^{n},\{0\}]&=\mathbb{S}^{n-1}\times [0,\infty),\\
[\bC^{n},\{0\}]&=\mathbb{S}^{2n-1}\times [0,\infty).
\end{align*}
Another example is blowing-up the origin in $\mathbb{R}\times [0, \infty)_\varepsilon$. Again, by using polar coordinates, we define
$$
[\mathbb{R}\times [0, \infty)_\varepsilon, \{0\}] = \mathbb{S}_+^{1} \times [0, \infty)_r,
$$
where $\mathbb{S}_+^{n}=\{(x_0,\ldots,x_n)\in \mathbb{S}^n:x_n\geq 0 \}$ with the blow-down map $\beta: \mathbb{S}_+^{1} \times [0, \infty)_r \to \mathbb{R}\times [0, \infty)_\varepsilon$ defined by $\beta(\omega,r) = r\omega$.\\
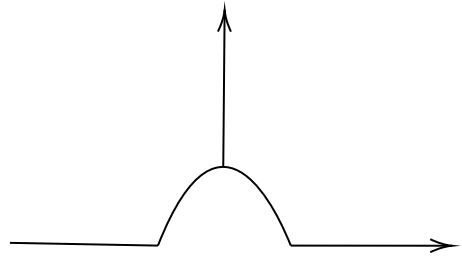
\begin{figure}
\centering
\tikzset{every picture/.style={line width=0.75pt}} 
\begin{tikzpicture}[x=0.75pt,y=0.75pt,yscale=-1,xscale=1]
\draw    (369.34,156.36) -- (448.6,156.43) ;
\draw [shift={(450.6,156.43)}, rotate = 180.05] [color={rgb, 255:red, 0; green, 0; blue, 0 }  ][line width=0.75]    (10.93,-3.29) .. controls (6.95,-1.4) and (3.31,-0.3) .. (0,0) .. controls (3.31,0.3) and (6.95,1.4) .. (10.93,3.29)   ;
\draw    (227.71,154.97) -- (302.33,156.36) ;
\draw    (335.28,116.48) -- (335.96,39.43) ;
\draw [shift={(335.97,37.43)}, rotate = 90.5] [color={rgb, 255:red, 0; green, 0; blue, 0 }  ][line width=0.75]    (10.93,-3.29) .. controls (6.95,-1.4) and (3.31,-0.3) .. (0,0) .. controls (3.31,0.3) and (6.95,1.4) .. (10.93,3.29)   ;
\draw    (302.33,156.36) .. controls (334.38,74.78) and (363.51,141.79) .. (369.34,156.36) ;
\end{tikzpicture}
\caption{\small Blowing-up the origin in $\mathbb{R}\times [0, \infty)_\varepsilon$}
\end{figure}

We can generalize the above idea and get that 
$$ [\bR^{n}\times [0,\infty),\{0\}]=\mathbb{S}_+^{n}\times [0,\infty),$$
and even more 
$$ [\bR^{n}_k,\{0\}]=\mathbb{S}_k^{n-1}\times [0,\infty).$$
\begin{align*}
\end{align*}

\begin{remark}
One can check that the action of $\GL(n)$ on $\mathbb{R}^n$ lifts to a smooth action of $\GL(n)$ on $[\mathbb{R}^n, \{0\}]$. This means that the Lie algebra, $\mathfrak{gl}(n)$, lifts to $[\mathbb{R}^n, \{0\}]$. Since the exponentials of linear vector fields are linear transformations, this implies that for each $i$ and $j$, there are smooth vector fields $V_{ij}$ on $[\mathbb{R}^n, \{0\}]$ such that
$$\beta_* V_{ij} = x_i \partial_{x_j}.$$
This shows that any smooth vector field on $\mathbb{R}^n$ which vanishes at $0$ lifts to a smooth vector field on $[\mathbb{R}^n, \{0\}]$:
$$\beta^{-1}_*(a_{ij}(x) x_i\partial_{x_j})=a_{ij}(r\theta) V_{ij}.$$
\end{remark}

In general, for a vector space $V$, the blow-up of $V$ at $0$ is defined as a set by
$$[V, \{0\}] =((V \setminus \{0\})/\mathbb{R}^{+})\bigsqcup(V\setminus \{0\}) .$$
Thus, the blow-up of $V$ at $\{0\}$ is the disjoint union of the projective sphere in $V$ and the complemmaent of $\{0\}$. The choice of a basis in $V$ gives a linear isomorphism $V \rightarrow \mathbb{R}^n$, which allows us to identify $[V, \{0\}]$ and $[\mathbb{R}^n,\{0\}]$. To show that the smooth structure, as a manifold with boundary, of $[V, \{0\}]$ is well-defined, we therefore need to check that the action of $\text{GL}(n)$ on $\mathbb{R}^n$ lifts to a smooth action of $\text{GL}(n)$ on $[\mathbb{R}^n, \{0\}]$.

If \(E\) is a vector bundle over a manifold with corners \(Y\), then we identify \(Y\) as the zero section of \(E\) and define \(E\) blown-up along \(Y\) to be
$$[E, Y] = \displaystyle\bigcup_{y \in Y} [E_y, \{0\}],$$
with the blow-down map $\beta : [E, Y] \to E$ that is, we simply blow-up the origin of each fiber. The blow-up \([E, Y]\) has a natural \(C^{\infty}\) structure as a manifold with corners. Now we can define the blow-up in Melrose sense formally.
\begin{definition}[Blow-up in Melrose sense along a submanifold]
Let $X$ be a smooth manifold with corners and \(Y \subset X\) be a closed \(p\)-submanifold. The blow-up of $X$ along $Y$, denoted $[X,Y]$, is a manifold with corners, given as a point set by 
$$[X, Y]= S(N_X(Y))\bigsqcup(X \setminus Y),$$
where $S(N_X(Y))$ represents the inward-pointing part of the spherical normal bundle $N_X(Y)$, i.e, the inward-pointing normal space \(N_{X,y}(Y)\) at a point \(y \in Y\) is defined as the quotient
$$N_{X,y}(Y) = T_y X /T_y Y,$$
and the spherical normal bundle is then given by
$$S(N_{X,y}(Y)) = (N_{X,y}(Y) \setminus \{0\})/{\mathbb{R}^+}.$$
The blow-up \([X, Y]\) has a natural \(C^{\infty}\) structure as a manifold with corners. There is a unique smooth map $[X,Y]\to X$ extending the identity on $X\setminus Y$ called the blow-down map. 
\end{definition}

\begin{example} 
Blowing-up $\bR^{n}$ along $\bR^{n-k}$ in Melrose sense:
\begin{align*}
[\bR^{n},\bR^{n-k}]&=[\bR^{n-k}\times \bR^{k},\bR^{n-k}\times \{0\}]\\
&=\bR^{n-k}\times[\bR^{k}, \{0\}]\\
&=\bR^{n-k}\times S(N_{\bR^{n}}(\bR^{n-k}))\bigsqcup(\bR^{k} \setminus \{0\})\\
&=\mathbb{S}^{k-1}\times [0,\infty)_r\times\bR^{n-k}.
\end{align*}
By $\bC^{n}\cong \bR^{2n}$ we get that
$$ [\bC^{n},\bC^{n-k}]=\mathbb{S}^{2k-1}\times [0,\infty)_r\times\bC^{n-k}.$$
We also have that
$$ [\bR^{n}\times [0,\infty)_\varepsilon,\bR^{n-k}\times \{0\}]=\mathbb{S}_+^{k}\times [0,\infty)_r\times\bR^{n-k},$$
$$ [\bC^{n}\times [0,\infty)_\varepsilon,\bC^{n-k}\times \{0\}]=\mathbb{S}_+^{2k}\times [0,\infty)_r\times\bC^{n-k}.$$
\end{example}

\begin{definition}[Lifting submanifolds]
If $Z \subset X$ is a closed subset of a manifold with corners, and $Y \subset X$ is a closed $p$-submanifold, we shall define the lift of $Z$ to $[X, Y]$ under the blow-down map $\beta: [X, Y]\to X$ in two distinct cases. First, if $Z \subset Y$, then
$$\beta(Z) = \beta^{-1}(Z).$$
Secondly, if $Z = \overline{Z \setminus Y}$, then
$$\beta(Z) =  \overline{\beta^{-1}(Z \setminus Y)}.$$
\end{definition}
\begin{definition}[Geometric fiber]
Let $M$ be a manifold with corners and let $V$ be a $C^\infty(M)$-module with module structure $C^\infty(M)\times V \ni (f,v)\to fv\in V$. For $p\in M$,  the ideal of smooth functions on $M$ vanishing at $p$, i.e, 
$$I_pM=\{f\in C^{\infty}(M):f(p)=0\},$$
is a complex subspace of $V$ and $V/((I_pM)V)$ is called the geometric fiber of $V$ at $p$. In general, geometric fibers are vector spaces with different dimensions. 
\end{definition}
Recall from algebra that $S\subset V$ is called a basis for a $C^\infty(M)$-module $V$, if any $v\in V$ could be written uniquely as $v=\displaystyle\sum_{s\in S}f_ss$ such that $f_s\in C^\infty(M)$ and $\{s\in S:f_s\neq 0\}$ is a finite set. A module is called free, if it has a basis. For free modules, geometric fibers have same dimension and this dimension is equal to the cardinality of the basis. A $C^\infty(M)$-module $V$ is called finitely generated projective, if there exist a $C^\infty(M)$-module $W$ such that $V\oplus W$ is free with finite basis. This is equivalent with the concept of locally free $C^\infty(M)$-modules.
\begin{theorem}[Serre-Swan]\label{Serre-Swan}
Let $V$ be a finitely generated projective $C^\infty(M)$-module. Then there exists a natural smooth vector bundle,
 $E\to M$, and a natural map $\iota:E\to V$  such that
$$V=\iota_*\Gamma(M,E),$$
and with the fiber of $E$ above $p\in M$ canonically identified with $V/((I_pM)V)$. Conversely, for any finite rank smooth vector bundle $E\to M$, $\Gamma(M,E)$ is a finitely generated projective $C^\infty(M)$-module.
\end{theorem}
We refer to \cite{karoubi2009k} for more details about the Serre-Swan theorem.
\begin{definition}[Structural Lie algebra of vector fields]\label{Structural Lie algebra of vector fields}
A structural Lie algebra of vector fields on a manifold $M$ (possibly with corners) is a subspace, $\mathcal{V} \subset \mathfrak{X}(M)$, of the real vector space of vector fields on $M$ with the following properties: 
\begin{enumerate}
\item
$\mathcal{V}$ is closed under Lie brackets.
\item
$\mathcal{V}$ is a finitely generated projective $C^{\infty}(M)$-module.
\item
the vector fields in $\mathcal{V}$ are tangent to all faces in $M$.
\end{enumerate}

\end{definition}

\begin{example}[Lie algebra of $b$-vector fields]\label{$b$-vector fields}
Let $M$ be a manifold with corners, and
\begin{align*}
\mathcal{V}_b(M) &= \{X \in \mathfrak{X}(M) : X \text{ is tangent to all faces of } M\}\\
&= \{X \in \mathfrak{X}(M) : X \rho_H=a_H\rho_H, a_H\in C^{\infty}(X),\forall H\in \partial M\},
\end{align*}
where $\rho_H$ is a boundary defining function of the hypersurface $H$. Then $\mathcal{V}_b(M)$ is a structural Lie algebra of vector fields.  This is the fundamental object in the theory of Melrose's b-calculus. A vector field $X \in \mathcal{V}_b(M)$ is called a $b$-vector field $X$. In local coordinates near a point $p\in \partial X$ any $b$-vector field $X$ is of the form
$$X = \displaystyle\sum_{i=1}^{k} a_i(x, y)x_i\partial_{x_i} + \displaystyle\sum_{i=1}^{n-k} b_i(x, y)\partial_{y_i},$$
where $x_1,\ldots,x_k$ are boundary defining functions, $y \in \mathbb{R}^{n-k}$, $a_i$ and $b_i$ are smooth functions. This shows that the Lie algebra of $b$-vector fields is generated in a neighborhood $U$ of $p$ by $x_j\partial_{x_j}$ and $\partial_{y_j}$ as a $C^{\infty}(M)$-module.
Any structural Lie algebra of vector fields on $M$ is contained in $\mathcal{V}_b(M)$.
\end{example}

\begin{example}[Lie algebra of scattering vector fields]
Let $M$ be a compact manifold with boundary, and let $x: M \to \mathbb{R}_+$ be a boundary defining function. Then the Lie algebra $\mathcal{V}_{\sca}(M) := x\mathcal{V}_b(M)$ does not depend on the choice of $x$, and the vector fields in $\mathcal{V}_{\sca}(M)$ are called scattering vector fields. In a local coordinate $(x, y_1,\ldots,y_{n-1})$ near a point $p$ any scattering vector field $X$ is of the form 
$$X = a(x, y)x^2\partial_{x} + \displaystyle\sum_{i=1}^{n-1} xb_i(x, y)\partial_{y_i},$$
where $a$ and $b_i$ are smooth functions. In fact the Lie algebra of scattering vector fields is generated in a neighborhood $U$ of $p$ by $x^2\partial_x$ and $x\partial_{y_i}$ as a $C^{\infty}(M)$-module.
\end{example}

\begin{example}[Lie algebra of edge vector fields]\label{Lie algebra of edge vector fields}
Let $M$ be a manifold with boundary $\partial M$, which is the total space of a fibration $\pi : \partial M \to B$ of smooth manifolds. We let
$$\mathcal{V}_e(M) = \{X \in \mathfrak{X}(M) : X \text{ is tangent to all fibers of } \pi \text{ at the boundary}\}$$
be the space of edge vector fields. Clearly, $\mathcal{V}_e(M)$ is closed under the Lie bracket. This is the fundamental object in the theory of Mazzeo's edge calculus. If $(x, y, z)$ are coordinates in a local product decomposition near the boundary, where $x$ corresponds to the boundary-defining function, $y$ corresponds to a set of variables on the base $B$ lifted to $\partial M$ through $\pi$, and $z$ is a set of variables in the fibers of $\pi$, then edge vector fields are generated by $x\partial_x$, $x\partial_y$, and $\partial_z$.
In other words, any $X \in \mathcal{V}_e(M)$ can be expressed locally as
$$X = \displaystyle a(x,y,z) x\partial _x +  \displaystyle\sum_{i=1}^{b} b_i(x,y,z) x\partial _{y_i} +  \displaystyle\sum_{i=1}^{f} c_j(x,y,z) \partial _{z_i},$$
where $a, b_i, c_j \in C^{\infty}(M)$.
\end{example}
\begin{example}\label{0-calculus}
As a particular case of edge vector fields, the fibration $\id:\partial M \to\partial M$ for a manifold with boundary $M$ yields the Lie algebra of $0-$vector fields in the $0$-calculus of Mazzeo-Melrose in \cite{mazzeo1987meromorphic}, i.e,
$$\mathcal{V}_0(M) = \{X \in \mathfrak{X}(M) : X{\big|}_{\partial M}=0\}=x\mathfrak{X}(M),$$
where $x$ is a boundary-defining function.
\end{example}

\begin{proposition}
If $\mathcal{V}$ is a structural Lie algebra of vector fields, then $\mathcal{V}$ is a finitely generated projective $C^{\infty}(M)$-module. So there exists a vector bundle $E$ such that $\mathcal{V}=\Gamma(M,E)$ and a natural vector bundle map $\varrho_\mathcal{V} : E \to TM$ such that the induced map $\varrho_\Gamma : \mathcal{V} \to \mathfrak{X}(M)$ identifies with the inclusion map.
\end{proposition}
See Proposition 2.12 in \cite{ammann2004geometry} for a proof.\\
Now, let us recall the definition of a Lie algebroid. In a general sense, a Lie algebroid can be viewed as the multi-object version of a Lie algebra. 
\begin{definition}[Lie algebroid]
A Lie algebroid $\mathcal{A}$ over a manifold $M$ is a vector bundle $\mathcal{A}$ over $M$, together with a Lie algebra structure on the space $\Gamma(\mathcal{A})$ of smooth sections of $\mathcal{A}$ and a bundle map $\varrho:\mathcal{A}\to TM$, extended to a map $\varrho_\Gamma:\Gamma(\mathcal{A})\to \Gamma(TM)$ between sections of these bundles, such that the right Leibniz rule is also satisfied
$$[X, fY] =f[X,Y]+(\varrho_\Gamma(X)f )Y, \quad \forall X, Y \in \Gamma(A), \, f \in C^{\infty}(M).$$
The map $\varrho_\Gamma$ is called the anchor of $\mathcal{A}$.
\end{definition}

\begin{remark}
By the antisymmetry of the bracket, the left Leibniz rule is also satisfied:
$$[fX, Y] =f[X,Y]-(\varrho_\Gamma(X)f )Y, \quad \forall X, Y \in \Gamma(A), \, f \in C^{\infty}(M).$$
\end{remark}

\begin{remark}
By direct calculation as Proposition 1.21 in \cite{ammar2021polyhomogeneite} one can check that the anchor map is a morphism of Lie algebras. In other words,
$$\varrho_\Gamma([X, Y]) = [\varrho_\Gamma(X), \varrho_\Gamma(Y)], \quad \forall X,Y \in \Gamma(A),$$
where on the left, we have the Lie algebroid bracket, and on the right, we have the Lie algebra bracket of vector fields.
\end{remark}

\begin{example}
Examples of Lie algebroid
\begin{enumerate}
\item All Lie algebras are Lie algebroids. In fact, a Lie algebroid over a one-point set, with the zero anchor, is a Lie algebra.
\item Any bundle of Lie algebras is a Lie algebroid with zero anchor and Lie bracket defined pointwise.
\item The tangent bundle $TM$ of a manifold $M$, with as bracket the Lie bracket of vector fields and with as anchor the identity of $TM$, is a Lie algebroid over $M$ which is called tangent Lie algebroid.
\item Given the action of a Lie algebra $\mathfrak{g}$ on a manifold $M$ that is, a homomorphism of Lie algebras $\rho:\mathfrak{g} \to \mathfrak{X}(M)$, the action algebroid is the trivial vector bundle $\mathfrak{g} \times M \to M$, with the anchor given by the Lie algebra action and brackets uniquely determined by the bracket of $\mathfrak{g}$ on constant sections $M \to \mathfrak{g}$ and by the Leibniz identity.
\end{enumerate}
\end{example}

\begin{definition}[Lie structure at infinity]\label{Lie structure at infinity}
A Lie structure at infinity on a smooth manifold $M^\circ$ is a pair $(M, \mathcal{V})$, where $M$ is a compact manifold possibly with corners, $\inter (M)=M^{\circ}$ and $\mathcal{V}$ is a structural Lie algebra of vector fields on $M$ such that its anchor $\varrho_\mathcal{V}:{}^\mathcal{V} TM\to TM$ is an isomorphism on $M^{\circ}$, i.e,$ {}^\mathcal{V} TM{\big|}_{TM^{\circ}}\cong TM^{\circ}$.
\end{definition}
\begin{remark}
If $M^{\circ}$ is compact without boundary, then it follows from the Definition \ref{Lie structure at infinity} that $M=M^{\circ}$ and ${}^\mathcal{V} TM=TM$, so a Lie structure at infinity on $M^{\circ}$ gives no additional information
on $M^{\circ}$. The interesting cases are thus the ones when $M^{\circ}$ is noncompact.
\end{remark}

\begin{definition}[Riemannian manifold with a Lie structure at infinity]
Let $(M,\mathcal{V})$ be a Lie structure at infinity for a manifold $M$. Let $\varrho_{\mathcal{V}}:{}^\mathcal{V} TM\to TM$ be the associated anchor and $g$ a Riemannian metric on ${}^\mathcal{V} TM$, that is, a smooth positive definite symmetric 2-tensor $g$ on ${}^\mathcal{V} TM$.
In this case, $(M^{\circ}, (\varrho_{\mathcal{V}}^{-1})^{*}(g{\big|}_{M^{\circ}}))$ is called a Riemannian manifold with a Lie structure at infinity.
\end{definition}
Riemannian manifold with a Lie structure at infinity have some nice geometric property, for instance the following proposition shows that the volume of any noncompact Riemannian
manifold with a Lie structure at infinity is infinite.
\begin{proposition}
Let $M^{\circ}$ be a Riemannian manifold with Lie structure $(M, \mathcal{V},g)$ at
infinity. Let $f\geq 0$ be a smooth function on $M$. If $\displaystyle\int_{M^{\circ}}f\dvol_g<\infty$, then $f$ vanishes on each boundary hyperface of $M$. In particular, the volume of any noncompact Riemannian
manifold with a Lie structure at infinity is infinite.
\end{proposition}
See Proposition 4.1.in \cite{ammann2004geometry} for a proof.
\begin{proposition}
Let $M^{\circ}$ be a Riemannian manifold with a Lie structure $(M, \mathcal{V},g)$ at infinity. Then $M^{\circ}$ is complete in the induced metric $g$.
\end{proposition}
See Corollary 4.9. in \cite{ammann2004geometry} for a proof.
\begin{proposition}
Let $M^{\circ}$ be a connected Riemannian manifold with a Lie structure $(M, \mathcal{V},g)$ at infinity. Then $(M^{\circ},g)$ is of bounded geometry.
\end{proposition}
See Corollary 4.3 in \cite{ammann2004geometry} and Theorem 5.2. in \cite{bui2020injectivity} for a proof.
\begin{definition}[Melrose $b$-metric]
Let $M$ be a compact Riemannian manifold with boundary, equipped with a Lie structure at infinity $(M, \mathcal{V}_b, g_b)$. The Riemannian metric $g_b$ is referred to as the $b$-metric. 
\end{definition}
\begin{example}[Manifold with asymptotically cylindrical end]
A manifold with cylindrical ends is a Riemannian manifold $(M, g)$ for which there exists a compact subset $K$ (topologic part) such that outside $K$, $M$ resembles a cylinder with the product metric. This can be expressed through the identification:
$$M\setminus K\cong N\times (0,\infty)_r,$$
where $N$ is a closed manifold with $\dim N=\dim M-1$ and a cylindrical metric on $N\times (0,\infty)_r$, i.e,
$$g{\big|}_{M\setminus K}=g_{\text{cyl}}=g_N+dr^2,$$
where $g_N$ is a metric on $N$ and $r\in (0,\infty)$ is a coordinate for $(0,\infty)$. 
Let $(M, \mathcal{V}_b)$ be a compact Riemannian manifold with boundary $M$, equipped with a Lie structure at infinity. By using a tubular neighborhood of $N=\partial M$ in $M$ to make
$$M^{\circ}\setminus K\cong \partial M\times (0,\infty)_r.$$
So, we have the cylindrical end with the cylindrical metric
$$g{\big|}_{M\setminus K}=g_{\partial M}+dr^2.$$
By attaching $\partial M$ at infinity, we obtain a compactification $M := M^{\circ} \cup \partial M$, which is a compact manifold with boundary.
Introducing the change of variable $x=e^{-r}$ where $r$ is the coordinate for $(0,\infty)_r$, yields a defining function for $\partial M$, so we can write this cylindrical metric as
$$g{\big|}_{M\setminus K}=g_{\partial M}+\dfrac{dx^2}{x^2}\in \Gamma(S^2({}^b T^*M)),$$
which is compatible with the Lie structure at infinity $\mathcal{V}_b$, i.e, it is a $b-$metric. \\
More generally, we say $(M,g)$ has asymptotically cylindrical end if $g\to g_{\partial M}+\dfrac{dx^2}{x^2}$ when $x\to 0$ in the following sense: there exist $\gamma>0$ such that
$$g-(g_{\partial M}+\dfrac{dx^2}{x^2})\in x^\gamma C^\infty_b(M, S^2({}^b T^*M))=x^\gamma C^\infty_b(M)\otimes _{C^\infty(M)}\Gamma(S^2({}^b T^*M)),$$
where
$$C^\infty_b(M^{\circ})=\{f\in C^\infty(M^{\circ}):\forall k\in \mathbb{N}_0,\{V_1,\ldots, V_k\}\subset \mathcal{V}_b(M) ,\displaystyle\sup_{M^{\circ}}|V_1 \ldots V_kf|<\infty\},$$
and 
$$x^\gamma C^\infty_b(M^{\circ})=\{f\in C^\infty (M^{\circ}):\dfrac{f}{x^\gamma}\in C^\infty_b(M^{\circ})\}.$$
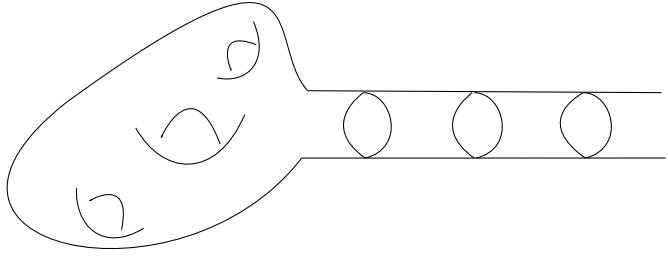
\begin{figure}
\centering
\begin{tikzpicture}[x=0.75pt,y=0.75pt,yscale=-1,xscale=1]

\draw    (128.2,64.2) .. controls (258.69,-31.03) and (227.2,33.8) .. (249.71,58.43) ;
\draw    (128.2,64.2) .. controls (34.2,138.2) and (185.2,171.8) .. (246.71,92.43) ;
\draw    (249.71,58.43) -- (428.2,59.4) ;
\draw    (246.71,92.43) -- (430.2,92.4) ;
\draw    (277.71,59.43) .. controls (292.71,59.43) and (299.71,87.43) .. (278.71,92.43) ;
\draw    (333.71,59.43) .. controls (348.71,59.43) and (355.71,87.43) .. (334.71,92.43) ;
\draw [color={rgb, 255:red, 0; green, 0; blue, 0 }  ,draw opacity=0.41 ]   (277.71,59.43) .. controls (277.71,59.43) and (254.71,75.43) .. (278.71,92.43) ;
\draw [color={rgb, 255:red, 0; green, 0; blue, 0 }  ,draw opacity=0.41 ]   (388.71,59.43) .. controls (382.2,63) and (365.71,75.43) .. (389.71,92.43) ;
\draw [color={rgb, 255:red, 0; green, 0; blue, 0 }  ,draw opacity=0.41 ]   (332.71,59.43) .. controls (332.71,59.43) and (309.71,75.43) .. (333.71,92.43) ;
\draw    (163,77.2) .. controls (176.49,101.2) and (203.47,104.19) .. (218,70.45) ;
\draw    (176.49,80.85) .. controls (172.34,89.85) and (189.98,43.35) .. (205.55,85.35) ;
\draw    (133.13,107.5) .. controls (132.2,126.84) and (146.54,140.36) .. (167,127.86) ;
\draw    (140.84,113.42) .. controls (135.23,116.63) and (162.05,98.37) .. (155.83,128.69) ;
\draw    (204,51.97) .. controls (219.28,55.17) and (230.94,43.71) .. (222.35,23.5) ;
\draw    (210.73,47.09) .. controls (212.92,52.61) and (200.17,25.88) .. (223.82,35.26) ;
\draw    (388.71,59.43) .. controls (403.71,59.43) and (410.71,87.43) .. (389.71,92.43) ;
\end{tikzpicture}
\caption{\small A manifold with a cylindrical end}
\end{figure}

\end{example}

\begin{definition}[Melrose $\sca$-metric]
A Riemannian metric compatible with a Lie structure at infinity $(M, \mathcal{V}_{\sca})$, where $M$ is a compact manifold with boundary, is called scattering metric ($\sca$-metric for short). 
\end{definition}
\begin{example}
Let $M$ be a compact manifold with boundary. A metric of the form 
\[
 g_{\sca}= \dfrac{g_{\partial M}}{x^2} + \dfrac{dx^2}{x^4},
\]
close to the boundary $\partial M$, is an example of $\sca$-metric, where $x$ is a defining function for the boundary, and $g_{\partial M}$ is the Riemannian metric $g$ restricted to the boundary. Notice that a $\sca$-metric is always of the form $g_{\sca}=\dfrac{g_b}{x^2}$ for some $b$-metric.
\end{example}

\begin{example}
A simple example in the Euclidean case is the radial compactification of $\mathbb{R}^n$ with the boundary being the sphere $\mathbb{S}^{n-1}$. This compactification is given by the stereographic projection $\SP$ defined by
\[\SP: \mathbb{R}^n \to \mathbb{S}^{n}_+ := \left\{z = (z_0, ..., z_n) \in \mathbb{S}^{n} : \ z_0 \geq 0\right\},\]
\[\SP(x) = \dfrac{1}{\sqrt{1 + |z|^2}} \left(1,z_1, ..., z_n\right),\]
where $\mathbb{S}^{n}_+$ is a compact manifold with boundary. $\SP$ identifies $\mathbb{R}^n$ with the interior of the upper half-sphere $\mathbb{S}^{n}_+$. The Euclidean metric is a scattering metric on $\mathbb{R}^n$ given by 
$$g_{\mathbb{R}^n} = dr^2 + r^2g_{\mathbb{S}^{n-1}}  = \dfrac{dx^2}{x^4} + \dfrac{g_{\mathbb{S}^{n-1}}}{x^2},$$
where $r = |z| $ and $x=\dfrac{1}{r}$ is a defining function for the boundary i.e. $\partial \mathbb{S}^{n}_+=\mathbb{S}^{n-1} = \{x = 0\}$.
\end{example}

\begin{example}[Manifold with asymptotically conical end]
A manifold with conical ends is a Riemannian manifold $(M, g)$ for which there exists a compact subset $K$ (topologic part) such that outside $K$, $(M,g)$ is a cone. This can be expressed through the identification:
$$M\setminus K\cong N\times (R,\infty)_r, \quad \text{for some } R>0,$$
where $N$ is a closed manifold with $\dim N=\dim M-1$ and a conic metric on $N\times (R,\infty)_r$, i.e,
$$g{\big|}_{M\setminus K}=g_{\text{cone}}=r^2g_N+dr^2, r\geq R,$$
where $g_N$ is a metric on $N$ and $r\in (0,\infty)$ is a coordinate for $(0,\infty)$. 
If $M$ is a compact manifold with boundary, by using a tubular neighborhood of $N=\partial M$ in $M$, we can find an identification 
$$M^{\circ}\setminus K\cong \partial M\times (R,\infty)_r,$$
so that a conical metric $g$ on $M^{\circ}$ with $g{\big|}_{M^{\circ}\setminus K}=r^2g_{\partial M}+dr^2$ is a $\sca$-metric. Indeed, attaching this cone to $\partial M$ at infinity, we obtain a compactification $M := M^{\circ} \cup \partial M$, which is a compact manifold with boundary.
Introducing the change of variable $x=\dfrac{1}{r}$ where $r$ is the coordinate for $(0,\infty)_r$, yields a defining function for $\partial M$, so we can write this conical metric as
$$g{\big|}_{M\setminus K}=\dfrac{g_{\partial M}}{x^2}+\dfrac{dx^2}{x^4}\in \Gamma(S^2({}^{\sca} T^*M)),$$
which is compatible with the Lie structure at infinity $\mathcal{V}_{\sca}$, i.e, it is a $\sca-$metric.

\begin{figure}
\centering
\begin{tikzpicture}[x=0.75pt,y=0.75pt,yscale=-1,xscale=1]
\draw    (128.2,110.6) .. controls (258.69,15.37) and (235.2,79.4) .. (249.71,104.83) ;
\draw    (128.2,110.6) .. controls (34.2,184.6) and (185.2,218.2) .. (246.71,138.83) ;
\draw    (163,123.6) .. controls (176.49,147.6) and (203.47,150.59) .. (218,116.85) ;
\draw    (176.49,127.25) .. controls (172.34,136.25) and (189.98,89.75) .. (205.55,131.75) ;
\draw    (133.13,153.9) .. controls (132.2,173.24) and (146.54,186.76) .. (167,174.26) ;
\draw    (140.84,159.82) .. controls (135.23,163.03) and (162.05,144.77) .. (155.83,175.09) ;
\draw    (204,98.37) .. controls (219.28,101.57) and (230.94,90.11) .. (222.35,69.9) ;
\draw    (210.73,93.49) .. controls (212.92,99.01) and (200.17,72.28) .. (223.82,81.66) ;
\draw    (249.71,104.83) .. controls (254.2,109.2) and (277.2,145.2) .. (378.2,29.2) ;
\draw    (246.71,138.83) .. controls (269.2,106.2) and (355.2,164) .. (378.2,195.2) ;
\draw    (294,135.37) .. controls (307.2,126.2) and (302.2,111.2) .. (294.2,105.2) ;
\draw    (350.2,167.2) .. controls (382.2,146.2) and (390.2,89.2) .. (351.2,59.2) ;
\draw [color={rgb, 255:red, 0; green, 0; blue, 0 }  ,draw opacity=0.5 ]   (350.2,167.2) .. controls (330.2,140.2) and (320.2,105.2) .. (351.2,59.2) ;
\draw [color={rgb, 255:red, 0; green, 0; blue, 0 }  ,draw opacity=0.5 ]   (294,135.37) .. controls (284.2,126.2) and (281.2,115.2) .. (294.2,105.2) ;
\end{tikzpicture}
\caption{\small A manifold with a conical end}
\end{figure}
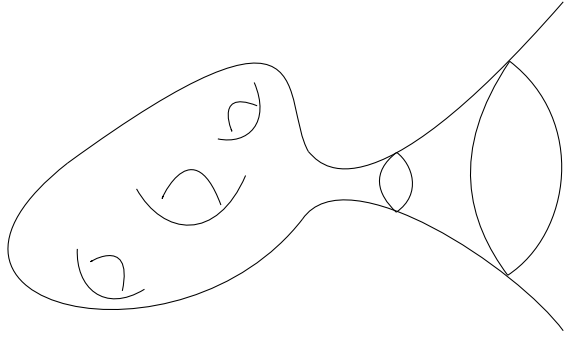

\end{example}
\begin{definition}[Asymptotically Conical Metric]
Let $(L, g)$ be a compact Riemannian manifold. On $C=(0, \infty) \times L$, consider the conic metric
$$g_{\text{cone}} = dr^2 + r^2 g,$$
where $r$ is the coordinate on $(0, \infty)$. A Riemannian manifold $(M, g)$ is called asymptotically conical (AC for short), asymptotic to $g_{\text{cone}}$, if there exists a diffeomorphism $\pi: M \setminus K \rightarrow  (R, \infty) \times L$, for some $R > 0$ and $K \subset M$ compact, there is a positive constants $c$ and $\mu$ such that for any $k\geq 0$,
$$|\nabla^k (\pi_{*}(g) - g_{\text{cone}})|_{g_{\text{cone}}}\leq \dfrac{c}{r^{\mu+k}}.$$
Here, $\nabla$ denotes the Levi-Civita connection of $g_{\text{cone}}$.
\end{definition}
Now, we define asymptotically locally Euclidean metrics, which are important examples of asymptotically conical metrics. In fact when $g_{\text{cone}}$ is a quotient of the euclidean flat space by a finite subgroup of the orthogonal matrices which acts freely on the unit sphere, the corresponding AC metrics are called  asymptotically locally euclidean or ALE for short.
\begin{definition}[Asymptotically Locally Euclidean metric]
Let $\Gamma$ be a finite subgroup of $U(n)$ acting freely on $\mathbb{C}^n \setminus \{0\}$, so $\mathbb{C}^n / \Gamma$ has an isolated quotient singularity at $0$, and the Euclidean metric is $\Gamma$-invariant. Thus, $(\mathbb{C}^n / \Gamma, g_{\Euc})$ is a Riemannian cone. Let $M$ be a non-compact complex manifold with end asymptotic to the cone $\mathbb{C}^n / \Gamma$ at infinity (e.g. a resolution of $\mathbb{C}^n / \Gamma$), i.e., there is a compact subset $K \subset M$ and a map $\pi: M \setminus K \to \mathbb{C}^n / \Gamma$ that is a diffeomorphism between $M \setminus K$ and $\{z\in \mathbb{C}^n / \Gamma: d_{\Euc}(z,0)>R\}$ for a fixed positive constant $R$. A Riemannian (Kähler) metric $g$ on $M$ is called asymptotically locally Euclidean (ALE-for short) if $\pi_*(g)$ is asymptotic to $g_{\Euc}$ at infinity, i.e., there is a positive constant $c$ such that for any $k\geq 0$,
$$|\nabla^k(\pi_*(g) - g_{\Euc})| \leq \dfrac{c}{r^{n+k}},$$
where $\nabla$ is the Levi-Civita connection of $g_{\Euc}$ on $\mathbb{C}^n / \Gamma$.
\end{definition}

\begin{example}[Burns and Simanca ALE scalar-flat Kähler metrics on $\mathcal{O}_{\mathbb{	CP}^{m-1}}(-1)$]\label{r=1}
In 1991 Burns (case $m=2$) and Simanca \cite{simanca1991kahler}  (case $m>1$) constructed a cscK metric on $\mathcal{O}_{\mathbb{CP}^{m-1}}(-1)$. They showed that the \kahler potential of this scalar flat \kahler metric on $\mathcal{O}_{\mathbb{CP}^{m-1}}(-1)$ is radially symmetric and of the form 
$$H_{\text{BS}}=\|Z\|^2+\gamma(\|Z\|)\log \|Z\|^2+\|Z\|^{4-2m}+\psi(\|Z\|^2),$$
where $\gamma:\mathbb{R}\to\mathbb{R}$ is the cut-off function such that $\gamma(t)=1$ for $t<1$, $\gamma(t)=0$ for $t>2$ and 
$$|\nabla^k \psi (t)| \leq \dfrac{c}{t^{m+k-2}},$$
for all $k\geq 0$. Here $\nabla$ is the Levi-Civita connection of $g_{\FS}$ on $\mathcal{O}_{\mathbb{CP}^{m-1}}(-1)$. 
\end{example}
More generally, for any natural number $r$, as discussed in section 2 of \cite{apostolov2017ale}, the line bundle $\mathcal{O}_{\mathbb{CP}^{m-1}}(-r)$ admits an ALE scalar-flat Kähler metric as follows.
\begin{example}[ALE scalar-flat Kähler metric on $\mathcal{O}_{\mathbb{CP}^{m-1}}(-r)$]\label{genrealr}
The Burns-Simanca metric is generalized by Eguchi-Hanson \cite{eguchi1979self} ($m = 2$, $r = 2$), LeBrun \cite{lebrun1988counter} ($m = 2$, $r > 2$), Pedersen-Poon  \cite{pedersen1991hamiltonian} and Rollin-Singer \cite{rollin2009construction} ($m > 2$, $r > 2$)  . In summary, the Burns-Eguchi-Hanson-LeBrun-Pedersen-Poon-Simanca metric on $\mathcal{O}_{\mathbb{CP}^{m-1}}(-r)$ has Kähler potential:
$$H=\dfrac{1}{2}\|Z\|^2+A\|Z\|^{4-2m}+O(\|Z\|^{3-2m}),$$
when $m\geq 3$ and 
$$H=\dfrac{1}{2}\|Z\|^2+A\log \|Z\|+O(\|Z\|^{-2}),$$
when $m=2$, where $A$ is some constant. Recall that $T=O(t^s)$ if $|T|\leq ct^s$ for some $c>0$.
\end{example}

There is a generalization of ALE-metrics to quasi-asymptotically locally Euclidean metrics (QALE-metrics for short) by D. Joyce in \cite{joyce2001asymptotically} when the action of $\Gamma$ is not necessarily free on $\mathbb{C}^n \setminus \{0\}$. There is as well a generalization of AC-metrics and QALE-metrics to quasi-asymptotically conical metrics (QAC-metrics for short) by Degeratu and Mazzeo \cite{degeratu2018fredholm}, which we will discuss and use in Chapter 4 for our main construction.

As a summary, we have these relations between these special metrics
\[\begin{array}{ccc}
\text{AE} \subset  \quad\text{ALE} & \subset & \text{AC}\quad  \subset \sca \\
\quad\quad\quad\cap & & \cap  \quad\quad\quad\\
\quad \quad\quad \text{QALE} & \subset & \text{QAC} \quad\quad\quad
\end{array}\]
We finish this chapter by defining the edge-metric in Mazzeo's sense.
\begin{definition}[Mazzeo edge metric]
A Riemannian metric compatible with a Lie structure at infinity $(M, \mathcal{V}_{e})$, where $M$ is a compact manifold with fibered boundary is called and edge metric. An edge metrics $g_{e}$ close to the boundary $\partial M$ is an element of $S^2({}^e T^*M)$ which locally generated by
$$\{\dfrac{dx^2}{x^2},\dfrac{dy_i^2}{x^2},\dfrac{dx\otimes dy_i}{x^2},dz_j^2,\dfrac{dx\otimes dz_j }{x},\dfrac{dy_i\otimes dz_j}{x}\},$$
in terms of the local coordinates of example \ref{Lie algebra of edge vector fields}.
\end{definition}
\begin{example}[$0-$metrics]
An interesting class of a edge metrics is the class of $0$-metrics, i.e., metrics corresponding to the Lie structure at infinity $(M, \mathcal{V}_{0})$ as discussed in Example \ref{0-calculus}. A $0$-metric $g_{0}$ close to the boundary $\partial M$ is of the form
\[ g_0 = \frac{dx^2}{x^2} + \frac{dy_i^2}{x^2}. \]
In fact, if $(M,g)$ is a compact Riemann manifold with boundary, and $x$ is a defining function for the boundary, then the metric in the interior of $M$,
\[ g_0 = \frac{g}{x^2}, \]
is complete and is an example of $0-$metric. In particular, the hyperbolic space is of this type. The sectional curvature of $g_0$ approaches $-|dx|_g^2$ at the boundary, so $g_0$ has negative curvature outside a compact set. For more information, see lemma 2.5 in \cite{mazzeo1987meromorphic}.
\end{example}

\section{Constant Scalar Curvature \kahler metrics}
\label{sec: Constant Scalar Curvature \kahler metrics}

In this section we focus on constant scalar curvature \kahler (cscK) metrics. A special case of a constant scalar curvature Kähler metric is a Kähler-Einstein (KE) metric, which has been the main focus of Kähler geometry since the inception of the celebrated Calabi conjecture on the existence of canonical Kähler metrics in the 1950s: In every Kähler Class of every compact Kähler manifolds, there must exists one best, canonical Kähler metrics.\\
In fact, Calabi proposed the following conjectures for an compact Kähler manifold $(X,\omega_X)$:\\
\textbf{Conjecture 1:} If $\Aut(X) = 1$, then there exists a unique cscK metric on $X$ in $[\omega_X]$.\\
\textbf{Conjecture 2:} There exists an extremal Kähler metric on $X$ in $[\omega_X]$, unique up to $\Aut(X)$.

Calabi's vision, now six decades later, has been the inspiration for fundamental work in Kähler geometry up to the present day. From Yau's celebrated theorem \cite{yau1978ricci}, based on Calabi's $C^3$ estimate for the Monge-Ampère equation in 1958 \cite{calabi1958improper}, for which he received the Fields Medal in 1976, to the conjecture of Yau-Tian-Donaldson in the Kähler-Einstein Fano case that was finally solved in 2012 by Chen, Donaldson, and Sun \cite{chen2015kahler1, chen2015kahler2, chen2015kahler3} and Tian \cite{MR3352459}.\\
We begin this section with the scalar curvature function and the definition of extremal metrics. Then, we briefly look at Kähler-Einstein metrics and Conjecture 1. Following that, we study cscK metrics and discuss classic results by Matsushima-Lichnerowicz and Arezzo-Pacard. Finally, we wrap up this section by constructing new examples of cscK orbifolds with singularities of type $\mathcal{I}$.
\begin{lemma}
Let $(M,g)$ be a Kähler manifold and $D$ denote its Levi-Civita connection. For a real 1-form $\alpha$ if we denote by $D^{-}\alpha$ the J-anti-invariant part of the covariant derivative $D\alpha$, then we have
$$D^{-}\alpha=-\dfrac{1}{2}g(J(\mathcal{L}_{\alpha^\sharp}J)\bullet,\bullet)=-\dfrac{1}{2}\omega(J(\mathcal{L}_{\alpha^\sharp}J)\bullet,\bullet).$$
\end{lemma}
See lemma 1.22.2. \cite{Gauduchon} for a proof.

\begin{definition}[Lichnerowicz operator]\label{Lichner}
Let $(M,g)$ be a Kähler manifold and $D$ the Levi-Civita connection. If we set $\mathcal{D}=D^{-}d$ and denote by $\mathcal{D}^*=(D^{-}d)^*$ its formal adjoint, then the fourth-order operator $\mathcal{D}^*\mathcal{D}:C^{\infty}(M,\mathbb{C})\to C^{\infty}(M,\mathbb{C})$ is called the Lichnerowicz operator. The Lichnerowicz operator $\mathcal{D}^*\mathcal{D}$ is a formally self-adjoint, semipositive
differential operator acting on (real) functions. For instance for complex-valued functions $\phi$, $\psi$ defined on compact Kähler manifold $(M,g)$ we have 
$$\displaystyle\int_M(\mathcal{D}^*\mathcal{D}\phi)\overline{\psi}\omega^n=\displaystyle \int_M \phi\overline{\mathcal{D}^*\mathcal{D}\psi}\omega^n.$$
\end{definition}

\begin{lemma}[Lichnerowicz operator]\label{Lichnerowicz}
Suppose that $(M,g)$ is a Kähler manifold and $u$ is a smooth complex-valued function defined on $M$. Then
\begin{align*}
\mathcal{D}^*\mathcal{D}u&=\dfrac{1}{2}\Delta_g^2u+\Rc_g.\nabla^2 u+\dfrac{1}{2}\nabla S(\omega).\nabla u\\
&=\dfrac{1}{2}\Delta_g^2u+R^{\bar{j}i}\nabla_{i}\nabla_{\bar{j}} u+\dfrac{1}{2}g^{i\bar{j}}\nabla_{i} S(\omega).\nabla_{\bar{j}} u.
\end{align*}
\end{lemma}
See lemma 1.22.5. \cite{Gauduchon} for a proof.

\begin{lemma}\label{linears}\label{nonlinears}
Let $M$ be a Kähler manifold with Kähler metric $g$ and corresponding Kähler form \(\omega\). Assume that \(\omega +\sqrt{-1}\dd u\) is a small perturbation of $\omega$ by a function $u\in C^4(M)$ with $\|u\|_{C^4(M)}<c$, for a sufficiently small $c>0$. Then we can linearize the scalar curvature operator in the following way:
$$S(\omega+\sqrt{-1}\dd u)=\displaystyle \sum_{k=0}^{+\infty}\dfrac{d^k}{dt^k}|_{t=0}S(\omega +t\sqrt{-1}\dd u)=S(\omega)+L_{\omega}(u)+Q_{\omega}(\nabla^2 u),$$
where $L_{\omega}(u)=\dfrac{d}{dt}|_{t=0}S(\omega +t\sqrt{-1}\dd u)$ is the linear part and 
 $$Q_{\omega}(\nabla^2 u)=\displaystyle \sum_{k=2}^{+\infty}\dfrac{d^k}{dt^k}|_{t=0}S(\omega +t\sqrt{-1}\dd u),$$
is a second-order nonlinear differential operator that collects all the nonlinear terms. Moreover, the linearization of the scalar curvature operator $L_{\omega}(u)$ can be expressed as:
$$L_{\omega}(u)=-(\dfrac{1}{2}\Delta^2_gu+\braket{\Rc_g,\sqrt{-1}\dd u}_g)=-\dfrac{1}{2}\Delta^2_gu-R^{\bar{j}i}\partial_i\partial_{\bar{j}}u=\dfrac{1}{2}\nabla S(\omega).\nabla u-\mathcal{D}^*\mathcal{D}u.$$ 
Also, the nonlinear part $Q_{\omega}$ could decomposes with finite sums as follows:
\begin{align*}
Q_{\omega}(\nabla^2u)&=\displaystyle\sum_{q}B_{q,4,2}(\nabla^4u,\nabla^2u)C_{q,4,2}(\nabla^2u)\\
&+\displaystyle\sum_{q}B_{q,3,3}(\nabla^3u,\nabla^3u)C_{q,3,3}(\nabla^2u)\\
&+|z|\displaystyle\sum_{q}B_{q,3,2}(\nabla^3u,\nabla^2u)C_{q,3,2}(\nabla^2u)\\
&+\displaystyle\sum_{q}B_{q,2,2}(\nabla^2u,\nabla^2u)C_{q,2,2}(\nabla^2u),
\end{align*}
where $B^i$s are bilinear forms and $C^i$s are smooth functions.
\end{lemma}
See lemma 2.158 \cite{besse2007einstein}, Equation (31) in \cite{arezzo2006blowing}, lemma 1.2 and lemma 1.3 in \cite{lena2013desingularization} for a proof.
\begin{lemma}\label{curvature of a conformally}
Let $M$ be a Kähler manifold of real dimension $n$ with Kähler metric $g$ and corresponding Kähler form \(\omega\). The scalar curvature of a conformally changed metric $\omega'=e^{2f}\omega$ can be computed by:
$$S(\omega')=e^{-2f}(S(\omega)+2(n-1)\Delta_\omega f-(n-1)(n-2)\|\nabla f\|^2_\omega).$$
\end{lemma}
See Section 1J in \cite{besse2007einstein} for a proof.
Extremal metrics were defined by Calabi \cite{calabi1982extremal} in 1982 as follows:
\begin{definition}[Extremal Metrics]
Suppose that $M$ is a compact Kähler manifold. An extremal Kähler metric on $M$ in the class $\Omega\in H^2_{\DR}(M,\mathbb{R})$ is a critical point
of the functional
\[\text{Cal}(\omega) = \displaystyle\int_M S(\omega)^2\omega^n,\]
for $\omega\in \Omega$, where $S(\omega)$ is the scalar curvature of $\omega$. This functional is called the
Calabi energy functional.
\end{definition}
\begin{theorem}
A Kähler metric $\omega$ on compact Kähler manifolds $M$ is extremal if and only if $S(\omega)$ is a Killing potential, i.e, one of the following equivalent conditions holds:
\begin{enumerate}
\item  $\nabla^{1,0}S(\omega)$ is a holomorphic vector field.
\item  $\mathcal{D}^*\mathcal{D} S(\omega)=0$. (Here, $\mathcal{D}^*\mathcal{D}$ is the Lichnerowicz operator defined in Definition \ref{Lichner} on page \pageref{Lichner}.)
\end{enumerate}
\end{theorem}
See lemma 1.23.2 in \cite{Gauduchon} or Theorem 4.2. in \cite{szekelyhidi2014introduction} for a proof.
\begin{example}
The most important examples of extremal metrics are cscK metrics. In particular Kähler-Einstein metrics have constant scalar curvature, so they are examples of extremal metrics, i.e,
\[\begin{array}{ccc}
\text{KE} \subset  \quad\text{cscK} & \subset & \text{Extremal.}
\end{array}\]
\end{example}
\begin{remark}
Suppose that $M$ is a compact Kähler manifold.
\begin{enumerate}
\item If $\mathfrak{h}(M)=0$, i.e, $M$ admits no non-trivial holomorphic vector fields, then every extremal Kähler metric must have constant scalar curvature.
\item If the Kähler class $\Omega$ is proportional to $c_1(M)$, then any constant scalar curvature metric in $\Omega$ is Kähler-Einstein.
\end{enumerate}
\end{remark}
See lemma 2.2.3 in \cite{futaki2006kahler} for a proof.\\
Now we discuss the \textbf{Conjecture 1}. The existence of Kähler-Einstein metrics for compact Kähler manifolds depends on the sign of the first Chern class of the Kähler manifold.
\begin{theorem}[Yau]
Let $M$ be a compact Kähler manifold with $c_1(M) = 0$. Then, every Kähler class contains a unique Ricci flat metric. These types of manifolds are called Calabi-Yau. Calabi-Yau manifolds are complex manifolds that generalize K3 surfaces to higher dimensions.
\end{theorem}
The case when the first Chern class is negative is proved independently in 1978 by Thierry Aubin \cite{aubin1976equations} and Shing-Tung Yau \cite{yau1978ricci} as follows:
\begin{theorem}[Aubin-Yau]
Let $M$ be a compact Kähler manifold with $c_1(M)<0$. Then, there is a unique Kähler metric $\omega\in -2\pi c_1(M)$ such that $\Rc(\omega)=-\omega$.
\end{theorem}
When the first Chern class is positive, existence of Kähler-Einstein metrics remained a well-known open problem for many years. In this case, there are a non-trivial obstructions to existence. In 2012, Xiuxiong Chen, Simon Donaldson, and Song Sun \cite{chen2015kahler1},\cite{ chen2015kahler2}, \cite{chen2015kahler3} as well as Tian \cite{MR3352459} proved that in this case existence is equivalent to an algebro-geometric property called K-stability. 
\\
Now we state Matsushima - Lichnerowicz theorem. This classical theorem gives us obstructions to the existence of cscK metric based on the structure of the Lie algebra of holomorphic vector fields.
\begin{theorem}[Matsushima-Lichnerowicz]
Let $(M, J, g)$ be a cscK manifold. Then, the Lie algebra $\mathfrak{h}(M)$ of holomorphic vector fields decomposes as a direct sum: 
$$\mathfrak{h}(M) = \mathfrak{h}_0(M) \oplus \mathfrak{a}(M),$$
where $\mathfrak{a}(M) \subset \mathfrak{h}(M)$ is the abelian subalgebra of parallel holomorphic vector fields and $\mathfrak{h}_0(M)$ is the subalgebra of holomorphic vector fields with zeros.  Furthermore, $\mathfrak{h}_0(M)$ is the complexification of the killing fields with zeros, i.e,
$$\mathfrak{h}_0(M) = (\mathfrak{k}(M, g) / \mathfrak{a}(M))\otimes_{\mathbb{R}} \mathbb{C},$$
where $\mathfrak{k}(M, g)$ denotes the Lie algebra of real Killing vector fields on $(M, g)$. In particular, $\mathfrak{h}(M)$ is a reductive Lie algebra, i.e, it is the direct sum of an abelian and a semisimple Lie algebra. 
\end{theorem}
See \cite{lichnerowicz1957transformations} and \cite{matsushima1957structure} for a proof. 
\begin{corollary}
Let $(M, J, g)$ be a cscK manifold. Then the identity component of $\Iso(M,g,J)$ is the maximal compact subgroup of the identity component $\Aut(M,J)$.
\end{corollary}
\begin{corollary}
The theorem of Lichnerowicz and Matsushima implies that a compact \kahler manifold $(M, J)$ whose identity component $\Aut_0(M,J)$ of the automorphism group is not reductive does not admit any cscK metric. 
\end{corollary}
\begin{example}
Let $n>1$, then the projective space $\mathbb{CP}^n$ blown-up at one or two points does not admit any cscK metric. See \cite{besse2007einstein} page 331 for more details.
\end{example}
\begin{remark}
There is a general version of the Matsushima-Lichnerowicz theorem for extremal metrics proved by Calabi (Theorem 2.3.6 in \cite{futaki2006kahler}). In particular, a compact complex manifold $(M,J)$ for which the connected group of automorphisms is non-trivial but has no connected compact subgroup apart from $\{\Id\}$ cannot have any extremal Kähler metric. Examples of Kähler compact complex surfaces satisfying these hypotheses, hence admitting no extremal Kähler metric, were first given by M. Levine \cite{levine1985remark}. As a consequence, the answer of \textbf{Conjecture 2} is negative in general. A Kähler manifold is called a \textbf{Calabi dream manifold} if every Kähler class on it admits an extremal metric. All compact Riemann surfaces, complex projective spaces $\mathbb{CP}^n$, Hirzebruch surfaces $\mathbb{F}_k\cong \mathbb{P}(\mathcal{O}\oplus \mathcal{O}(k))$, and all compact Calabi-Yau manifolds \cite{yau1978ricci} are Calabi dream manifolds.
\end{remark}
\begin{proposition}[Calabi]
Let $g$ be an extremal Kähler metric on a compact complex manifold $(M, J)$. Let $\Iso_{0}(M,g)$ denote the identity component of the isometry group of $(M, g)$, and let $\Aut_{0}(M,J)$ denote the identity component of the biholomorphism group of $(M, J)$. Then $\Iso_{0}(M,g)$ is a maximal compact subgroup of $\Aut_{0}(M,J)$.
\end{proposition}
See Theorem 3.5.1 \cite{Gauduchon} for a proof.
\begin{proposition}
Let $(M, g)$ be a compact cscK manifold and let $ L_{\omega}(u)=-\mathcal{D}^*\mathcal{D}u$ be the linearization of the scalar curvature operator in lemma \ref{linears}, then
$$\dim_{\mathbb{R}}(\ker(L_{\omega}))=\dim_{\mathbb{C}}(\mathfrak{h}_0(M))+1.$$
In particular, when the identity component of the biholomorphism group $\Aut(M,J)$ is discrete, $\ker(L_{\omega})$ consists only of constant functions.
\end{proposition}
See Proposition 1 in \cite{lena2013desingularization} for a proof.

Now, we recall some notable results of Arezzo-Pacard \cite{arezzo2006blowing} and \cite{arezzo2009blowing} that provide a vast collection of cscK manifolds.
\begin{theorem}[Arezzo-Pacard]
Let $(M, \omega)$ be a constant scalar curvature compact Kähler manifold or Kähler orbifold of complex dimension $m$ with isolated singularities. Assume that there is no nonzero holomorphic vector field vanishing somewhere on $M$. Then, given finitely many smooth points $p_1, \ldots, p_n$ in $M$ and positive numbers $a_1, \ldots, a_n > 0$, there exists $\varepsilon_0 > 0$ such that the blow-up of $M$ at $p_1, \ldots, p_n$ carries constant scalar curvature Kähler forms
$$\omega_\varepsilon \in \pi^*[\omega] - \epsilon^2 \left( a_1^{\frac{1}{m-1}} [E_1] + \ldots + a_n^{\frac{1}{m-1}}  [E_n] \right),$$
where $[E_i]$ are the Poincaré duals of the $(2m-2)$-homology classes of the exceptional divisors of the blow-up at $p_i$, and $\varepsilon \in (0, \varepsilon_0)$. Moreover, as $\varepsilon$ tends to $0$, the sequence of metrics $(g_\varepsilon)_\varepsilon$  converges to $g$ (in smooth topology) on compact subsets away from the exceptional divisors.
\end{theorem}
See Theorem 1.1 in \cite{arezzo2006blowing} for a proof.

\begin{theorem}[Arezzo-Pacard]
Assume that $(M, J, g, \omega)$ is a constant scalar curvature compact Kähler manifold. There exists $n_g \geq 1$ such that for all $n \geq n_g$, there exists a nonempty open subset 
$$V_n \subset M^n_\Delta:=\{(p_1, \ldots, p_n) \in M^n \mid p_a \neq p_b \text{ for all } a \neq b\},$$
such that for all $(p_1, \ldots, p_n) \in V_n$, the blow-up of $M$ at $p_1, \ldots, p_n$ carries a family of constant scalar curvature Kähler metrics $(g_\varepsilon)_\varepsilon$ converging to $g$ (in smooth topology) on compact subsets away from the exceptional divisors, as the parameter $\varepsilon$ tends to $0$.

\end{theorem}

See Theorem 1.2 in \cite{arezzo2009blowing} for a proof.
\begin{theorem}[Arezzo-Pacard]
Assume that $(M, J, g, \omega)$ is a compact Kähler manifold of complex dimension $m$ with constant scalar curvature and that $(p_1, \ldots, p_n) \in M^n_\Delta$ are chosen so that:
 \begin{enumerate} 
 \item 
 $\xi(p_1), \ldots, \xi(p_n)$ span $\mathfrak{h}^*$, where $\mathfrak{h}$, the space of Killing vector fields with zeros.
 \item there exist $a_1, \ldots, a_n > 0$ such that $\displaystyle\sum_{i=1}^n a_i \xi(p_i) = 0 \in \mathfrak{h}^*$. 
 \end{enumerate} 
 Then, there exist $c > 0$, $\varepsilon_0 > 0$, and for all $\varepsilon \in (0, \varepsilon_0)$, there exists on the blow-up of $M$ at $p_1, \ldots, p_n$ a constant scalar curvature Kähler metric $g_\varepsilon$ which is associated to the Kähler form
$$\omega_\varepsilon \in [\omega] - \varepsilon^2 \left(a_{1,\varepsilon}^{\frac{1}{m-1}}[E_1] + \ldots + a_{n,\varepsilon}^{\frac{1}{m-1}}[E_n]\right),$$ 
where the $[E_i]$ are the Poincaré duals of the $(2m-2)$-homology classes of the exceptional divisors of the blow-up at $p_i$ and where \[ |a_{i,\varepsilon} - a_i| \leq c \varepsilon^{\frac{2}{2m+1}}. \] Finally, the sequence of metrics $(g_\varepsilon)_\varepsilon$ converges to $g$ (in smooth topology) on compact subsets, away from the exceptional divisors. 
\end{theorem}
See Theorem 1.3 in \cite{arezzo2009blowing} for a proof.
\begin{theorem}[Kronheimer-Joyce]
Let $(M, \omega)$ be a nondegenerate compact $m$-dimensional constant scalar curvature Kähler orbifold with $m = 2$ or $3$ and isolated singularities. Let $p_1, \ldots, p_n \in M$ be any set of points with a neighborhood biholomorphic to a neighborhood of the origin in $\mathbb{C}^m / \Gamma_i$, where $\Gamma_i$ is a finite subgroup of $SU(m)$. Let further $N_i$ be a Kähler crepant resolution of $\mathbb{C}^m / \Gamma_i$ (which always exists). Then there exists $\varepsilon_0 > 0$ such that, for all $\varepsilon \in (0, \varepsilon_0)$, there exists a constant scalar curvature Kähler form $\omega_\varepsilon$ on $\displaystyle M \displaystyle\bigsqcup_{p_1,\varepsilon} N_1 \bigsqcup_{p_2,\varepsilon} \cdots \bigsqcup_{p_n,\varepsilon} N_n$.
\end{theorem}
See Corollary 8.2 in \cite{arezzo2006blowing} for a proof.
\begin{theorem}[Apostolov-Rollin]\label{Apostolov-Rollin thm}
For any $k\geq 2$, the orbifold $\mathbb{CP}_{(-w_0,w)}^{k}$ admits a scalar-flat Kähler ALE metric $g_{\ALE}$ with quotient singularity at infinity $\mathbb{C}^{n}/\Gamma_{(-w_0,w)}$ and a Kähler potential $H$ for the Kähler form $\omega_{\ALE}$ written as 
\begin{equation}\label{pot1}
H=\dfrac{1}{2}\|Z\|^2+A\|Z\|^{4-2k}+O(\|Z\|^{3-2k}),
\end{equation}
when $k\geq 3$ and 
$$H=\dfrac{1}{2}\|Z\|^2+A\log \|Z\|+O(\|Z\|^{-2}),$$
when $k=2$, where $A$ is a real constant and $\|Z\|^2$ is the square norm function on $\mathbb{C}^k$. Furthermore, the constant $A=0$ iff the metric $g_{\ALE}$ is Ricci-flat.
\end{theorem}
See Proposition 17 in \cite{apostolov2017ale} for a proof.
\begin{theorem}[Arezzo-Pacard]
Any compact complex surface of general type admits cscK metrics.
\end{theorem}
We finish this section with a concrete example that we construct. We will demonstrate how to find a constant scalar curvature Kähler metric on it using our theorem at the end of the paper.
\subsection{An example of orbifold with singularities of type $\mathcal{I}$}
Let $r$ be a natural number and consider the cyclic subgroup of $U(1)\leq U(k)$ given by
 $$\Gamma_{(-r,1,\ldots,1)}=\braket{\xi \Id_k}\cong\mathbb{Z}_r,$$
where $\xi=e^{\frac{2\pi i }{r}}$ is the primitive $r$-th rooth of unity. Then for $l\geq 2$,
$$\Gamma:=\braket{\diag(\Id _l,\gamma): \gamma\in \Gamma_{(-r,1,\ldots,1)}}\subset U(l)\times U(k)\subset U(l+k),$$
is a finite subgroup of the isometry group $\Iso(\mathbb{CP}^{l+k-1},g_{\FS})\cong U(l+k)/U(1)$ acting on $\mathbb{CP}^{l+k-1}$ via the standard action on $\mathbb{C}^{l+k}$, where $g_{\FS}$ is the Fubini-Study metric. 
\begin{lemma}
The orbifold $M=\mathbb{CP}^{l+k-1}/\Gamma$ has two disjoint singular strata of type $\mathcal{I}$ at 
\begin{align*}
S_1&=\{[z_0:\ldots:z_{l-1}:0:\ldots:0]\in \mathbb{CP}^{l+k-1}/\Gamma:[z_0:\ldots:z_{l-1}]\in \mathbb{CP}^{l-1}\}\cong \mathbb{CP}^{l-1}/\Gamma,\\
S_2&=\{[0:\ldots:0:z_{l}:\ldots:z_{l+k-1}]\in \mathbb{CP}^{l+k-1}/\Gamma:[z_{l}:\ldots:z_{l+k-1}]\in \mathbb{CP}^{k-1}\}\cong   \mathbb{CP}^{k-1}/\Gamma.
\end{align*}
Also, 
$$N_M(S_1)=(\underbrace{\mathcal{O}_{\mathbb{CP}^{l-1}}(1)\oplus\ldots\oplus\mathcal{O}_{\mathbb{CP}^{l-1}}(1)}_{k-\text{ times}})/\Gamma
$$
and
$$N_M(S_2)=(\underbrace{\mathcal{O}_{\mathbb{CP}^{k-1}}(1)\oplus\ldots\oplus\mathcal{O}_{\mathbb{CP}^{k-1}}(1)}_{l-\text{ times}})/\Gamma.
$$
\end{lemma}
\begin{proof}
The set of points of $\mathbb{CP}^{l+k-1}$ fixed by the action of $\Gamma$ are $[z_0:\ldots:z_{l+k-1}]\in \mathbb{CP}^{l+k-1}$ such that
$$\diag(\Id _l,\xi^s \Id_k)[z_0:\ldots:z_{l+k-1}]=[z_0:\ldots:z_{l+k-1}],0\leq s<r,$$
so this action fixes the two disjoint submanifolds
\begin{align*}
S_1&=\{[z_0:\ldots:z_{l-1}:0:\ldots:0]\in \mathbb{CP}^{l+k-1}/\Gamma:[z_0:\ldots:z_{l-1}]\in \mathbb{CP}^{l-1}\}\cong \mathbb{CP}^{l-1},\\
S_2&=\{[0:\ldots:0:z_{l}:\ldots:z_{l+k-1}]\in \mathbb{CP}^{l+k-1}/\Gamma:[z_{l}:\ldots:z_{l+k-1}]\in \mathbb{CP}^{k-1}\}\cong   \mathbb{CP}^{k-1}/\Gamma_{(-r,1,\ldots,1)}.
\end{align*}
To identify the normal bundles note that $S_1$ could be considered as the intesection of $k$ hyperplanes $D_1 \cap \ldots \cap D_{k}$, so that
\begin{equation}
N_M(S_1)=(N_M(D_1)|_{S_1}\oplus\ldots\oplus N_M(D_{k})|_{S_1})/\Gamma\\
=(\mathcal{O}_{\mathbb{CP}^{l-1}}(1)\oplus\ldots\oplus\mathcal{O}_{\mathbb{CP}^{l-1}}(1))/\Gamma.
\end{equation}
Similarly, for $S_2$, we can write it as the intersection of $l$ hyperplanes $D_1'\cap D_2'\cap \ldots \cap D_{l}'$, so we have that
\begin{equation}
N_M(S_2)=(N_M(D_1')|_{S_2}\oplus\ldots\oplus N_M(D_{l}')|_{S_2})/\Gamma=(\mathcal{O}_{\mathbb{CP}^{k-1}}(1)\oplus\ldots\oplus\mathcal{O}_{\mathbb{CP}^{k-1}}(1))/\Gamma.
\end{equation}
\end{proof}
Note that $\Aut(\mathbb{CP}^{l+k-1}) = \PGL(l+k, \mathbb{C})$. For \[ U = \left(\begin{array}{@{}c|c@{}} A & B \\ \hline C & D \end{array}\right) \in M_{l+k}(\mathbb{C}) \] and $\gamma = (\Id_l, \xi^q \Id_k) \in \Gamma$, we have \[ U \gamma = \left(\begin{array}{@{}c|c@{}} A & \xi^q B \\ \hline C & \xi^q D \end{array}\right) \] and \[ \gamma U = \left(\begin{array}{@{}c|c@{}} A & B \\ \hline \xi^q C & \xi^q D \end{array}\right), \] so $U\gamma=\gamma U$ and $\xi^q\neq 1$ imply that $B=0$ and $C=0$, so the orbifold $M=\mathbb{CP}^{l+k-1}/\Gamma$ has group of automorphism $\mathbb{P}(\GL(l,\mathbb{C})\times \GL(k,\mathbb{C}))$, which is still quite large. To obtain an example with discrete automorphism group, we will use Theorem 1.4 in \cite{arezzo2009blowing} by blowing-up $\mathbb{CP}^{l+k-1}$ at a sufficient number of points $\{p_1,p_2,\ldots,p_n\}$. Let $\mathfrak{h}=\Lie(\PGL(l+k))$ be the Lie algebra of killing vector fields with zeros on $\mathbb{CP}^{l+k-1}$ and denote by $\mathfrak{h}^{\Gamma}$ the Lie subalgebra of $\mathfrak{h}$ consisting of $\Gamma$-invariant vector fields. Also denote the corresponding momentum maps by $\mu:\mathbb{CP}^{l+k-1}\to \mathfrak{h}^*$ and $\mu^{\Gamma}:\mathbb{CP}^{l+k-1}\to {\mathfrak{h}^\Gamma}^*$. 
If $\iota: \mathfrak{h}^{\Gamma}\to \mathfrak{h}$ is the inclusion map and $\iota^*: \mathfrak{h}^*\to \mathfrak{h}^{\Gamma^*}$ the dual map, notice that $\mu^{\Gamma}=\iota^* \circ \mu$. 
\\Using lemma 6.3 in \cite{arezzo2009blowing}, there exists $n_g\geq \dim \mathfrak{h}=(l+k-1)^2-1$  such that for all $n\geq n_g$, there exists a nonempty open set $V_n\subset \{(p_1,p_2,\ldots,p_n)\in (\mathbb{CP}^{l+k-1})^n:p_a\neq p_b \quad\forall a\neq b\}$ such that, for all $(p_1,p_2,\ldots,p_n)\in V_n$, $\{\mu(p_1),\ldots, \mu(p_n)\}$ spans $\mathfrak{h}^*$ and there exist $a_1,a_2,\ldots,a_n>0$ such that $\displaystyle\sum_{i=1}^na_i\mu(p_i)=0\in \mathfrak{h}^*$. 
In particular, since $\mu^{\Gamma}=\iota^* \circ \mu$, this means that:
\begin{enumerate}
\item The set $\{\mu^{\Gamma}(p_1),\ldots, \mu^{\Gamma}(p_n)\}$ spans ${\mathfrak{h}^\Gamma}^*$.
\item There exist positive integers $a_1,a_2,\ldots,a_n$ with $a_i=a_j$ if $p_j=\sigma(p_{j'})$ for some $\sigma \in \Gamma$  and $\displaystyle\sum_{i=1}^na_i\mu^{\Gamma}(p_i)=0\in {\mathfrak{h}^\Gamma}^*$.
\end{enumerate}
Without loss of generality, we can choose $\{p_1,p_2,\ldots,p_n\} \subset \mathbb{CP}^{l+k-1} \setminus (S_1 \cup S_2)$, since $V_n$ is open and $\mathbb{CP}^{l+k-1} \setminus (S_1 \cup S_2)$ is open and dense. Note that if we add a point $p_{n+1}$, conditions (1) and (2) remain valid because of Lemma 6.2 in \cite{arezzo2009blowing}. Thus, by adding points if necessary, we can assume that the set $\{p_1,\ldots,p_n\}$ is $\Gamma$-invariant. Using Theorem 1.4 in \cite{arezzo2009blowing}, there exist $c>0$ and $\varepsilon_0>0$ such that for all $\varepsilon \in (0,\varepsilon_0)$, there exists a $\Gamma$-invariant constant scalar curvature \text{Kähler} metric $\widetilde{g}_\varepsilon$ on $\Bl_{\{p_1,\ldots,p_n\}}^{\mathbb{CP}^{l+k-1}}$. This metric induces a constant scalar curvature \text{Kähler} metric $g_\varepsilon$ on $\Bl_{\{p_1,\ldots,p_n\}}^{\mathbb{CP}^{l+k-1}/\Gamma}$. Consider the quotient map $q: \mathbb{CP}^{l+k-1} \to \mathbb{CP}^{l+k-1}/\Gamma$ and choose $q_1,\ldots,q_m \in \mathbb{CP}^{l+k-1}$ so that $q(\{p_1,\ldots,p_n\}) = \{q_1,\ldots,q_m\}$, where $m = \frac{n}{|\Gamma|}$.
The natural identification,
$$\left(\Bl_{\{p_1, \ldots, p_n\}}^{\mathbb{CP}^{l+k-1}}\right)/\Gamma \cong \Bl_{\{q_1, \ldots, q_m\}}^{\mathbb{CP}^{l+k-1}/\Gamma},$$
 then gives us a constant scalar curvature \kahler metric on the orbifold $X=\Bl_{\{q_1,  \ldots, q_m\}}^{\mathbb{CP}^{l+k-1}/\Gamma}$. 
 \\Using the Proposition \ref{bihologrroupblow} on page \pageref{bihologrroupblow} the Lie algebra of holomorphic vector fields of $\Bl_{\{q_1,  \ldots, q_m\}}^{\mathbb{CP}^{l+k-1}/\Gamma}$ is realized as 
$$\{v\in \Lie(\mathbb{P}(\GL(l)\times \GL(k))):v(q_1)=\ldots =v(q_m)=0\},$$
so the identity component of the automorphism group of $\Aut\left( \Bl_{\{q_1, \ldots, q_m\}}^{\mathbb{CP}^{l+k-1}/\Gamma} \right) $ is the subgroup of elements of $\mathbb{P}(\GL(l)\times \GL(k))$ which fix $\{q_1, \ldots, q_m\}$. Now, a fixed point of $f\in \mathbb{P}(\GL(l)\times \GL(k))$ is the same as an eigenvector of a choice of representative $\widetilde{f}\in \GL(l)\times \GL(k)$, so $f(q)=q$ means $\widetilde{f}(p)=\lambda p$ for a representative $p$. 

Suppose that we pick $n> l+k$ and the first $l+k$ points $p_1,\ldots, p_{l+k}$ in $\mathbb{CP}^{l+k-1}$ represented by points $\tilde{p}_1,\ldots, \tilde{p}_{l+k}\in \mathbb{C}^{l+k}$ forming a basis of $\mathbb{C}^{l+k}$ and with $q(p_i)\neq q(p_j)$ for $i\neq j$ with $i,j\leq l+k$, where $q:\mathbb{CP}^{l+k-1}\to \mathbb{CP}^{l+k-1}/\Gamma$ is the quotient map. Then with respect to this basis, the lift $\widetilde{f}\in \GL(l)\times \GL(k)$ of an element $f\in \mathbb{P}(\GL(l)\times \GL(k))$ such that $f(p_i)=p_i$ for $i\in\{1,\ldots,k+l\}$ will be diagonal. Then if we take $p_{l+k+1}$ such that its representative is $\tilde{p}_{l+k+1}=\displaystyle \sum_{i=1}^{l+k}\tilde{p}_i$, we see that the only element $f\in \mathbb{P}(\GL(l)\times \GL(k))$ fixing $p_1,\ldots, p_{l+k+1}$ is the identity (i.e, $\tilde{f}$ is a multiple of the identity). Perturbing the representatives $\tilde{p}_1,\ldots, \tilde{p}_{l+k}$ if necessary, we can assume that $q(p_{l+k+1})\notin S_1\cup S_2$. We have proven the following theorem:
\begin{theorem}\label{example}
Let $r$ be a natural number and consider the cyclic subgroup of $U(1)\leq U(k)$ given by
 $$\Gamma_{(-r,1,\ldots,1)}=\braket{\xi \Id_k}\cong\mathbb{Z}_r,$$
where $\xi=e^{\frac{2\pi i }{r}}$ is the primitive $r$-th rooth of unity. Then for $l\geq 2$,
$$\Gamma:=\braket{\diag(\Id _l,\gamma): \gamma\in \Gamma_{(-r,1,\ldots,1)}}\subset U(l+k),$$
is a finite subgroup of the isometry group $\Iso(\mathbb{CP}^{l+k-1},g_{\FS})$ acting on $\mathbb{CP}^{l+k-1}$ via the standard action on $\mathbb{C}^{l+k}$, where $g_{\FS}$ is the Fubini-Study metric. Then the orbifold $M=\mathbb{CP}^{l+k-1}/\Gamma$ has two disjoint singularities of type $\mathcal{I}$ at 
\begin{align*}
S_1 & = \{ [z_0 : \ldots : z_{l-1} : 0 : \ldots : 0] \in \mathbb{CP}^{l+k-1} : [z_0 : \ldots : z_{l-1}] \in \mathbb{CP}^{l-1} \} \cong \mathbb{CP}^{l-1}, \\
S_2 & = \{ [0 : \ldots : 0 : z_{l} : \ldots : z_{l+k-1}] \in \mathbb{CP}^{l+k-1} : [z_{l} : \ldots : z_{l+k-1}] \in \mathbb{CP}^{k-1} \} \cong \mathbb{CP}^{k-1} / \Gamma_{(-r, 1, \ldots, 1)}.
\end{align*}
Also, we can choose $\{ p_1, p_2, \ldots, p_{r_m} \} \subset \mathbb{CP}^{l+k-1} \setminus (S_1 \cup S_2)$ with $m \geq l+k+1$ such that there exists a constant scalar curvature \text{Kähler} metric on the orbifold.
$$
X = \Bl_{\{q_1, \ldots, q_m\}}^{\mathbb{CP}^{l+k-1}/\Gamma},
$$
where $q_1, \ldots, q_m \in \mathbb{CP}^{l+k-1}/\Gamma$ are such that 
$$
q(\{p_1, \ldots, p_{r_m}\}) = \{q_1, \ldots, q_m\},
$$
with $q : \mathbb{CP}^{l+k-1} \to \mathbb{CP}^{l+k-1}/\Gamma$ being the quotient map. Furthermore, the identity component of the automorphism group $\Aut\left( \Bl_{\{q_1, \ldots, q_m\}}^{\mathbb{CP}^{l+k-1}/\Gamma} \right)$ is trivial.
\end{theorem}

\section{Gluing Technique}
\label{sec: Gluing Technique}
Let \(X\) be an orbifold of depth one  with singularities of type \(\mathcal{I}\) along a connected suborbifold \(Y\) of complex codimension \(k\) i.e., we can find a finite subgroup \(\Gamma_{(-w_0, w)}\) of \(U(k)\) acting freely on \(\mathbb{C}^k \setminus \{0\}\) such that any point $p\in Y$ has a local orbifold uniformization chart of the form \( \mathbb{C}^{n-k} \times (\mathbb{C}^k / \Gamma_{(-w_0, w)})\) and \(\Gamma_{(-w_0, w)}\) is of type \(\mathcal{I}\) in the sense of Definition \ref{typeI}. Suppose that \(X\) admits a cscK metric. Let \(\pi: \widehat{X} \to X\) be a partial resolution of \(X\) by performing a \((-w_0, w)\)-weighted blow-up along \(Y\) as described on page \pageref{wbb}. 

The goal of this section will be to construct a \kahler metric on the resolution $\widehat{X}$ which is close to the cscK metric $\omega_X$ on $X$. To do so, we will follow the approach of \cite{conlon2019quasi} and introduce an auxiliary space on which this construction will take place. We follow the following steps:\\

Step 1: Consider first the orbifold with boundary $X\times [0,\infty)_{\varepsilon}$ and blow-up the submanifold $Y\times \{0\}$ in the sense of Melrose to obtain the orbifold with corners 
$\mathcal{X}:=[X\times [0,\infty)_{\varepsilon}, Y\times \{0\}]$ where $\varepsilon$ is the parameter of deformation. Let $\beta: \mathcal{X}\to X\times [0,\infty)_{\varepsilon}$ be the corresponding blow-down map.

\begin{figure}[hbt!]
\begin{tikzpicture}[x=0.75pt,y=0.75pt,yscale=-1,xscale=1]

\draw    (310.5,260.25) -- (310.89,143.44) ;
\draw [shift={(310.9,141.44)}, rotate = 90.19] [color={rgb, 255:red, 0; blue, 0; blue, 0 }  ][line width=0.75]    (10.93,-3.29) .. controls (6.95,-1.4) and (3.31,-0.3) .. (0,0) .. controls (3.31,0.3) and (6.95,1.4) .. (10.93,3.29)   ;
\draw [color={rgb, 255:red, 0; blue, 118; blue, 255 }  ,draw opacity=1 ]   (91,260) -- (530,260.5) ;
\draw  [color={rgb, 255:red, 249; blue, 7; blue, 7 }  ,draw opacity=1 ][fill={rgb, 255:red, 247; blue, 4; blue, 4 }  ,fill opacity=1 ] (307.94,260.25) .. controls (307.94,258.83) and (309.08,257.69) .. (310.5,257.69) .. controls (311.92,257.69) and (313.06,258.83) .. (313.06,260.25) .. controls (313.06,261.67) and (311.92,262.81) .. (310.5,262.81) .. controls (309.08,262.81) and (307.94,261.67) .. (307.94,260.25) -- cycle ;

\draw (292.83,120.27) node [anchor=north west][inner sep=0.75pt]    {$[ 0,+\infty )_{\varepsilon }$};
\draw (489,274.4) node [anchor=north west][inner sep=0.75pt]  [color={rgb, 255:red, 0; blue, 119; blue, 255 }  ,opacity=1 ]  {$X$};
\draw (304,274.4) node [anchor=north west][inner sep=0.75pt]    {$\textcolor[rgb]{1,0,0}{Y}$};
\end{tikzpicture}
\caption{\small Orbifold with corner $X\times [0,\infty)_{\varepsilon}$}
\end{figure}
In local coordinates near $Y\times \{0\}$, this means that we replace $\bC^{n-k}\times \bC^k\Big/\Gamma\times [0,\infty)$ by $\bC^{n-k}\times \mathbb{S}^{2k}_+\Big/\Gamma\times [0,+\infty)$, where $\mathbb{S}^{2k}_+$ is the half sphere $\mathbb{S}^{2k}_+=\{(z,\varepsilon)\in \bC^k\times [0,\infty):\vert z\vert ^2+\varepsilon^2=1\}$ and the local blow-down map 
$$\bC^{n-k}\times \mathbb{S}^{2k}_+\Big/\Gamma\times [0,+\infty) \to \bC^{n-k}\times \bC^k\Big/\Gamma\times [0,\infty),$$
 is given by $(w,(z,\varepsilon),r)\mapsto (w,rz,r\varepsilon)$.\\
  Let $H_1$ be the boundary hypersurface of $\mathcal{X}$ obtained by the blow-up of $Y\times \{0\}$ and let $H_2$ be the boundary hypersurface corresponding to the lift of $X\times \{0\}$ in $X\times [0,\infty)_{\varepsilon}$.  In our metric model $H_1=\overline{N_X(Y)}=N_X(Y)\sqcup S(N_X(Y))$ is the radial compactification of the normal bundle $N_X(Y)$ of $Y$ and $H_2=[X,Y]$. Also note that $H_1\cap H_2=\partial H_1=\partial H_2=S(N_X(Y))$ and 
\begin{equation}\label{inner}
\widehat{H}_1^{\circ}\cong \mathcal{O}_{E\diagup Y}(-w_0),
\end{equation}
where $E=\mathbb{P}_w(W)$ is the weighted projectivization of some vector bundle $W\to Y$ of rank $k$ such that $N_X(Y)=W\Big/\Gamma_{(-w_0,w)}$.

\begin{figure}[hbt!]
\begin{tikzpicture}[x=0.75pt,y=0.75pt,yscale=-1,xscale=1]

\draw  [draw opacity=0][line width=0.75]  (247.45,226.96) .. controls (248.09,185.96) and (281.53,152.92) .. (322.69,152.92) .. controls (363.83,152.92) and (397.25,185.92) .. (397.94,226.89) -- (322.69,228.17) -- cycle ; \draw  [color={rgb, 255:red, 238; blue, 10; blue, 10 }  ,draw opacity=1 ][line width=0.75]  (247.45,226.96) .. controls (248.09,185.96) and (281.53,152.92) .. (322.69,152.92) .. controls (363.83,152.92) and (397.25,185.92) .. (397.94,226.89) ;  
\draw [color={rgb, 255:red, 0; blue, 255; blue, 250 }  ,draw opacity=1 ]   (397.94,226.89) -- (540.06,226.49) ;
\draw    (322.03,152.38) -- (322.42,35.57) ;
\draw [shift={(322.43,33.57)}, rotate = 90.19] [color={rgb, 255:red, 0; blue, 0; blue, 0 }  ][line width=0.75]    (10.93,-3.29) .. controls (6.95,-1.4) and (3.31,-0.3) .. (0,0) .. controls (3.31,0.3) and (6.95,1.4) .. (10.93,3.29)   ;
\draw [color={rgb, 255:red, 0; blue, 255; blue, 250 }  ,draw opacity=1 ]   (105.32,227.36) -- (247.45,226.96) ;

\draw (312.76,158.1) node [anchor=north west][inner sep=0.75pt]  [color={rgb, 255:red, 244; blue, 9; blue, 9 }  ,opacity=1 ] [align=left] {$\displaystyle H_{1}$};
\draw (458.27,232.85) node [anchor=north west][inner sep=0.75pt]  [color={rgb, 255:red, 0; blue, 118; blue, 255 }  ,opacity=1 ] [align=left] {$\displaystyle H_{2}$};
\draw (222.69,232.07) node [anchor=north west][inner sep=0.75pt]    {$\mathcal{X} =[ X\times [ 0,+\infty )_{\varepsilon } ,Y\times \{0\}]$};
\draw (304.83,11.87) node [anchor=north west][inner sep=0.75pt]    {$[ 0,+\infty )_{\varepsilon }$};
\draw (162.27,232.85) node [anchor=north west][inner sep=0.75pt]  [color={rgb, 255:red, 0; blue, 118; blue, 255 }  ,opacity=1 ] [align=left] {$\displaystyle H_{2}$};

\end{tikzpicture}
\caption{\small Blowing-up the orbifold $X\times [0,\infty)_{\varepsilon}$ along $Y\times \{0\}$}
\end{figure}

Step 2: The partial resolution $\pi: \widehat{X}\to X$ naturally induces a partial resolution
$$\widehat{\mathcal{X}}\xrightarrow{\pi}\mathcal{X}=[X\times  [0,\infty)_{\varepsilon};\overline{Y}\times \{0\}]\xrightarrow{\beta}X\times  [0,\infty)_{\varepsilon}.$$
As an orbifold, $X$ is automatically a stratified space with two strata $\Sigma_1=Y$ and $\Sigma_2=X\setminus Y$. The orbiofld $\mathcal{X}$ has corners with boundary hypersurfaces $H_1=\overline{N_X(Y)}$ and $H_2=[X,Y]$ corresponding to the strata $\Sigma_1$ and $\Sigma_2$. The boundary hypersurfaces are naturally equipped with a fiber bundle structure
$$\xymatrix{\overline{V_1}=\overline{\bC^k\Big/\Gamma}\ar[rr]&&H_1=\overline{N_X(Y)}\ar[d]^{\varphi_1}\\&& S_1=Y},\quad\quad\xymatrix{V_2=\{\pt\}\ar[rr]&&H_2=[X,Y]\ar[d]^{\varphi_2=\id}\\&& S_2=[X,Y]},$$
where $\overline{V_1}$ and $V_2$ are the fibers. \\
This is a particular case of Lemma 4.3 in \cite{conlon2019quasi}, namely the orbifold  $\mathcal{X}$ is in fact an orbifold with fibered corners $(\mathcal{X},\varphi)$ where $\varphi_1:H_1=\overline{N_X(Y)}\to Y$ and $\varphi_2=\id:H_2=[X, Y]\to H_2$ are the fiber bundle maps. Indeed, the partial order in this case is just the order $H_1<H_2$ on the set $\{H_1,H_2\}$ of boundary hypersurfaces and $H_1\cap H_2=\partial H_1=\partial H_2= \partial(\overline{N_X(Y)})$, $\varphi_1{\big|}_{H_1\cap H_2}: H_1\cap H_2\to S_1$ is clearly a surjective submersion, $S_{21}:=\varphi_2(H_1\cap H_2)=H_1\cap H_2=\partial H_2$ is the boundary of $H_2=S_{2}$ and $\varphi_{21}:=\varphi_1{\big|}_{H_1\cap H_2}$ is such that $\varphi_{21} \circ \varphi_2=\varphi_1 \circ \id=\varphi_1$.

Now we consider the Lie algebra of \textbf{$b$-vector fields} in the sense of Example \ref{$b$-vector fields} on the orbifold with fibered corners $\mathcal{X}$, that is, smooth vector fields on $\mathcal{X}$ which are tangent to all boundary
hypersurfaces:
\begin{align*}
\mathcal{V}_{b}(\mathcal{X}):&=\{\xi\in\mathfrak{X}(\mathcal{X}):\xi \text{ tangent to } H_1 \text{ and } H_2 \}\\
&=\{\xi\in\mathfrak{X}(\mathcal{X}):\xi x_1\in x_1\mathcal{C}^\infty(\mathcal{X}) \text{ and } \xi x_2\in x_2\mathcal{C}^\infty(\mathcal{X}) \},
\end{align*}
where $x_1$ and $x_2$ are boundary defining functions for $H_1$ and $H_2$, i.e, $x_i\geq 0$, $x_i^{-1}(0)=H_i$ and $dx_i\neq 0$ on $H_i$. Choose $x_1$ and $x_2$ so that $\varepsilon=x_1.x_2$ is the corresponding total boundary function. In the local coordinates $(x_1,x_2,u_i)$, vector fields in $\mathcal{V}_{b}(\mathcal{X})$ are of the form 
$$\xi=ax_1\dfrac{\partial}{\partial x_1}+bx_2\dfrac{\partial}{\partial x_2}+\displaystyle \sum_{i} c_iu_i\dfrac{\partial}{\partial u_i},$$
where $a, b, c_i$ are smooth functions.\\
Also we consider the Lie subalgebra $\mathcal{V}_{\QAC}(\mathcal{X})$ of \textbf{quasi asymptotically conical vector fields} or \textbf{$\QAC$-vector fields} on $\mathcal{X}$ originally introduced by Conlon, Degeratu and Rochon in \cite{conlon2019quasi}. By definition (see Definitions 1.11 and 1.14 in \cite{conlon2019quasi}), these are the $b$-vector fields $\xi\in \mathcal{V}_{b}(\mathcal{X})$ such that:
\begin{enumerate}
\item $\QAC1$: $\xi{\big|}_{H_i}$ is tangent to the fibers of $\varphi_i:H_i\to S_i$ for all $i$,
\item $\QAC2$: $\xi \varepsilon\in \varepsilon^2C^\infty(\mathcal{X})$.
\end{enumerate}
The Lie algebra $\mathcal{V}_{\QAC}(\mathcal{X})$ is a finitely generated projective $\mathcal{C}^\infty(\mathcal{X})$-module, so by the Serre-Swan theorem (Theorem \ref{Serre-Swan}), there exists a natural smooth vector bundle,
 ${}^\varphi T\mathcal{X}\to \mathcal{X}$, and a natural map $\iota_\varphi:{}^\varphi T\mathcal{X}\to T\mathcal{X}$  such that
$$\mathcal{V}_{\QAC}(\mathcal{X})=(\iota_\varphi)_*\Gamma(\mathcal{X},{}^\varphi T\mathcal{X}).$$
The above vector bundle ${}^\varphi T\mathcal{X}$, originally introduced in \cite{conlon2019quasi} is called the \textbf{$\QAC$-tangent bundle} over $\mathcal{X}$. The \textbf{$\QAC$-cotangent bundle} ${}^\varphi T^*\mathcal{X}$ over $\mathcal{X}$ is the vector bundle dual to ${}^\varphi T\mathcal{X}$. \\
Using the function $\varepsilon=\pr_2\beta\in\mathcal{C}^{\infty}(\mathcal{X})$, where $\pr_2:X\times [0,+\infty)\to [0,+\infty)$ is the projection on the second factor,  we can define the Lie subalgebra of $\mathcal{V}_{\QAC}(\mathcal{X})$,  corresponding to $\QAC$-vector fields tangent to the level sets of $\varepsilon$:
$$\mathcal{V}_{\QAC,\varepsilon}(\mathcal{X}):=\{\xi\in\mathcal{V}_{\QAC}(\mathcal{X}):\xi\varepsilon\equiv 0\}.$$
Again by the Serre-Swan theorem (Theorem \ref{Serre-Swan} above), there exist a natural vector bundle $\mathcal{E}\to \mathcal{X}$ and a natural map $\iota_\varepsilon:\mathcal{E}\to T\mathcal{X}$ such that there is a canonical identification
$$\mathcal{V}_{\QAC,\varepsilon}(\mathcal{X})=(\iota_\varepsilon)_*\Gamma(\mathcal{X},\mathcal{E}).$$
In fact, $\mathcal{E}$ is naturally a vector subbundle of ${}^\varphi T\mathcal{X}$, which induces a natural map $\iota_{\varepsilon}^*:{}^\varphi T^*\mathcal{X}\to\mathcal{E}^*$. This means that a smooth $\QAC$-metric (i.e. a bundle metric for ${}^\varphi T\mathcal{X}$) naturally restricts to define an element of $\Gamma(\mathcal{X},\mathcal{E}^*\otimes \mathcal{E}^*)$. Now we look at the pull-back of the orbifold cscK-metric $g_X$ on $X$ to $g_\mathcal{X}:=\displaystyle\beta^*\pr_1^*{g}_X$ on $\mathcal{X}$, where $\pr_1:X\times [0,+\infty)\to X$ is the projection on the first factor.

\begin{lemma}\label{41}
The pull-back $g_\varepsilon:=\iota_\varepsilon^*\beta^*\pr_1^*g_X$ is such that $\dfrac{g_\varepsilon}{\varepsilon^2}\in \Gamma(\mathcal{X},\mathcal{E}^*\otimes \mathcal{E}^*)$.
\end{lemma}
\begin{proof}
This is a special case of Lemma 4.5 in \cite{conlon2019quasi} but we will provide a direct proof for the convenience of the reader. It is sufficient to show that $\beta^*\pr_1^*s\in \Gamma(\mathcal{X},{}^\varphi T^*\mathcal{X})$ for $s\in \Gamma(X,T^*X)$. Let us choose $(y,z)$ as a local coordinate of $X$, where $y$ is a coordinate of $Y\subseteq X$ and $z=(z_1,\ldots,z_k)\in \bC^k/\Gamma$ is normal to $Y$, i.e, $Y=\{z=0\}$. Write $z=(r,\omega)\in\mathbb{R}^+\times\mathbb{S}^{2k-1}$ in the spherical coordinates. The boundary defining function of $H_1$ and $H_2$ can be choosen to be $x_1=\sqrt{r^2+\varepsilon^2}$ and $x_2=\dfrac{\varepsilon}{x_1}=\dfrac{\varepsilon}{\sqrt{r^2+\varepsilon^2}}$. The bundle ${}^\varphi T^*\mathcal{X}$ over $\mathcal{X}$ is then  generated by $\{\dfrac{\mathrm{d}\varepsilon}{\varepsilon^2},\dfrac{\mathrm{d}x_2}{x_2^2},\dfrac{\mathrm{d}y}{\varepsilon},\dfrac{\mathrm{d}\omega}{x_2} \}$, while $\mathcal{E}^*$ is generated by $\{\dfrac{\mathrm{d}x_2}{x_2^2},\dfrac{\mathrm{d}y}{\varepsilon},\dfrac{d\omega}{x_2} \}$. We need to check that $\dfrac{r \mathrm{d}\omega}{\varepsilon}$ and $\dfrac{\mathrm{d}r}{\varepsilon}$ are sections of ${}^\varphi T^*\mathcal{X}$.
Now
$$\dfrac{r \mathrm{d}\omega}{\varepsilon}=\dfrac{r}{\sqrt{r^2+\varepsilon^2}}\dfrac{\sqrt{r^2+\varepsilon^2} \mathrm{d}\omega}{\varepsilon}=\dfrac{r}{\sqrt{r^2+\varepsilon^2}}\dfrac{ \mathrm{d}\omega}{x_2}.$$
Since the coefficient $\dfrac{r}{\sqrt{r^2+\varepsilon^2}}=\sqrt{1-x_2^2}$ is a smooth function on $\mathcal{X}$, this shows that $\dfrac{r \mathrm{d}\omega}{\varepsilon}$ is a smooth section of $\mathcal{E}^*$. On the other hand:
\begin{align*}
\dfrac{\mathrm{d}r}{\varepsilon}&=\dfrac{1}{\varepsilon}\mathrm{d}(\sqrt{x_1^2-\varepsilon^2})\\
&=\dfrac{1}{\varepsilon}\mathrm{d}(\varepsilon\sqrt{\dfrac{x_1^2}{\varepsilon^2}-1})\\
&=\dfrac{1}{\varepsilon}\mathrm{d}(\varepsilon\sqrt{\dfrac{1}{x_2^2}-1})\\
&=\dfrac{\mathrm{d}\varepsilon}{\varepsilon}\sqrt{\dfrac{1}{x_2^2}-1}+\dfrac{\varepsilon}{\varepsilon}\mathrm{d}(\sqrt{\dfrac{1}{x_2^2}-1})\\
&=\dfrac{\mathrm{d}\varepsilon}{\varepsilon}\sqrt{\dfrac{1}{x_2^2}-1}-\dfrac{\mathrm{d}x_2}{x_2^2\sqrt{1-x_2^2}}.
\end{align*}
Since $\dfrac{\mathrm{d}\varepsilon}{\varepsilon}\sqrt{\dfrac{1}{x_2^2}-1}= \dfrac{\sqrt{1-x_2^2}}{{x_2}}\dfrac{\mathrm{d}\varepsilon}\varepsilon= x_1\sqrt{1-x_2^2}\dfrac{\mathrm{d}\varepsilon}{\varepsilon^2}$ is a smooth section of ${}^\varphi T^*\mathcal{X}$ over $\mathcal{X}$ and $\dfrac{\mathrm{d}x_2}{x_2^2\sqrt{1-x_2^2}}$ is a section of  $\mathcal{E}^*$ over $\mathcal{X}$, this shows that $\dfrac{\mathrm{d}r}{\varepsilon}$ is a smooth section of ${}^\varphi T^*\mathcal{X}$.
\end{proof}
Now we describe the restrictions of $\dfrac{g_\varepsilon}{\varepsilon^2}$ to $H_1$ and $H_2$. First, we describe the restriction of $\mathcal{E}$ to $H_1$ and $H_2$. On $H_1$, restricting the boundary defining function of $\mathcal{X}$ to $H_1$ gives us a Lie algebra of $\QAC$-vector fields and a corresponding $\QAC$-tangent bundle ${}^\varphi T(H_1/S_1)$ in the fibers of $\varphi_1: H_1 \to S_1$. Since these fibers are manifolds with boundary, this means that in each fiber, we have the Lie algebra of scattering vector fields in the Melrose sense. Thus, there is a natural map $\mathcal{E} \big|_{H_1} \to {}^{\SC} T(H_1/S_1)$ and a short exact sequence.
$$0 \to N_1 \mathcal{E} \to \mathcal{E} \big|_{H_1} \to {}^{\SC} T(H_1 / S_1) \to 0,$$
where $N_1\mathcal{E}:=\ker(\mathcal{E}{\big|}_{H_1}\to {}^{\SC} T(H_1/S_1))$.
\\
Since there is a natural inclusion  ${}^{\SC} T(H_1/S_1) \subset \mathcal{E}{\big|}_{H_1}$, the above short exact sequence splits. There is another natural short exact sequence involving ${}^{\SC} T(H_1/S_1)$,
$$0 \to {}^{\SC} T(H_1/S_1) \to {}^{\SC} TH_1 \to x_2\varphi_1^*( TS_1) \to 0,$$
where $TS_1$ is the tangent bundle of $S_1=Y$, so
\begin{equation}\label{yek}
x_2\varphi_1^*(TS_1)={}^{\SC} TH_1/{}^{\SC} T(H_1/S_1).
\end{equation}
Multiplication by the boundary defining function $x_1$ induces the identification
\begin{equation}\label{do}
{}^{\SC} TH_1/{}^{\SC} T(H_1/S_1)=\ker{(\mathcal{E}{\big|}_{H_1}\to {}^{\SC} T(H_1/S_1))}.
\end{equation}
In particular, we see from \eqref{yek} and \eqref{do} that there is a natural identification
$$\ker{(\mathcal{E}{\big|}_{H_1}\to {}^{\SC} T(H_1/S_1))}\cong \varepsilon\varphi_1^*(TS_1).$$
Hence, we have a canonical decomposition 
$$\mathcal{E}{\big|}_{H_1}=\varepsilon\varphi_1^*( TS_1) \oplus {}^{\SC} T(H_1/S_1).$$ 
By Lemma \ref{41}, the family of metrics $\dfrac{g_\varepsilon}{\varepsilon^2}$ splits accordingly
$$\dfrac{g_\varepsilon}{\varepsilon^2}{\big|}_{H_1}=\dfrac{\varphi_1^*g_{S_1}}{\varepsilon^2}+g_{\varphi_1},$$
where $g_{\varphi_1}$ and $\dfrac{\varphi_1^*g_{S_1}}{\varepsilon^2}$ are the metrics induced by $\dfrac{g_\varepsilon}{\varepsilon^2}$ in the fiber of $\varphi_1:H_1\to S_1$ and the bundle $N_1\mathcal{E}$.

The function $\varepsilon$ on $\mathcal{X}$ naturally extends to a smooth function on $\widehat{\mathcal{X}}$, also denoted by $\varepsilon$. Similarly, the boundary defining functions $x_1$ can be chosen to lift to a smooth boundary defining function on $\widehat{\mathcal{X}}$, yielding a natural Lie algebra $\mathcal{V}_{\QAC}(\widehat{\mathcal{X}})$. We can define a Lie subalgebra by
$$\mathcal{V}_{\QAC,\varepsilon}(\widehat{\mathcal{X}}):=\{\xi\in\mathcal{V}_{\QAC}(\widehat{\mathcal{X}}):\xi \varepsilon\equiv 0\}.$$
There is a corresponding vector bundle $\widehat{\mathcal{E}}\to \widehat{\mathcal{X}}$ and a natural map $\iota_\varepsilon:\widehat{\mathcal{E}}\to T\widehat{\mathcal{X}}$ yielding a canonical identification 
$$\mathcal{V}_{\QAC,\varepsilon}( \widehat{\mathcal{X}})=(\iota_\varepsilon)_*\Gamma(\widehat{\mathcal{X}},\widehat{\mathcal{E}}).$$
The following theorem constructs a family of K\"ahler metrics on $\widehat{X}$ which are approximately cscK.
\begin{theorem}\label{main}
Let $(X,\omega_X)$ be a compact cscK complex orbifold of depth one with singularities of type $\mathcal{I}$ along a suborbifolds $Y$. Then, there exists a smooth closed $(1,1)$-form $\omega_{\widehat{\varphi}_1}$ on $\widehat{H}_1$ restricting on each fiber of $\widehat{\varphi}_1$ to the K\"ahler form of a scalar flat $\ALE$-metric asymptotic to $\omega_{\varphi_1}$. Moreover, for $\mu>0$ small, there is $\widehat{\omega}_\varepsilon\in \Gamma(\widehat{\mathcal{X}}, \widehat{\mathcal{E}^*}\wedge \widehat{\mathcal{E}^*})$ such that $\dfrac{\widehat{\omega}_\varepsilon}{\varepsilon^2}{\big|}_{\widehat{H}_1}=\omega_{\widehat{\varphi}_1}+\dfrac{\varphi_1^*\omega_{S_1}}{\varepsilon^2}$,  $\widehat{\omega}_\varepsilon{\big|}_{H_2}=\omega_X{\big|}_{H_2}$ and which yields a positive definite closed $(1,1)-$form on
$$\widehat{X}_{c}=\{p\in\widehat{\mathcal{X}}: \varepsilon(p)=c\}\cong \widehat{X},$$
for each $0<c<\mu$.
\end{theorem}
\begin{proof}
Let $\Gamma = \langle \gamma \rangle$. Since the unitary matrix $\gamma$ is diagonalizable, the eigenspaces of $\gamma$ in each fiber of $W \to Y$ induce an $\omega_X$-orthogonal decomposition
\[ W = \displaystyle\bigoplus_{i=1}^l W_i, \quad \text{where} \quad \dim W_i = k_i, \quad \sum_{i=1}^l k_i = k. \]
Now we can consider the action of $\displaystyle\bigtimes_{i=1}^l U(k_i)$ on $W$. Let $e = (e_1, \ldots, e_k)$ be an orthonormal basis for smooth sections of $W$ on an open set $U \subset Y$, compatible with the decomposition of $W$. This gives us a trivialization $W|_U \cong \mathbb{C}^k \times U$, and in this trivialization, $\gamma$ acts diagonally on $\mathbb{C}^k$ by
$$\gamma=\diag(e^{\frac{iw_1}{w_0}},\ldots,e^{\frac{iw_k}{w_0}}).$$
Let $W|_V \cong \mathbb{C}^k \times V$ be another such trivialization for an open subset $V \subset Y$ with orthonormal basis $e' = (e'_1, \ldots, e'_k)$. Then $e' = fe$ for a smooth function $f:U\cap V\to \displaystyle\bigtimes_{i=1}^l  U(k_i)$. Note that $\widehat{H}_1=\widehat{\overline{N_X(Y)}}$ and $\widehat{H}_1^{\circ}$ is the total space of a vector bundle with fibers isomorphic to $\mathbb{CP}_{(-w_0,w)}^{k}=\mathcal{O}_{\mathbb{CP}_{w}^{k-1}}(-w_0)$. We get a natural line bundle:
\[
\varpi: \widehat{H}_1^{\circ} \to E, \quad \text{where} \quad E=\mathbb{P}_{w}(W) \quad \text{is the weighted fiberwise projectivization of } W.
\]
As discussed in Theorem \ref{Apostolov-Rollin thm} on page \pageref{Apostolov-Rollin thm}, there exists a scalar-flat Kähler metric $g_{\ALE}$ on $\mathbb{CP}_{(-w_0,w)}^{k}$ modelled on $\mathbb{C}^k/\Gamma_{(-w_0,w)}$ at infinity with a Kähler potential $H$ for the Kähler form $\omega_{\ALE}$ written as 
\begin{equation}\label{pot1}
H=\dfrac{1}{2}\|Z\|^2+A\|Z\|^{4-2k}+O(\|Z\|^{3-2k}),
\end{equation}
when $k\geq 3$ and 
$$H=\dfrac{1}{2}\|Z\|^2+A\log \|Z\|+O(\|Z\|^{-2}),$$
when $k=2$, where $A$ is a constant. Since the metric $g_{\ALE}$ and the potential $H$ are invariant under the action of $\displaystyle\bigtimes _{i=1}^l U(k_i)$, we can use the above trivialization to obtain a well-defined fiberwise potential 
$$H_N: N_X(Y)\to \mathbb{R},$$
corresponding to $H$ in each fibers of $N_X(Y)$ for any choice of trivialization as described above. Set 
$$\omega_{\widehat{\varphi}_1}=2\sqrt{-1}\dd (H_N)\quad \text{on} \quad N_X(Y)\setminus Y\cong \widehat{H}_1^{\circ}   \setminus E.$$
On each fiber of $\widehat{H}_1^{\circ} \to Y$, this closed $(1,1)$-form extends to an $\ALE$ scalar flat metric. Globally, this extends to a smooth $(1,1)$-form on $\widehat{H}_1^{\circ}$ that we will also denoted by $\omega_{\widehat{\varphi}_1}$. Since $\omega_{\widehat{\varphi}_1}$ is closed on the complement of $E$, it is closed everywhere by continuity. Hence, $\omega_{\widehat{\varphi}_1}$ is the desired closed $(1,1)$-form. Finally we define
$$\widehat{\omega}_\varepsilon=\omega_X+\varepsilon^2\sqrt{-1}\dd [\gamma_1(\dfrac{d}{r_\varepsilon})f(\dfrac{d}{\varepsilon})],$$
where $\gamma_1:\mathbb{R}\to\mathbb{R}$ is a cut-off function such that $\gamma_1(t)=1$ for $t<1$ and $\gamma_1(t)=0$ for $t>2$, $d=d_{g_X}\circ\pr_1\circ \beta\circ \pi: \widehat{\mathcal{X}}\to [0,+\infty)$, where $d_{g_X}$ is the distance from $Y$ on $X$ with respect to the cscK metric $g_X$ (we can use $d=r$ in terms of the coordinates used in the proof of Lemma \ref{41}), $r_\varepsilon=\varepsilon^{\frac{2k}{2k+1}}$ and
$$f=A\|Z\|^{4-2k}+O(\|Z\|^{3-2k}),$$
is a function defined on the complement of the exceptional divisor such that
$$\omega_{\widehat{\varphi}_1}=\omega_{\varphi_1}+\sqrt{-1}\dd f.$$
By construction, we obtain three different regions on $\widehat{X}$ as follows:
\begin{enumerate}
\item Near the exceptional divisor, on‌ $\Omega_1=\{x\in \widehat{X}:d_{}(x)<r_\varepsilon\}$ we have $\gamma_1(\dfrac{d}{r_\varepsilon})=1$ so we get 
$$\widehat{\omega}_\varepsilon=\omega_X+\varepsilon^2\sqrt{-1}\dd f(\dfrac{d}{\varepsilon}).$$
‌‌‌‌\item On the intermediate region‌ $\Omega_2=\{x\in \widehat{X}:r_\varepsilon<d_{}(x)<2r_\varepsilon\}$ we get 
$$\widehat{\omega}_\varepsilon=\omega_X+\varepsilon^2\sqrt{-1}\dd [\gamma_1(\dfrac{d}{r_\varepsilon})f(\dfrac{d}{\varepsilon})].$$
\item Far from the exceptional divisor,‌ $\Omega_3=\{x\in \widehat{X}:2r_\varepsilon<d_{}(x)\}$ we have‌ ‌‌$\gamma_1(\dfrac{d}{r_\varepsilon})=0$, so we get
$$\widehat{\omega}_\varepsilon=\omega_{X}.$$

\begin{figure}
\centering
\begin{tikzpicture}[x=0.75pt,y=0.75pt,yscale=-1,xscale=1]

\draw  [draw opacity=0][line width=0.75]  (247.45,226.96) .. controls (248.09,185.96) and (281.53,152.92) .. (322.69,152.92) .. controls (363.83,152.92) and (397.25,185.92) .. (397.94,226.89) -- (322.69,228.17) -- cycle ; \draw  [color={rgb, 255:red, 238; blue, 10; blue, 10 }  ,draw opacity=1 ][line width=0.75]  (247.45,226.96) .. controls (248.09,185.96) and (281.53,152.92) .. (322.69,152.92) .. controls (363.83,152.92) and (397.25,185.92) .. (397.94,226.89) ;  
\draw [color={rgb, 255:red, 0; blue, 255; blue, 250 }  ,draw opacity=1 ]   (397.94,226.89) -- (540.06,226.49) ;
\draw    (322.03,152.38) -- (322.42,35.57) ;
\draw [shift={(322.43,33.57)}, rotate = 90.19] [color={rgb, 255:red, 0; blue, 0; blue, 0 }  ][line width=0.75]    (10.93,-3.29) .. controls (6.95,-1.4) and (3.31,-0.3) .. (0,0) .. controls (3.31,0.3) and (6.95,1.4) .. (10.93,3.29)   ;
\draw [color={rgb, 255:red, 0; blue, 255; blue, 250 }  ,draw opacity=1 ]   (105.32,227.36) -- (247.45,226.96) ;
\draw    (397.94,226.89) .. controls (407.65,228.6) and (538.71,184.43) .. (539.71,40.43) ;
\draw    (397.94,226.89) .. controls (407.65,228.6) and (449.71,185.43) .. (450.71,41.43) ;
\draw    (247.45,226.96) .. controls (243.43,229.14) and (193.43,186.14) .. (194.43,42.14) ;
\draw    (247.45,226.96) .. controls (246.94,227.08) and (103.94,188.08) .. (104.94,44.08) ;

\draw (312.76,158.1) node [anchor=north west][inner sep=0.75pt]  [color={rgb, 255:red, 244; blue, 9; blue, 9 }  ,opacity=1 ] [align=left] {$\displaystyle H_{1}$};
\draw (455.27,224.85) node [anchor=north west][inner sep=0.75pt]  [color={rgb, 255:red, 13; blue, 243; blue, 246 }  ,opacity=1 ] [align=left] {$\displaystyle H_{2}$};
\draw (294.83,13.87) node [anchor=north west][inner sep=0.75pt]    {$[ 0,+\infty )_{\varepsilon }$};
\draw (157.27,225.85) node [anchor=north west][inner sep=0.75pt]  [color={rgb, 255:red, 13; blue, 243; blue, 246 }  ,opacity=1 ] [align=left] {$\displaystyle H_{2}$};
\draw (251,96) node [anchor=north west][inner sep=0.75pt]    {$\Omega _{1}$};
\draw (169,138) node [anchor=north west][inner sep=0.75pt]    {$\Omega _{2}$};
\draw (113,174) node [anchor=north west][inner sep=0.75pt]    {$\Omega _{3}$};
\draw (493,175) node [anchor=north west][inner sep=0.75pt]    {$\Omega _{3}$};
\draw (368,97) node [anchor=north west][inner sep=0.75pt]    {$\Omega _{1}$};
\draw (455,138) node [anchor=north west][inner sep=0.75pt]    {$\Omega _{2}$};
\end{tikzpicture}
\caption{\small Three different regions on $\widehat{X}$}
\end{figure}
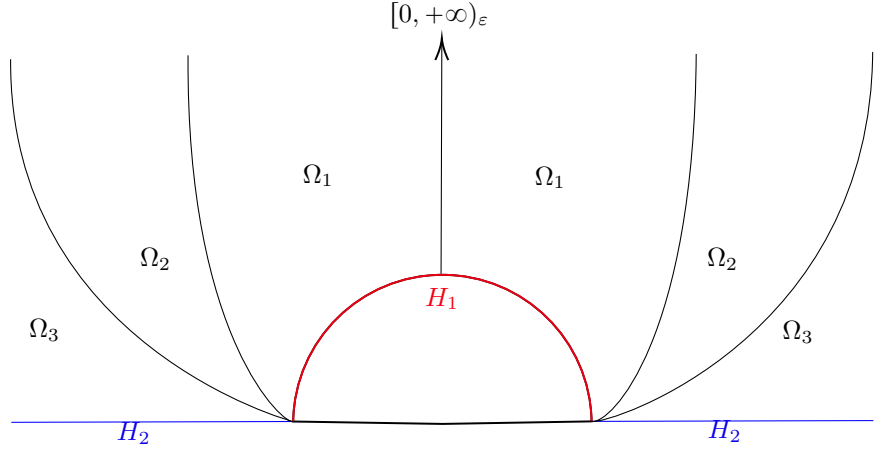
\end{enumerate}
Far from the exceptional divisor, $\widehat{\omega}_\varepsilon$ is well-defined on the complement of the exceptional divisor. Moreover, by construction, $\dfrac{\widehat{\omega}_\varepsilon}{\varepsilon^2}$ extends to a metric on the bundle $\widehat{\mathcal{E}}$ for small $\varepsilon$. In fact on $\widehat{H}_1$ we get
$$\dfrac{\widehat{\omega}_\varepsilon}{\varepsilon^2}{\big|}_{H_1}=\omega_{\widehat{\varphi}_1}+\dfrac{\varphi_1^*\omega_{S_1}}{\varepsilon^2},$$
and on $\widehat{H}_2$ we get same restriction as $\dfrac{\omega_X}{\varepsilon^2}$. Since $\dfrac{\widehat{\omega}_\varepsilon}{\varepsilon^2}$ is positive definite on both $\widehat{\mathcal{E}}{\big|}_{H_1}$ and $\widehat{\mathcal{E}}{\big|}_{H_2}$, so is  $\dfrac{\widehat{\omega}_\varepsilon}{\varepsilon^2}$ near $\widehat{H}_1$ and $\widehat{H}_2$. Consequently, it remains positive definite for small $\varepsilon > 0$, which yields a positive definite closed $(1,1)-$form on the level sets $\widehat{X}_{c}=\{p\in\widehat{\mathcal{X}}: \varepsilon(p)=c\}\cong \widehat{X}$ as claimed.
\end{proof}

\begin{remark}\label{4.3}
Since the cohomology in degree 2 of $\mathbb{CP}_{(-w_0,w)}^{k}$ is generated by the divisor $[\mathbb{CP}_{w}^{k-1}]$, notice that by reparametrizing $\varepsilon$, if necessary, we can assume without loss of generality in Theorem \ref{main} that $[\widehat{\omega}_\varepsilon] = [\omega_X] - \varepsilon^2[E]$.
\end{remark}

\section{Linear analysis}
\label{sec: Linear analysis}
The Kähler metrics $\widehat{\omega}_\varepsilon$ provided by Theorem \ref{main} are not necessarily cscK, but since their asymptotic models on $\widehat{H}_1$ and $H_2$ are, they will be almost cscK for small $\varepsilon$. We can therefore hope to solve the nonlinear equation
$$S(\widehat{\omega}_\varepsilon+\sqrt{-1}\dd u)=R,$$
perturbatively, for some well-chosen constant $R$.\\
To do this, the purpose of this section is to first study the linearization of the scalar curvature of the metric $\widehat{\omega}_\varepsilon$ perturbed by a potential $u$:
\begin{equation}\label{32}
S(\widehat{\omega}_\varepsilon+\sqrt{-1}\dd u)=S(\widehat{\omega}_\varepsilon)+L_{\widehat{\omega}_\varepsilon}(u)+Q_{\widehat{\omega}_\varepsilon}(\nabla^2 u),
\end{equation}
where $L_{\widehat{\omega}_\varepsilon}$ is linear part and $Q_{\widehat{\omega}_\varepsilon}$ is nonlinear part. 
By the Proposition \ref{linears} we get,
$$L_{\widehat{\omega}_\varepsilon}(u)=-(\dfrac{1}{2}\Delta^2_{\widehat{g}_\varepsilon}+\Rc_{\widehat{g}_\varepsilon}.\nabla^2_{\widehat{g}_\varepsilon})u.$$
In terms of Lichnerowicz operator $\mathcal{D}^*\mathcal{D}$, we can write
$$L_{\widehat{\omega}_\varepsilon}(u)=\dfrac{1}{2}\nabla S(\widehat{\omega}_\varepsilon).\nabla u-\mathcal{D}^*\mathcal{D}u.$$
Let us define $\widetilde{L}_\varepsilon:\mathcal{C}^\infty(\widehat{X})_0\times \bR\to \mathcal{C}^\infty(\widehat{X})$ by
\begin{equation}\label{52}
\widetilde{L}_\varepsilon(u,R)=L_{\widehat{\omega}_\varepsilon}(u)-R,
\end{equation}
where $\mathcal{C}^\infty(\widehat{X})_0$ is the space of smooth functions $u$ that have zero integral with respect to the metric $\widehat{g}_\varepsilon$.

\begin{definition}[H\"{o}lder space]
Let $(M, g)$ is a smooth Riemannian manifold, $k\in \mathbb{N}_0$ and $\alpha\in (0,1]$. Then the H\"{o}lder space $\mathcal{C}^{k,\alpha}_g(M)$ consists of functions $f:M\to\mathbb{R}$ such that
$$\|f\|_{\mathcal{C}^{k,\alpha}_g}:=\| f\|_{g,k}+[\nabla^k f]_{g,\alpha}<\infty,$$
where $\| f\|_{g,k}=\displaystyle\sum_{i=0}^k\sup_{p\in M}\| \nabla^i f(p)\|_g$ with $\|\cdot\|_g$ the pointwise norm induced by the metric $g$ and $[\nabla^k f]_{g,\alpha}$ is the H\"{o}lder semi norm defined by
$$[\nabla^k f]_{g,\alpha}=\sup_{\gamma(0)\neq \gamma(1)}\{\dfrac{\|P_\gamma(\nabla^k f(\gamma(0)))-\nabla^k f(\gamma(1))\|_g}{l(\gamma)^\alpha}:\gamma\text{ a smooth curve on } M \},$$
and $P_\gamma$ is parallel transport along $\gamma$.
\end{definition}

\begin{definition}[Weighted H\"{o}lder space]
Let $(M, g)$ is a smooth Riemannian manifold, $k\in \mathbb{N}_0$ and $\alpha\in (0,1]$. The weighted H\"{o}lder space $\rho\mathcal{C}^{k,\alpha}_g(M)$ for a weight function $\rho\in\mathcal{C}^\infty(M,\mathbb{R}^+)$ consists of functions $f:M\to\mathbb{R}$ such that
$$\|f\|_{\rho\mathcal{C}^{k,\alpha}_g}:=\|\dfrac{f}{\rho}\|_{\mathcal{C}^{k,\alpha}_g}<\infty.$$
\end{definition}

On a fiber $\overline{Z_1}$ of $\widehat{\varphi}_1 : \widehat{H}_1 \to \widehat{S}_1$, the metric $\dfrac{\widehat{g}_\varepsilon}{\varepsilon^2}$ induced by restriction is a scalar flat Kähler ALE metric $g_{\widehat{\varphi}_1}$. It is convenient to study the mapping properties of its operator $L_{\widehat{\omega}_\varepsilon}$ in terms of the weighted Hölder space induced by the $b$-metric $g_{\widehat{\varphi}_{1,b}}$ obtained by restriction of $\dfrac{\widehat{g}_\varepsilon}{\rho^2}$ to $\overline{Z_1}$, where $\rho=\sqrt{d^2+\varepsilon^2}$ is a boundary defining function for $\widehat{H}_1$ in $\widehat{\mathcal{X}}$.

\begin{remark}
\begin{enumerate}
\item The restriction of $\widehat{g}_\varepsilon$ to $\widehat{H}_1$ induces a family of fiberwise b-metrics in the fibers of $\widehat{\varphi}_1:\widehat{H}_1\to \widehat{S}_1$.
\item The restriction $\dfrac{\widehat{g}_\varepsilon}{\rho^2}|_{H_2}=\dfrac{g_X}{d^2}|_{H_2}$ is an edge metric in the sense of Mazzeo.
\end{enumerate}
\end{remark}

\begin{lemma}\label{red}
If $\delta<0$ and $k>2$, then the linear operator
$$L_{\omega_{\widehat{\varphi}_1}}:(\dfrac{\rho}{\varepsilon})^\delta\mathcal{C}^{4,\alpha}_{g_{\widehat{\varphi}_{1,b}}}(Z_1)\to (\dfrac{\rho}{\varepsilon})^{\delta-4}\mathcal{C}^{0,\alpha}_{g_{\widehat{\varphi}_{1,b}}}(Z_1),$$
has trivial kernel, where $Z_1$ is the interior of $\overline{Z_1}$.
\end{lemma}
\begin{proof}
We will proceed as the proof of Proposition 8.9 in \cite{szekelyhidi2014introduction}. Since the scalar curvature of $g_{\widehat{\varphi}_1}$ is zero,
$$L_{\omega_{\widehat{\varphi}_1}}(u)=-\mathcal{D}^*\mathcal{D}u.$$
Also, by construction in the proof of Theorem \ref{main}, the \kahler potential of $\omega_{\widehat{\varphi}_1}$ as $\| Z\|\to \infty$ is
$$\dfrac{1}{2}\| Z\|^2+A\| Z\|^{4-2k}+O(\| Z\|^{3-2k}),$$
where $Z$ denotes the Euclidean coordinates on $(\mathbb{C}^k/\Gamma_{(-w_0,w)})\setminus \{0\}$ identified with complement of the exceptional divisor in $Z_1$.
Suppose $u\in (\dfrac{\rho}{\varepsilon})^\delta\mathcal{C}^{4,\alpha}_{g_{\widehat{\varphi}_{1,b}}}(Z_1)$ and $\mathcal{D}^*\mathcal{D}u=0$. Consider a smooth cutoff function $\gamma$ such that it vanishes on $\pi^{-1}(B_1(0))$ and is 1 outside $\pi^{-1}(B_2(0))$, where $B_r(0)$ is the ball with radius $r$ centred origin in $\mathbb{C}^k/\Gamma_{(-w_0,w)}$, so $\gamma u$ is a smooth function on $\bC^k/\Gamma_{(-w_0,w)}$. Now we try to compare the Lichnerowicz operator $\mathcal{D}^*\mathcal{D}$ with Euclidean operator $\Delta_{\Euc}^2$, note that $\mathcal{D}_{\Euc}^*\mathcal{D}_{\Euc}=\dfrac{1}{2}\Delta^2_{\Euc}$, but
$$\mathcal{D}^*\mathcal{D}(u)=\dfrac{1}{2}\Delta^2_{g_{\widehat{\varphi}_1}}(u)+R^{\bar{k}j}\partial_j\partial_ku+\dfrac{1}{2}\nabla S(\omega_{\ALE})\cdot \nabla u,$$
so $\mathcal{D}^*\mathcal{D}-\dfrac{1}{2}\Delta_{\Euc}^2$ is in $O(\| Z\|^{2-2k})$ as a scattering operator as $\| Z\|\to \infty$. 

Since $u\in (\dfrac{\rho}{\varepsilon})^\delta\displaystyle\mathcal{C}^{4,\alpha}_{g_{\widehat{\varphi}_{1,b}}}(Z_1)$, this means that $\Delta_{\Euc}^2(\gamma u)\in(\dfrac{\rho}{\varepsilon})^{\delta-2-2k}\mathcal{C}^{0,\alpha}_{g_{\widehat{\varphi}_{1,b}}}((\bC^k\setminus\{0\})/\Gamma_{(-w_0,w)})$. Applying Theorem 8.3 in \cite{szekelyhidi2014introduction} on the orbifold $\bC^k/\Gamma_{(-w_0,w)}$, we conclude that there exist a function\linebreak $v\in (\dfrac{\rho}{\varepsilon})^{\delta+2-2k}\mathcal{C}^{4,\alpha}_{g_{\widehat{\varphi}_{1,b}}}((\bC^k\setminus \{0\})/\Gamma_{(-w_0,w)})$ such that $\Delta_{\Euc}^2(v)=\Delta_{\Euc}^2(\gamma u)$. Hence $v-\gamma u$ is a biharmonic function that decays at infinity. Since there is no indicial roots in $(4-2k,0)$ we have 
$$v-\gamma u\in (\dfrac{\rho}{\varepsilon})^{4-2k}\mathcal{C}^{4,\alpha}_{g_{\widehat{\varphi}_{1,b}}}((\bC^k\setminus \{0\})/\Gamma_{(-w_0,w)}\setminus B_2(0)),$$
 and this implies that
$$\gamma u\in (\dfrac{\rho}{\varepsilon})^{4-2k}\mathcal{C}^{4,\alpha}_{g_{\widehat{\varphi}_{1,b}}}((\bC^k\setminus \{0\})/\Gamma_{(-w_0,w)}\setminus B_2(0)),$$ 
so $u\in (\dfrac{\rho}{\varepsilon})^{4-2k}\mathcal{C}^{4,\alpha}_{g_{\widehat{\varphi}_{1,b}}}(Z_1)$. Since we assume $k>2$, this decay allows us to integrate by parts
$$\displaystyle\int_{Z_1}\|\mathcal{D} u\|^2 \ \mathrm{d} \omega_{\widehat{\varphi}_1}=\displaystyle\int _{Z_1}u \mathcal{D}^*\mathcal{D} u   \mathrm{d}\omega_{\widehat{\varphi}_1}=0,$$
and then $\mathcal{D} u=0$, so $\nabla^{1,0}u$ is a holomorphic vector field on $Z_1=\widehat{\bC^k/\Gamma_{(-w_0,w)}}$. The resolution \linebreak $\pi:Z_1\to \bC^k/\Gamma_{(-w_0,w)}$ implies a biholomorphism $Z_1\setminus \pi^{-1}(0)\cong (\bC^k\setminus\{0\})/\Gamma_{(-w_0,w)}$, so we can represent the holomorphic vector field $\nabla^{1,0}u|_{Z_1\setminus \pi^{-1}(0)}$ as $\displaystyle \sum_{i=1}^ka_i\dfrac{\partial}{\partial z_i}$ where $a_i$ are $\Gamma_{(-w_0,w)}$-invariant function on $\bC^k\setminus\{0\}$. By applying the Hartogs theorem, for each $a_i$ there exist a holomorphe extension $\widetilde{a}_i$ on $\bC^k$. Because of the boundedness of $\widetilde{a}_i$, Liouville's theorem implies that $\widetilde{a}_i$ is a constant function, hence its decay at infinity implies that $\widetilde{a}_i=0$. So $\nabla^{1,0}u\equiv 0$ on $Z_1\setminus \pi^{-1}(0)$. Now notice that $Z_1\setminus \pi^{-1}(0)$ is a dense subset of $Z_1$, so by continuity $\nabla^{1,0}u\equiv 0$ on $Z_1$. Finally, $\nabla u=\nabla^{1,0}u+\nabla^{0,1}u=\nabla^{1,0}u+\overline{\nabla^{1,0}u}=0$, so $u$ is a constant function on $Z_1$ that decays at infinity, implying that $u=0$. 
\end{proof}
\begin{lemma}\label{new11}
If $\delta<0$ and $L_0(u)= 0$ for $u\in (\dfrac{\rho}{\varepsilon})^\delta C^{4,\alpha}_{\frac{\widehat{g}_\varepsilon}{\rho^2}}(Z_1\times \mathbb{C}^{n-k})$, then $u=0$. Here $L_0$ denotes the Lichnerowicz operator on the product space.
\end{lemma}
\begin{proof}
It suffices to replace $\Bl_0^{\mathbb{C}^k}$ by $\mathbb{CP}^k_{(-w_0,w)}$ in Lemma 11 of \cite{seyyedali2020extremal} and use Lemma \ref{red} instead of Proposition 8.9 in \cite{szekelyhidi2014introduction} in the proof of Lemma 11 in \cite{seyyedali2020extremal}.
\end{proof}
\begin{lemma}\label{yellow}
If $4-2k<\delta<0$ and $\Delta_{\Euc}^2u= 0$ for $u\in (1+\dfrac{d^2}{\varepsilon^2})^{\frac{\delta }{2}}C^{4,\alpha}_{\frac{g_X}{\rho^2}}(((\mathbb{C}^k/\Gamma_{(-w_0,w)})\setminus\{0\})\times \mathbb{C}^{n-k})$, then $u=0$.
\end{lemma}
\begin{proof}
It suffices to pull-back to $(\mathbb{C}^k\setminus\{0\})\times \mathbb{C}^{n-k}$ and use Lemma 12 in \cite{seyyedali2020extremal}.
\end{proof}

\begin{lemma}\label{green}
Suppose that $X$ has no non trivial holomorphic vector fields. If $4-2k<\delta<0$, then the linear operator
$$\widetilde{L}_{\omega_0}:\rho^\delta\mathcal{C}^{4,\alpha}_{\frac{g_{X}}{\rho^2}}(H_2)_0\times \mathbb{R}\to \rho^{\delta-4}\mathcal{C}^{0,\alpha}_{\frac{g_{X}}{\rho^2}}(H_2),$$
defined by $\widetilde{L}_{\omega_0}(u,R)=L_{\omega_0}(u)-R$ where $\omega_0=\omega_X$ and $\rho^\delta\mathcal{C}^{4,\alpha}_{\frac{g_{X}}{\rho^2}}(H_2)_0$ is the subspace of functions in $\rho^\delta\mathcal{C}^{4,\alpha}_{\frac{g_{X}}{\rho^2}}(H_2)$ that are $L^2$-orthogonal to the constant functions, has trivial kernel.
\end{lemma}
\begin{proof}
On $H_2$,
$$\widetilde{L}_{\omega_0}(u,R)=\dfrac{1}{2}\nabla S(\omega_0).\nabla u-\mathcal{D}^*\mathcal{D}u-R=-\mathcal{D}^*\mathcal{D}u-R.$$
Let $u\in \rho^\delta\mathcal{C}^{4,\alpha}_{\frac{g_X}{\rho^2}}(H_2)_0$ and $R\in \mathbb{R}$ such that $\widetilde{L}_{\omega_\varepsilon}(u,R)=0$, so $\mathcal{D}^*\mathcal{D}u+R=0$.
Both $\mathcal{D}^*\mathcal{D}u$ and $u$ belongs to $\rho^{\delta-4}\mathcal{C}^{0,\alpha}_{\frac{g_{X}}{\rho^2}}(H_2)$ and $\delta-4>-2k$, so they are integrable on $X$. This ensures that for a test function $\varphi\in \mathcal{C}^\infty(X)$,
$$\displaystyle\int_{X}\varphi \mathcal{D}^*\mathcal{D}u \mathrm{d}\omega_{X}=\displaystyle\int _{X} (\mathcal{D}^*\mathcal{D} \varphi)u \mathrm{d}\omega_{X},$$
that is, $\mathcal{D}^*\mathcal{D}u+R=0$ in the sense of distribution on $X$. Elliptic regularity implies that $u$ is a smooth function in the orbifold sense on $X$. Hence, on $X$,
\begin{align*}
\|\mathcal{D}^*\mathcal{D}u\|^2_{L^2}&=\braket{\mathcal{D}^*\mathcal{D}u,\mathcal{D}^*\mathcal{D}u}_{L^2}\\
&=\braket{-R,\mathcal{D}^*\mathcal{D}u}_{L^2}\\&=\braket{-\mathcal{D}R,\mathcal{D}u}_{L^2}\\&=\braket{0,\mathcal{D}u}_{L^2}=0,
\end{align*}
so $\mathcal{D}^*\mathcal{D}u=0$ and so $R=0$. Since the kernel of $\mathcal{D}^*\mathcal{D}$ on $X$ consists of constant functions (we assume $X$ has no non trivial holomorphic vector fields), this shows that $u=0$. 
\end{proof}

\begin{proposition}\label{inv}
Assume $X$ has no non trivial holomorphic vector fields. For $4-2k<\delta<0$ and $\varepsilon>0$ small enough, the operator \eqref{52}
$$\widetilde{L}_\varepsilon:\rho^\delta C^{4,\alpha}_{\frac{\widehat{g}_\varepsilon}{\rho^2}}(\widehat{X})_0\times \bR\to \rho^{\delta-4}C^{0,\alpha}_{\frac{\widehat{g}_\varepsilon}{\rho^2}}(\widehat{X}),$$
where $\rho^\delta C^{4,\alpha}_{\frac{\widehat{g}_\varepsilon}{\rho^2}}(\widehat{X})_0$ is the subspace of functions in $\rho^\delta C^{4,\alpha}_{\frac{\widehat{g}_\varepsilon}{\rho^2}}(\widehat{X})$ that are $L^2$-orthogonal to the constant functions, is invertible and its inverse $P_\varepsilon:=\widetilde{L}_\varepsilon^{-1}$ is bounded by a constant independent of $\varepsilon$. 
\end{proposition}
\begin{proof}
We will closely follow the proof proposition 9 in \cite{seyyedali2020extremal}. By the Schauder estimates, there is a uniform constant $C$ independent of $\varepsilon$, such that
\begin{equation}\label{aval}
\|u\|_{\rho^\delta C^{4,\alpha}_{\frac{\widehat{g}_\varepsilon}{\rho^2}}}+|R|\leq C(\|u\|_{\rho^\delta C^{0}_{\frac{\widehat{g}_\varepsilon}{\rho^2}}}+|R|+\|\widetilde{L}_\varepsilon(u,R)\|_{\rho^{\delta-4}C^{0,\alpha}_{\frac{\widehat{g}_\varepsilon}{\rho^2}}}).
\end{equation}
We want to show that there exists a uniform constant $\widetilde{C}$ such that 
\begin{equation}\label{dovom}
\|u\|_{\rho^\delta C^{4,\alpha}_{\frac{\widehat{g}_\varepsilon}{\rho^2}}}+|R|\leq \widetilde{C}\|\widetilde{L}_\varepsilon(u,R)\|_{\rho^{\delta-4}C^{0,\alpha}_{\frac{\widehat{g}_\varepsilon}{\rho^2}}}.
\end{equation}
To prove \eqref{dovom},  for a contradiction suppose that there is a sequence $\varepsilon_i\to 0$ with $u_i$ and $R_i$ such that $\|u_i\|_{\rho^\delta C^{4,\alpha}_{\frac{\widehat{g}_{\varepsilon_i}}{\rho^2}}} + |R_i| > i \|\widetilde{L}_{\varepsilon_i}(u_i, R_i)\|_{\rho^{\delta-4} C^{0,\alpha}_{\frac{\widehat{g}_{\varepsilon_i}}{\rho^2}}} $ for each $i$. In particular, by \eqref{aval},
$$i \|\widetilde{L}_{\varepsilon_i}(u_i, R_i)\|_{\rho^{\delta-4} C^{0,\alpha}_{\frac{\widehat{g}_{\varepsilon_i}}{\rho^2}}} < \|u_i\|_{\rho^\delta C^{4,\alpha}_{\frac{\widehat{g}_{\varepsilon_i}}{\rho^2}}} + |R_i| \leq C \left( \|u_i\|_{\rho^\delta C^{0}_{\frac{\widehat{g}_{\varepsilon_i}}{\rho^2}}} + |R_i| + \|\widetilde{L}_{\varepsilon_i}(u_i, R_i)\|_{\rho^{\delta-4} C^{0,\alpha}_{\frac{\widehat{g}_{\varepsilon_i}}{\rho^2}}} \right) $$
so that
$$\left( \frac{i}{C} - 1 \right) \left\| \widetilde{L}_{\varepsilon_i}(u_i, R_i) \right\|_{\rho^{\delta-4} C^{0,\alpha}_{\frac{\widehat{g}_{\varepsilon_i}}{\rho^2}}} < \left\| u_i \right\|_{\rho^\delta C^{0,\alpha}_{\frac{\widehat{g}_{\varepsilon_i}}{\rho^2}}} + |R_i| .$$
Without loss of generality, by multiplying $u_i$ and $R_i$ by a constant $\lambda_i$, we can suppose that  $\|u_i\|_{\rho^\delta C^{0}_{\frac{g_{\varepsilon_i}}{\rho^2}}}+|R_i|=1$, so $\|\widetilde{L}_\varepsilon(u_i,R_i)\|_{\rho^{\delta-4}C^{0,\alpha}_{\frac{g_{\varepsilon_i}}{\rho^2}}}<\dfrac{1}{\frac{i}{C}-1}$. This shows that $\|\widetilde{L}_\varepsilon(u_i,R_i)\|_{\rho^{\delta-4}C^{0,\alpha}_{\frac{g_\varepsilon}{\rho^2}}}\to 0$ as $i\to \infty$. Moreover, the Schauder estimates \eqref{aval} show that $u_i:X\setminus Y\to\mathbb{R}$ is uniformly bounded and uniformly equicontinuous in $\rho^{\delta-4}C^{4,\alpha}_{\frac{g_\varepsilon}{\rho^2}}$, so by the Arzelà-Ascoli theorem there exists a convergent subsequence $\{u_{i_j}\}$ converging a function $u\in \rho^\delta C^{4,\alpha}_{\frac{g_X}{\rho^2}}(X\setminus Y)$ with convergence in $\rho^\delta C^{4}_{\frac{g_X}{\rho^2}}(K)$ on each compact subset $K\subset X\setminus Y$. Also, $\{R_i\}$ is a bounded numerical sequence, so by the Bolzano-Weierstrass theorem, we can assume that the subsequence $\{R_{i_j}\}$ converges to some $R$, so that
  $$\widetilde{L}_{\omega_X}(u,R)=\displaystyle\lim_{j\to \infty}\widetilde{L}_{\omega_{\varepsilon_{i_j}}}(u_{i_j},R_{i_j})=0.$$ 
Lemma \ref{red} implies that $u=0$ and $R=0$, so $R_{i_j}\to 0$. Since $\nabla S(\omega_\varepsilon)\to \nabla S(\omega_X)=0$ when $\varepsilon\to 0$ and
$$\|\widetilde{L}_\varepsilon(u_{i_j}, R_{i_j})\|_{\rho^{\delta-4} C^{0,\alpha}_{\frac{g_{\varepsilon_{i_j}}}{\rho^2}}} = \|L_{\omega_{\varepsilon_{i_j}}}(u_{i_j}) - \frac{1}{2} \nabla S(\omega_{\varepsilon_{i_j}}) \cdot \nabla u_{i_j} - R_{i_j}\|_{\rho^{\delta-4} C^{0,\alpha}_{\frac{g_{\varepsilon_{i_j}}}{\rho^2}}} ,$$
we see that
$$\displaystyle\lim_{j\to \infty}\|L_{\omega_{\varepsilon_{i_j}}}(u_{i_j})\|_{\rho^{\delta-4}C^{0,\alpha}_{\frac{g_{\varepsilon_{i_j}}}{\rho^2}}}=0.$$
On the other hand, $\displaystyle\sup_{x\in \widehat{X}} |\dfrac{u_{i_j}(x)}{\rho^\delta(x)}|\leq \|\dfrac{u_{i_j}}{\rho^\delta}\|_{C^{0}_{\frac{g_{\varepsilon_{i_j}}}{\rho^2}}}\leq 1$, so $\dfrac{u_{i_j}(x)}{\rho^\delta(x)}$ is bounded by 1 on the compact manifold $\widehat{X}$ for each $j$. In particular, it achieves a maximum, say at $q_j\in \widehat{X}$. In particular $\dfrac{u_{i_j}(q_j)}{\rho^\delta(q_j)}\leq 1$ and $\dfrac{u_{i_j}(q_j)}{\rho^\delta(q_j)}\to 1$ as $j\to \infty$, since $R_{i_j}\to 0$ as $j\to \infty$. On any compact subset of $X\setminus Y$, $u_{i_j}(q_j)\to 0$, so we must have $\rho(q_{j})\to 0$. Taking a subsequence if needed, we can therefore assume that $q_j\to q\in \widehat{H}_1$. There are two possibilities, either $q\in \widehat{H}_1\setminus (\widehat{H}_1\cap \widehat{H}_2)$, or else $q\in \widehat{H}_1\cap \widehat{H}_2$. If $q\in \widehat{H}_1\setminus (\widehat{H}_1\cap \widehat{H}_2)$, then $\dfrac{u_{i_j}(x)}{\rho^\delta(q_j)}$ converges to a function that satisfies Lemma \ref{new11}, so $\dfrac{u_{i_j}(x)}{\rho^\delta(q_j)}\to 0$, in contradiction $\dfrac{u_{i_j}(q_j)}{\rho^\delta(q_j)}\to 1$. Otherwise, if $q\in \widehat{H}_1\cap \widehat{H}_2$, then $\dfrac{u_{i_j}(x)}{\rho^\delta(q_j)}$ converges to a function that satisfies Lemma  \ref{yellow}, $\dfrac{u_{i_j}(q_j)}{\rho^\delta(q_j)}\to 0$, again yielding to a contradiction with $\dfrac{u_{i_j}(q_j)}{\rho^\delta(q_j)}\to 1$. Consequently the inequality \eqref{dovom} holds. The inequality \eqref{dovom} shows that the kernel of $\widetilde{L}_\varepsilon$ is trivial and since it has index zero, it is invertible. Also $\widetilde{L}_\varepsilon$ is bijective bounded linear operator from one Banach space to another, so $\widetilde{L}_\varepsilon$ has bounded inverse.
\end{proof}

\section{Nonlinear analysis and the main Theorem}
\label{sec: Nonlinear analysis and the main Theorem}

If a perturbed metric $\widetilde{\omega}_\varepsilon=\widehat{\omega}_\varepsilon+\sqrt{-1}\dd u$ has a constant scalar curvature $S(\widetilde{\omega}_\varepsilon)=R_\varepsilon$ for all $\varepsilon$, then $R_\varepsilon$ can be determined from the \kahler class $[\widehat{\omega}_\varepsilon]$ of $\widehat{\omega}_\varepsilon$. We try to find an approximate for $R$.
\begin{proposition}\label{chern}
Let $X$ be a compact complex orbifold with singularities of type $\mathcal{I}$ along a subset $Y$ with codimension $k$ greater than 2, i.e., the normal bundle of $Y$ in $X$ has fibers of the form $\bC^k/\Gamma_{(-w_0,w)}$
for some weight $w$.
Then the first Chern class of the $(-w_0,w)-$weighted blow-up $\widehat{X}$ of $X$ along $Y$ is
$$c_1(\widehat{X})=\pi^*c_1(X)-(\dfrac{1}{w_0}\displaystyle \sum_{i=1}^kw_i-1)[E]|_E,$$
where $[E]$ is the Poincaré class of the exceptional divisor.
\end{proposition}
\begin{proof}
Away from $Y$, we have the canonical identification of canonical bundles $K_{\widehat{X}}=K_X$. Let $N_{\widehat{X}}(E)$ be the normal bundle of $E$ in $\widehat{X}$. By the adjunction formula 
$$K_{\widehat{X}}|_E=K_E\otimes N_{\widehat{X}}^*(E)=K_E\otimes [-E]|_E.$$
On the other hand, if $V=\ker (\pi|_{E_*})$ is the vertical tangent bundle of $\pi|_{E}:E\to Y$, then
\begin{equation}\label{chern}
T^*E\cong V^*\oplus \pi^*TY.
\end{equation}
Since $\pi: E\to Y$ is a projective bundle with fiber $\mathbb{CP}_w^{k-1}$, $E=\mathbb{P}_w(W)$ is the weighted projectivization of some vector bundle $W\to Y$ of rank $k$ such that $N_X(Y)=W/\Gamma_{(-w_0,w)}$. Now the canonical bundle of the weighted projective space is given by  
$$ K_{\mathbb{CP}_w^{k-1}} = \mathcal{O}_{\mathbb{CP}_w^{k-1}}(-\displaystyle\sum_{i=1}^{k} w_i), $$  
see, for instance, \cite{dolgachev2006weighted} or 6.7.2 in \cite{joyce2000compact}. Keeping in mind the decomposition of $W$ in the proof of Theorem \ref{main}, and using equation \eqref{chern}, this means that
\begin{align*}
K_E&=\wedge ^{k-1}(V^*)\otimes \pi^*(K_Y)\\
&= \pi^*(K_Y)\otimes \pi^*(\wedge ^{k}(W^*))\otimes \mathcal{O}_{E\diagup Y}(-\displaystyle\sum_{i=1}^{k} w_i)\\
&=\pi^*(K_X)|_E\otimes  \mathcal{O}_{E\diagup Y}(-\displaystyle\sum_{i=1}^{k} w_i).
\end{align*}

Finally, we have that
\begin{align*}
K_{\widehat{X}}|_E&=K_E\otimes N_{\widehat{X}}^*(E)\\
&=\pi^*(K_X)|_E \otimes \mathcal{O}_{E\diagup Y}(-\displaystyle\sum_{i=1}^{k} w_i)\otimes \mathcal{O}_{E\diagup Y}(w_0)\\
&=\pi^*(K_X)|_E \otimes \mathcal{O}_{E\diagup Y}(-\displaystyle\sum_{i=1}^{k} w_i+w_0)=\pi^*(K_X)|_E \otimes (N_{\widehat{X}}(E))^{\displaystyle\frac{1}{w_0}\sum_{i=1}^{k} w_i-1}.
\end{align*}
Globally on $\widehat{X}$, $N_{\widehat{X}}(E)=[E]$ is trivial on $\widehat{X}\setminus E$, hence
$$K_{\widehat{X}}=\pi^*(K_X)\otimes (N_{\widehat{X}}(E))^{\displaystyle\frac{1}{w_0}\sum_{i=1}^{k} w_i-1}.$$
Since $c_1(X)=-c_1(K_X)$ and $c_1(\widehat{X})=-c_1(K_{\widehat{X}})$, we finally obtain
$$c_1(\widehat{X})=\pi^*c_1(X)-(\dfrac{1}{w_0}\displaystyle\sum_{i=1}^{k} w_i-1)[E]|_E.$$
\end{proof}
\begin{proposition}\label{lambda}
Let $\widehat{\omega}_\varepsilon$ be the family of \kahler metrics of Theorem \ref{main}. Assume the singularity of type $\mathcal{I}$ is $\bC^k/\Gamma_{(-w_0,w)}$. If there is a constant scalar curvature metric $\widetilde{\omega}_\varepsilon$ in the \kahler class  $[\widehat{\omega}_\varepsilon]$, the scalar curvature of $\widetilde{\omega}_\varepsilon$ can be represented by
$$S(\widetilde{\omega}_\varepsilon)=S(\omega_X)+\lambda\varepsilon^{2k-2}+R_\varepsilon,$$
where $|R_\varepsilon|\leq c\varepsilon^{2k}$ for some constant $c>0$ independent of $\varepsilon$ and $\lambda$ is a topological constant depending on the \kahler class $[\widehat{\omega}_\varepsilon]$ and first Chern class of $\widehat{X}$.
\end{proposition}
\begin{proof}
In this sense,
\begin{align*}
S(\widetilde{\omega}_\varepsilon) \Vol_{\widetilde{\omega}_\varepsilon}(\widehat{X})&=\displaystyle \int_{\widehat{X}} S(\widetilde{\omega}_\varepsilon) \widetilde{\omega}_\varepsilon^n\\
&=\displaystyle \int_{\widehat{X}} (S(\widetilde{\omega}_\varepsilon) \widetilde{\omega}_\varepsilon) \wedge \widetilde{\omega}_\varepsilon^{n-1}\\
&=\displaystyle \int_{\widehat{X}} 2n\rho \wedge\widetilde{\omega}_\varepsilon^{n-1}\\
&=\displaystyle \int_{\widehat{X}} 2n(2\pi c_1(\widehat{X}) )\wedge\widetilde{\omega}_\varepsilon^{n-1}\\
&=4n\pi\displaystyle \int_{\widehat{X}}  c_1(\widehat{X}) \cup[\widetilde{\omega}_\varepsilon]^{n-1}\\
&=4n\pi\displaystyle \int_{\widehat{X}}  c_1(\widehat{X}) \cup [\widehat{\omega}_\varepsilon]^{n-1},
\end{align*}
so that
$$S(\tilde{\omega}_\varepsilon) = \dfrac{4n\pi \displaystyle\int_{\widehat{X}} c_1(\widehat{X}) \cup [\widehat{\omega}_\varepsilon]^{n-1}}{\Vol_{\widehat{\omega}_\varepsilon}(\widehat{X})} = \dfrac{4\pi n}{\displaystyle\int_{\widehat{X}} [\widehat{\omega}_\varepsilon]^n} \int_{\widehat{X}} c_1(\widehat{X}) \cup [\widehat{\omega}_\varepsilon]^{n-1} .$$
We set $C_\varepsilon=\dfrac{4\pi n}{\displaystyle \int_{\widehat{X}}[\widehat{\omega}_\varepsilon]^n}$ and $C=\dfrac{4\pi n}{\displaystyle \int_{\widehat{X}}[\omega_X]^n}$ as well, then by Remark \ref{4.3},
\begin{align*}
\displaystyle \int_{\widehat{X}}[\widehat{\omega}_\varepsilon]^{n}&=\displaystyle \int_{X}([\omega_X]-\varepsilon^2[E])^n\\
&=\displaystyle \int_{X}[\omega_X]^n+\displaystyle \int_{X}\displaystyle \sum_{i=1}^{n}{n\choose i}(-\varepsilon^2[E])^i[\omega_X]^{n-i}\\
&=\displaystyle \int_{X}[\omega_X]^n+\displaystyle \sum_{i=1}^{n}{n\choose i}(-\varepsilon^2)^i\displaystyle \int_{E}[E]^{i-1}[\omega_X]^{n-i}\\
&=\displaystyle \int_{X}[\omega_X]^n+\displaystyle \sum_{i=k}^{n}{n\choose i}(-\varepsilon^2)^i\displaystyle \int_{E}[E]^{i-1}[\omega_X]^{n-i}\\
&=\displaystyle \int_{X}[\omega_X]^n+O(\varepsilon^{2k}),
\end{align*}
since we must have that $i-1\geq k-1$ for the second integral to be non-zero. Indeed, the only vertical contribution of $[E]^{i}[\omega_X]^{n-i}$ with respect to the fiber bundle $\pi:E\to Y$ is coming from $[E]^i$.\\
Hence
$$C_\varepsilon=\dfrac{4\pi n}{\displaystyle \int_{\widehat{X}}[\widehat{\omega}_\varepsilon]^n}=\dfrac{4\pi n}{\displaystyle \int_{\widehat{X}}[\omega_X]^n+O(\varepsilon^{2k})}=\dfrac{4\pi n}{\displaystyle \int_{\widehat{X}}[\omega_X]^n}+O(\varepsilon^{2k})=C+O(\varepsilon^{2k}).$$
From the Lemma \ref{chern}, $c_1(\widehat{X})=\pi^* c_1(X)-(\displaystyle\dfrac{1}{w_0}\sum_{i=1}^{k} w_i-1)[E]$ and by the Remark \ref{4.3} again, $[\widehat{\omega}_\varepsilon]=[\omega_X]-\varepsilon^2[E]$, so

\begin{align*}
S(\widetilde{\omega}_\varepsilon)&=C_\varepsilon\displaystyle \int _{\widehat{X}}(\pi^* c_1(X)-(\displaystyle\frac{1}{w_0}\sum_{i=1}^{k} w_i-1)[E])\cup ([\omega_X]-\varepsilon^2[E])^{n-1}\\
&=C_\varepsilon\displaystyle \int _X  c_1(X) \cup [\omega_X]^{n-1}\\+&C_\varepsilon\displaystyle \sum_{i=1}^{n-1}{n-1 \choose i}(-1)^i\varepsilon^{2i} \displaystyle \int _{\widehat{X}} \pi^*c_1(X)\cup [\omega_X]^{n-1-i}\cup [E]^i\\
-&C_\varepsilon\displaystyle \sum_{i=0}^{n-1}(\displaystyle\dfrac{1}{w_0}\sum_{i=1}^{k} w_i-1){n-1 \choose i}(-1)^i\varepsilon^{2i}\displaystyle \int _X [\omega_X]^{n-1-i}\cup [E]^{i+1}\\
&=S(\omega_X)+C_\varepsilon\displaystyle \sum_{i=1}^{n-1}{n-1 \choose i}(-1)^i\varepsilon^{2i} \displaystyle \int _E \pi^*c_1(X)\cup [\omega_X]^{n-1-i}\cup [E]^{i-1}\\
+&C_\varepsilon\displaystyle \sum_{i=0}^{n-1}(\displaystyle\dfrac{1}{w_0}\sum_{i=1}^{k} w_i-1){n-1 \choose i}(-1)^i\varepsilon^{2i}\displaystyle \int _E [\omega_X]^{n-1-i}\cup [E]^{i}+O(\varepsilon^{2k}).
\end{align*}
Since $\pi^*c_1(X)$ and $\pi^*[\omega_X]^{n-i-1}$ are basic with respect to the bundle map $\pi:E\to Y$, in the first sum, we must have that $i-1\geq k-1$ for the integral to be non-zero, while in the second sum, $i\geq k-1$ for the integral to be non-zero. Hence we see that 
$S(\widetilde{\omega}_\varepsilon)=S(\omega_X)+\lambda\varepsilon^{2k-2}+R_\varepsilon$ with constant coefficient $\lambda=C(\displaystyle\dfrac{1}{w_0}\sum_{i=1}^{k} w_i-1)\displaystyle{n-1 \choose k-1}(-1)^{k-1}\displaystyle \int _E [\omega_X]^{n-k}\cup [E]^{k-1}$ and $R_\varepsilon$ as claimed.

%

\end{proof}

Now we can find a better approximation of $u$ by looking at solution of $\mathcal{D}^*\mathcal{D}\Gamma=\lambda$ on $X\setminus Y$ for $\lambda$ defined in Proposition \ref{lambda}. 
To do so, we will consider the function
$$\Lambda(x) =\displaystyle\int_Y G(x,y)\mathrm{d} y,$$
where $G(x,y)$ is the Green function of the Lichnerowicz operator $\mathcal{D}^*\mathcal{D}$. The operator $G$ is associated to the Green function of $\mathcal{D}^*\mathcal{D}$ and by definition
$$\mathcal{D}^*\mathcal{D}G=\mathbf{Id}-\mathbf{P},$$
where $\mathbf{P}$ is the projection on constant functions, i.e, $\mathbf{P}(f)=\displaystyle \int_X \dfrac{f(y)}{\Vol(X)}$ . In terms of Schwartz kernels,
$$\mathcal{D}_x^*\mathcal{D}_xG(x,y)=\delta(x-y)-\dfrac{1}{\Vol(X)}, \quad \text{on X}.$$
In the distributional sense, we thus have that
$$\mathcal{D}_x^*\mathcal{D}_x\Lambda=\delta_Y-\dfrac{\Vol(Y)}{\Vol(X)}, \quad \text{on X},$$
where $\delta_Y$ is the current of integration along $Y$. Let us consider the function
$$\Gamma(x)=-\dfrac{\Vol(X)}{\Vol(Y)}\lambda\Lambda(x).$$
To determine the asymptotic behaviour of $\Lambda(x)$ and $\Gamma(x)$ near $Y$, notice that in local coordinates near $Y$,
\begin{align*}
G(x,y)&=\mathcal{F}^{-1}(\sigma_4((\mathcal{D}^*\mathcal{D})^{-1}))(x-y)+O(|x-y|^{5-2n})\\
&=\mathcal{F}^{-1}(|\xi|^{-4})(x-y)+O(|x-y|^{5-2n})\\
&=\dfrac{1}{(2\pi)^n}\displaystyle\int_{\mathbb{R}^n}e^{i(x-y). \xi}|\xi|^{-4}\mathrm{d}\xi+O(|x-y|^{5-2n})\\
&=\dfrac{c}{|x-y|^{2n-4}}+O(|x-y|^{5-2n}),
\end{align*}
where $c$ is some positive constant and $\mathcal{F}^{-1}$ denotes the inverse Fourier transform on $\mathbb{R}^n$.
Let $x_s$ be the projection of $x$ on $Y$ and $d=|x-x_s|$ be the distance to $Y$ in local coordinates. Then $|x-y|^2=d^2+|y-x_s|^2$, so setting $y_s=y-x_s$, we get that
$$\displaystyle\int_Y\dfrac{\mathrm{d}y_{1}\ldots \mathrm{d}y_{2n-2k}}{|x-y|^{2n-4}}=\displaystyle\int_Y\dfrac{\mathrm{d}y_{s_{1}}\ldots \mathrm{d}y_{s_{2n-2k}}}{(d^2+|y_s|^2)^{n-2}}.$$
In polar coordinates in a ball of radius 1, this yields 
\begin{align*}
\displaystyle\int_{B_1(0)}\dfrac{\mathrm{d}y_{s_{1}}\ldots \mathrm{d}y_{s_{2n-2k}}}{(d^2+|y_s|^2)^{n-2}} &=\displaystyle\int_{\mathbb{S}^{2n-2k-1}}\int_0^1\dfrac{r^{2n-2k-1}}{(d^2+r^2)^{n-2}}\mathrm{d}r\mathrm{d}\omega\\&=\Vol(\mathbb{S}^{2n-2k-1})\int_0^1\dfrac{r^{2n-2k-1}}{(d^2+r^2)^{n-2}}\mathrm{d}r\\&=\dfrac{2\pi^{n-k}}{\Gamma(n-k)}\int_0^1\dfrac{r^{2n-2k-1}}{(d^2+r^2)^{n-2}}\mathrm{d}r.
\end{align*}
By substituting $r=Rd$, we get
\begin{align*}
\displaystyle\int_0^1\dfrac{r^{2n-2k-1}}{(d^2+r^2)^{n-2}}\mathrm{d}r&=\displaystyle\int_0^{\frac{1}{d}}\dfrac{(Rd)^{2n-2k-1}d}{(d^2+(Rd)^2)^{n-2}}\mathrm{d}R=\dfrac{1}{d^{2k-4}}\displaystyle\int_0^{\frac{1}{d}}\dfrac{R^{2n-2k-1}}{(1+R^2)^{n-2}}\mathrm{d}R.
\end{align*}
Note that the integral $\displaystyle\int_0^{\frac{1}{d}}\dfrac{R^{2n-2k-1}}{(1+R^2)^{n-2}}\mathrm{d}R$ converges when $d\to 0$ by the Riemann criterion because\linebreak $2(n-2)-(2n-2k-1)=2k-3>1$ since $k>2$. Moreover the integral can be computed explicitly
\begin{align*}
\displaystyle\int_0^{+\infty}\dfrac{R^{2n-2k-1}}{(1+R^2)^{n-2}}\mathrm{d}R&=
\displaystyle\int_0^{\frac{\pi}{2}}\sin^{2n-2k-1}\theta\cos^{2k-5}\theta \mathrm{d}\theta, \text{ posing } R=\tan\theta\\
&=\displaystyle\int_0^{\frac{\pi}{2}}(\sin^2\theta)^{n-k-1}(\cos^{2}\theta)^{k-2}\sin \theta\cos\theta \mathrm{d}\theta\\
&=\dfrac{1}{2}\displaystyle\int_0^1t^{n-k-1}(1-t)^{k-3}\mathrm{d}t=\dfrac{\Gamma(n-k)\Gamma(k-2)}{2\Gamma(n-2)}.
\end{align*}
Hence,
\begin{equation}\label{lambda}
\Lambda(x) \approx c(\dfrac{2\pi^{n-k}}{\Gamma(n-k)})(\dfrac{\Gamma(n-k)\Gamma(k-2)}{2\Gamma(n-2)})(\dfrac{1}{d})^{2k-4}+O((\dfrac{1}{d})^{2k-5})=c'(\dfrac{1}{d})^{2k-4}+O((\dfrac{1}{d})^{2k-5}),
\end{equation}
for $c'$ another constant.

Recalling that $r_\varepsilon=\varepsilon^{\frac{2k}{2k+1}}$, we can use the function $\Gamma$ to define a new metric 
$$\widetilde{\omega}_\varepsilon:=\widehat{\omega}_\varepsilon+\sqrt{-1}\partial\bar{\partial}(\varepsilon^{2k-2}\gamma_2(\dfrac{d}{r_\varepsilon})\Gamma),$$
where $\gamma_2:\mathbb{R}\to\mathbb{R}$ is a cutoff function such that $\gamma_2(t)=0$ for $t<1$ and $\gamma_2(t)=1$ for $t>2$. On the support of $\gamma_2(\dfrac{d}{r_\varepsilon})$, $\rho\leq Cd$ for some constant and $d\geq r_\varepsilon$, so $\rho^2(\dfrac{1}{d})^{2k-4}\leq C^2(\dfrac{1}{d})^{2k-6}\leq C^2r_\varepsilon^{6-2k}$. Hence if we denote $\Omega=\{x\in \widehat{X}:r_\varepsilon\leq d_{}(x)\}$, then using \eqref{lambda},
\begin{align*}
\|\varepsilon^{2k-2}\gamma_2(\dfrac{d}{r_\varepsilon})\Gamma\|_{\rho^{2}C^{l,\alpha}_{\frac{\widehat{g}_\varepsilon}{\rho^2}}(\widehat{X})}&\leq c\varepsilon^{2k-2}\|\Gamma\|_{\rho^{2}C^{l,\alpha}_{\frac{\widehat{g}_\varepsilon}{\rho^2}}(\Omega)}\\
&\leq c'\varepsilon^{2k-2}r_\varepsilon^{6-2k}\\
&\leq c'\varepsilon^{\frac{5(2k)-2}{2k+1}}\\
&\leq c'\varepsilon^{\frac{4(2k)+2k-2}{2k+1}}\leq c'\varepsilon^{4}, \text{since } k>2,
\end{align*}
which tends to zero as $\varepsilon\to 0$. So $\widetilde{\omega}_\varepsilon=\widehat{\omega}_\varepsilon+\sqrt{-1}\partial\bar{\partial}(\varepsilon^{2k-2}\gamma_2(\dfrac{d}{r_\varepsilon})\Gamma)$ is a small perturbation of $\widehat{\omega}_\varepsilon$ and so $L_{\widetilde{\omega}_\varepsilon}$ is a small perturbation of $L_{\widehat{\omega}_\varepsilon}$ for sufficiently small $\varepsilon$.
Now we would like to solve the nonlinear equation
\begin{equation}\label{nonlinear-}
S(\widehat{\omega}_\varepsilon+\sqrt{-1}\dd u)=R, 
\end{equation}
with $u$ and $R$ of the form
\begin{align*}
u&=\varepsilon^{2k-2}\gamma_2(\dfrac{d}{r_\varepsilon})\Gamma+v,\\
R&=S(\omega_X)+\lambda\varepsilon^{2k-2}+R_\varepsilon.
\end{align*}
So the goal is to find $v$. By replacing $u$ in the left side of \eqref{nonlinear-}, we get from \eqref{32} page \pageref{32} that
\begin{align*}
S(\widehat{\omega}_\varepsilon+\sqrt{-1}\dd u)
&=S(\widehat{\omega}_\varepsilon)+L_{\widehat{\omega}_\varepsilon}(u)+Q_{\widehat{\omega}_\varepsilon}(\nabla^2 u)\\
&=S(\widehat{\omega}_\varepsilon)+\varepsilon^{2k-2}L_{\widehat{\omega}_\varepsilon}(\gamma_2(\dfrac{d}{r_\varepsilon})\Gamma)+L_{\widehat{\omega}_\varepsilon}(v)+Q_{\widehat{\omega}_\varepsilon}(\nabla^2 u),
\end{align*}
so solving \eqref{nonlinear-} means to solve
$$L_{\widehat{\omega}_\varepsilon}(v)-R_\varepsilon=S(\omega_X)-S(\widehat{\omega}_\varepsilon)+\lambda\varepsilon^{2k-2}-\varepsilon^{2k-2}L_{\widehat{\omega}_\varepsilon}(\gamma_2(\dfrac{d}{r_\varepsilon})\Gamma)-Q_{\widehat{\omega}_\varepsilon}(\nabla^2 u).$$
Now we define $F_\varepsilon:\rho^\delta C^{4,\alpha}_{\frac{\widehat{g}_\varepsilon}{\rho^2}}(\widehat{X})_0\times \bR\to \rho^{\delta-4}C^{0,\alpha}_{\frac{\widehat{g}_\varepsilon}{\rho^2}}(\widehat{X})\times \bR$ by 
$$F_\varepsilon(v,R):=S(\omega_X)-S(\widehat{\omega}_\varepsilon)+\lambda\varepsilon^{2k-2}-\varepsilon^{2k-2}L_{\widehat{\omega}_\varepsilon}(\gamma_2(\dfrac{d}{r_\varepsilon})\Gamma)-Q_{\widehat{\omega}_\varepsilon}(\nabla^2 u).$$


\begin{remark}
Note that the function $F_\varepsilon$ does not depends on $R$, but to use Banach fixed point theorem, we consider $F_\varepsilon$ as a function of $v$ and $R$.
\end{remark}

\begin{lemma}\label{bil}
Suppose $\delta>0$. Then there exists constants $c_0,c>0$ such that if $\|\varphi\|_{\rho^{\delta}C^{4,\alpha}_{\frac{g_\varepsilon}{\rho^2}}(\widehat{X})}<c_0$, then
$$\|Q_{\widehat{\omega}_\varepsilon}(\nabla^2\varphi) \|_{\rho^{\delta-4}C^{0,\alpha}_{\frac{\widehat{g}_\varepsilon}{\rho^2}}(\widehat{X})} \leq c\|\varphi\|_{\rho^{\delta}C^{4,\alpha}_{\frac{\widehat{g}_\varepsilon}{\rho^2}}(\widehat{X})}\|\varphi\|_{\rho^{2}C^{4,\alpha}_{\frac{\widehat{g}_\varepsilon}{\rho^2}}(\widehat{X})}.$$
\end{lemma}
\begin{proof}
From Lemma \ref{linears} we have the following decomposition with finite sums:
\begin{align*}
Q_{\widehat{\omega}_\varepsilon}(\nabla^2\varphi)&=\displaystyle\sum_{q}B_{q,4,2}(\nabla^4\varphi,\nabla^2\varphi)C_{q,4,2}(\nabla^2\varphi)\\
&+\displaystyle\sum_{q}B_{q,3,3}(\nabla^3\varphi,\nabla^3\varphi)C_{q,3,3}(\nabla^2\varphi)\\
&+|z|\displaystyle\sum_{q}B_{q,3,2}(\nabla^3\varphi,\nabla^2\varphi)C_{q,3,2}(\nabla^2\varphi)\\
&+\displaystyle\sum_{q}B_{q,2,2}(\nabla^2\varphi,\nabla^2\varphi)C_{q,2,2}(\nabla^2\varphi),
\end{align*}
where $B^i$s are bilinear forms and $C^i$s are smooth functions. Each term of the above decomposition is controlled by the $\|\varphi\|_{\rho^{\delta}C^{4,\alpha}_{\frac{\widehat{g}_\varepsilon}{\rho^2}}(\widehat{X})}\|\varphi\|_{\rho^{2}C^{4,\alpha}_{\frac{\widehat{g}_\varepsilon}{\rho^2}}(\widehat{X})}$, for example
\begin{align*}
\|B_{q,3,3}(\nabla^3\varphi,\nabla^3\varphi) \|_{\rho^{\delta-4}C^{0,\alpha}_{\frac{g_\varepsilon}{\rho^2}}(\widehat{X})} &\leq \|\rho^{4-\delta}B_{q,3,3}(\nabla^3\varphi,\nabla^3\varphi) \|_{C^{0,\alpha}_{\frac{\widehat{g}_\varepsilon}{\rho^2}}(\widehat{X})}\\
&\leq \|B_{q,3,3}\|_{\op}\|\rho^{3-\delta}\nabla^3\varphi\|_{C^{0,\alpha}_{\frac{\widehat{g}_\varepsilon}{\rho^2}}(\widehat{X})}\|\rho\nabla^3\varphi\|_{C^{0,\alpha}_{\frac{\widehat{g}_\varepsilon}{\rho^2}}(\widehat{X})}\\
&= \|B_{q,3,3}\|_{\op}\|\nabla^3\varphi\|_{\rho^{\delta-3}C^{0,\alpha}_{\frac{\widehat{g}_\varepsilon}{\rho^2}}(\widehat{X})}\|\nabla^3\varphi\|_{\rho^{-1}C^{0,\alpha}_{\frac{\widehat{g}_\varepsilon}{\rho^2}}(\widehat{X})}\\
&\leq c\|\varphi\|_{\rho^{\delta}C^{3,\alpha}_{\frac{\widehat{g}_\varepsilon}{\rho^2}}(\widehat{X})}\|\varphi\|_{\rho^{2}C^{3,\alpha}_{\frac{\widehat{g}_\varepsilon}{\rho^2}}(\widehat{X})}\\
&\leq c\|\varphi\|_{\rho^{\delta}C^{4,\alpha}_{\frac{\widehat{g}_\varepsilon}{\rho^2}}(\widehat{X})}\|\varphi\|_{\rho^{2}C^{4,\alpha}_{\frac{\widehat{g}_\varepsilon}{\rho^2}}(\widehat{X})},
\end{align*}
using the fact that the embedding $\rho^\delta C^{k,\alpha}_g\to \rho^{\delta'} C^{k',\alpha'}_g$ is compact for $k'+\alpha'<k+\alpha$ and $\delta'<\delta$ and also that, $\|\nabla^if\|_{\rho^\delta C^{k,\alpha}_g}\leq c\|f\|_{\rho^{\delta+i} C^{k+i,\alpha}_g}$.
\end{proof}

\begin{proposition}\label{forf}
For $4-2k<\delta<0$ very close to $4-2k$, there is a constant $c$ independent of $\varepsilon$
such that
$$\|F_\varepsilon(0,0)\|_{\rho^{\delta-4}C^{0,\alpha}_{\frac{\widehat{g}_\varepsilon}{\rho^2}}(\widehat{X})}\leq cr_{\varepsilon}^{3-\delta}.$$
\end{proposition}

\begin{proof}
We consider four possible regions:
\begin{enumerate}
\item On‌ $\Omega_1=\{x\in \widehat{X}:d(x)<\varepsilon\}$ we have $\gamma_2(\dfrac{d}{r_\varepsilon})=0$ so we get 
\begin{equation} \label{1of3} \begin{aligned} F_\varepsilon(0,0) &= S(\omega_X) - S(\widehat{\omega}_\varepsilon) + \lambda \varepsilon^{2k-2} \\ &\quad - \varepsilon^{2k-2} L_{\widehat{\omega}_\varepsilon} ( \gamma_2( \dfrac{d}{r_\varepsilon} ) \Gamma ) \\ &\quad - Q_{\widehat{\omega}_\varepsilon} ( \nabla^2 ( \varepsilon^{2k-2} \gamma_2( \dfrac{d}{r_\varepsilon} ) \Gamma ) ) \\ &= S(\omega_X) - S(\widehat{\omega}_\varepsilon) + \lambda \varepsilon^{2k-2}. \end{aligned} \end{equation}
 Furthermore, in this region, 
$$\rho=\sqrt{\varepsilon^2+d^2}\leq \sqrt{\varepsilon^2+\varepsilon^2}=\sqrt{2}\varepsilon.$$
By Lemma \ref{curvature of a conformally}, the scalar curvature of the conformally changed metric $\omega'=e^{2f}\omega$ is equal to
$$S(\omega')=e^{-2f}(S(\omega)+2(2n-1)\Delta_\omega f-(2n-1)(2n-2)\|\nabla f\|^2_\omega).$$
Using this formula for $\omega'=\varepsilon^{-2}\widehat{\omega}_\varepsilon$ and $f=-\ln \varepsilon$ constant, we get $S(\varepsilon^{-2}\widehat{\omega}_\varepsilon)=\varepsilon^{2}S(\widehat{\omega}_\varepsilon)$ or $S(\widehat{\omega}_\varepsilon)=\varepsilon^{-2}S(\varepsilon^{-2}\widehat{\omega}_\varepsilon)$. Since $\varepsilon^{-2}\widehat{\omega}_\varepsilon$ tends to a scalar flat $\ALE$ metric in the fibers of $\widehat{\varphi_1}:\widehat{H}_1\to Y$, by Theorem \ref{main} on page \pageref{main}, $S(\widehat{\omega}_\varepsilon)=\varepsilon^{-2}O(\varepsilon)=O(\varepsilon^{-1})$. Hence
‌‌\begin{align*}‌
\|‌F_\varepsilon(0,0)\|_{\rho^{\delta-4}C^{0,\alpha}_{\frac{\widehat{g}_\varepsilon}{\rho^2}}(\widehat{X})} &=\|S(\omega_X)-S(\widehat{\omega}_\varepsilon)+\lambda\varepsilon^{2k-2}\|_{\rho^{\delta-4}C^{0,\alpha}_{\frac{\widehat{g}_\varepsilon}{\rho^2}}(\widehat{X})}\\‌
&=\|\rho^{4-\delta}(S(\omega_X)-S(\widehat{\omega}_\varepsilon)+\lambda\varepsilon^{2k-2})\|_{C^{0,\alpha}_{\frac{\widehat{g}_\varepsilon}{\rho^2}}(\widehat{X})}\\‌
&\leq (\sqrt{2}\varepsilon)^{4-\delta}(\|S(\omega_X)\|_{C^{0,\alpha}_{\frac{\widehat{g}_\varepsilon}{\rho^2}}(\widehat{X})}+\|S(\widehat{\omega}_\varepsilon)\|_{C^{0,\alpha}_{\frac{\widehat{g}_\varepsilon}{\rho^2}}(\widehat{X})}+\|\lambda\varepsilon^{2k-2}\|_{C^{0,\alpha}_{\frac{\widehat{g}_\varepsilon}{\rho^2}}(\widehat{X})})\\‌
&\leq (\sqrt{2}\varepsilon)^{4-\delta}(\|S(\omega_X)\|_{C^{0,\alpha}_{\frac{\widehat{g}_\varepsilon}{\rho^2}}(\widehat{X})}+\|\varepsilon^{-2}S(\varepsilon^{-2}\widehat{\omega}_\varepsilon)\|_{C^{0,\alpha}_{\frac{\widehat{g}_\varepsilon}{\rho^2}}(\widehat{X})}+\|\lambda\varepsilon^{2k-2}\|_{C^{0,\alpha}_{\frac{\widehat{g}_\varepsilon}{\rho^2}}(\widehat{X})})\\
&\leq  (\sqrt{2}\varepsilon)^{4-\delta}(c_1+\varepsilon^{-2}c_2\varepsilon‌+\varepsilon^{2k-2}c_3)\\
&\leq c(\sqrt{2}\varepsilon)^{3-\delta},
\end{align*}‌
where $c_1$, $c_2$, $c_3$ and $c$ are constants independent of $\varepsilon$.

‌‌‌‌\item On‌ $\Omega_2=\{x\in \widehat{X}:\varepsilon<d(x)<r_\varepsilon\}$ we have $\gamma_2(\dfrac{d}{r_\varepsilon})=0$,
$$\varepsilon\leq \rho=\sqrt{\varepsilon^2+d^2}\leq \sqrt{\varepsilon^2+r_\varepsilon^2}\leq \sqrt{r_\varepsilon^2+r_\varepsilon^2}=\sqrt{2}r_\varepsilon,$$
and \eqref{1of3} still holds. The terms $S(\omega_X)$ and $\lambda\varepsilon^{2k-2}$ can be estimated as in (a). By \eqref{pot1} on page \pageref{pot1} and the fact that $\omega_{\widehat{\varphi_1}}$ is scalar flat, we have that
$$S(\varepsilon^{-2}\widehat{\omega}_\varepsilon)=O(\varepsilon(\dfrac{d}{\varepsilon})^{-2k}),$$
$$S(\widehat{\omega}_\varepsilon)=\varepsilon^{-2}S(\varepsilon^{-2}\widehat{\omega}_\varepsilon)=O(\varepsilon^{-1}(\dfrac{d}{\varepsilon})^{-2k})=O(\dfrac{1}{d}).$$
This means that
$$\| S(\widehat{\omega}_\varepsilon)\|_{\rho^{\delta-4}C^{0,\alpha}_{\frac{\widehat{g}_\varepsilon}{\rho^2}}(\widehat{X})}\leq C\| \dfrac{\rho^{\delta-4}}{d}\|_{C^{0,\alpha}_{\frac{\widehat{g}_\varepsilon}{\rho^2}}(\widehat{X})}\leq C(\sqrt{2}r_\varepsilon)^{3-\delta},$$
for some constant $C$.

‌‌‌‌\item On‌ $\Omega_3=\{x\in \widehat{X}:r_\varepsilon<d(x)<2r_\varepsilon\}$ we have
\begin{align*}
F_\varepsilon(0,0)&=S(\omega_X)-S(\widehat{\omega}_\varepsilon)+\lambda\varepsilon^{2k-2}-\varepsilon^{2k-2}L_{\widehat{\omega}_\varepsilon}(\gamma_2(\dfrac{d}{r_\varepsilon})\Gamma)-Q_{\widehat{\omega}_\varepsilon}(\nabla^2(\varepsilon^{2k-2}\gamma_2(\dfrac{d}{r_\varepsilon})\Gamma))\\
&=S(\omega_X)+\lambda\varepsilon^{2k-2}-(S(\widehat{\omega}_\varepsilon)+\varepsilon^{2k-2}L_{\widehat{\omega}_\varepsilon}(\gamma_2(\dfrac{d}{r_\varepsilon})\Gamma)+Q_{\widehat{\omega}_\varepsilon}(\nabla^2(\varepsilon^{2k-2}\gamma_2(\dfrac{d}{r_\varepsilon})\Gamma)))\\
&=S(\omega_X)+\lambda\varepsilon^{2k-2}-S(\widetilde{\omega}_\varepsilon).
\end{align*}
As in the previous case, we have $$\| S(\omega_X) \|_{\rho^{\delta-4} C^{0,\alpha}_{\frac{\widehat{g}_\varepsilon}{\rho^2}}(\widehat{X})} \leq c r_{\varepsilon}^{3-\delta}$$ and $$\| \lambda \varepsilon^{2k-2} \|_{\rho^{\delta-4} C^{0,\alpha}_{\frac{\widehat{g}_\varepsilon}{\rho^2}}(\widehat{X})} \leq c r_{\varepsilon}^{3-\delta}.$$ Thus, we just need to control $$\| S(\widetilde{\omega}_\varepsilon) \|_{\rho^{\delta-4} C^{0,\alpha}_{\frac{\widehat{g}_\varepsilon}{\rho^2}}(\Omega_3)}.$$ To check this, let us write $$\widetilde{\omega}_{\varepsilon} = \omega_{\Euc} + \sqrt{-1} \dd H,$$ where $\omega_{\Euc} = \sqrt{-1} \dd (|z|^2 + |w|^2)$ and 
\begin{align*}
H&=\phi_1(z,w)+A\varepsilon^{2k-2}|z|^{4-2k}(1+\phi_2(z,w))^{4-2k}+\varepsilon^{2k-2}\gamma_2(\dfrac{d}{r_\varepsilon})\widetilde{\Gamma}+O(\varepsilon^{2k-1}|z|^{3-2k})\\
&=A\varepsilon^{2k-2}|z|^{4-2k}+\widetilde{H},
\end{align*}
where $\phi_1$ and $\phi_2$ are smooth functions and $A$ is a constant. Note that 
\begin{equation}\label{p47}
\nabla^2H=O(r_\varepsilon+\varepsilon^{2k-2}r_\varepsilon^{2-2k}+\varepsilon^{2k}r_\varepsilon^{-2k})=O(\varepsilon^{2k-2}r_\varepsilon^{2-2k})\to 0 \text{ as } \varepsilon\to 0.
\end{equation}
Now
\begin{align*}
\| S(\widetilde{\omega}_\varepsilon)\|_{\rho^{\delta-4}C^{0,\alpha}_{\frac{\widehat{g}_\varepsilon}{\rho^2}}(\Omega_3)}&\leq
\| S(\widetilde{\omega}_\varepsilon)-L_{\omega_{\Euc}}(H)\|_{\rho^{\delta-4}C^{0,\alpha}_{\frac{\widehat{g}_\varepsilon}{\rho^2}}(\Omega_3)}+\| L_{\omega_{\Euc}}(H)\|_{\rho^{\delta-4}C^{0,\alpha}_{\frac{\widehat{g}_\varepsilon}{\rho^2}}(\Omega_3)}\\
&=\| Q_{\omega_{\Euc}}(\nabla^2 H)\|_{\rho^{\delta-4}C^{0,\alpha}_{\frac{\widehat{g}_\varepsilon}{\rho^2}}(\Omega_3)}+\| \Delta^2_{\Euc}H\|_{\rho^{\delta-4}C^{0,\alpha}_{\frac{\widehat{g}_\varepsilon}{\rho^2}}(\Omega_3)}, 
\end{align*}
and with the same procedure in the Proposition 13 in \cite{seyyedali2020extremal}, we will show that each term is $O(r_\varepsilon^{3-\delta})$.
 From \eqref{p47} on page \pageref{p47} we get
\begin{align*}
\| Q_{\omega_{\Euc}}(\nabla^2 H)\|_{\rho^{\delta-4}C^{0,\alpha}_{\frac{\widehat{g}_\varepsilon}{\rho^2}}(\Omega_3)}&=\| \displaystyle\sum_{q}B_{q,4,2}(\nabla^4H,\nabla^2H)C_{q,4,2}(\nabla^2H)\\+\displaystyle\sum_{q}B_{q,3,3}(\nabla^3H,\nabla^3H)&C_{q,3,3}(\nabla^2H)\|_{\rho^{\delta-4}C^{0,\alpha}_{\frac{\widehat{g}_\varepsilon}{\rho^2}}(\Omega_3)}\\
&\leq c_1 r_\varepsilon^{4-\delta} \|\nabla^4H\|_{C^{0,\alpha}_{\frac{\widehat{g}_\varepsilon}{\rho^2}}(\Omega_3)}\|\nabla^2H\|_{C^{0,\alpha}_{\frac{\widehat{g}_\varepsilon}{\rho^2}}(\Omega_3)}\\+ c_2 r_\varepsilon^{4-\delta} \|\nabla^3H\|_{C^{0,\alpha}_{\frac{\widehat{g}_\varepsilon}{\rho^2}}(\Omega_3)}&\|\nabla^3H\|_{C^{0,\alpha}_{\frac{\widehat{g}_\varepsilon}{\rho^2}}(\Omega_3)}\\
&\leq c_1r_\varepsilon^{4-\delta}(\varepsilon^{2k-2}r_\varepsilon^{-2k})(\varepsilon^{2k-2}r_\varepsilon^{2-2k})\\
+c_2r_\varepsilon^{4-\delta}(\varepsilon^{2k-2}r_\varepsilon^{1-2k})(&\varepsilon^{2k-2}r_\varepsilon^{1-2k})\\
&\leq c\varepsilon^{4k-4}r_\varepsilon^{6-4k-\delta}\leq cr_\varepsilon^{3-\delta}.
\end{align*}
For the $L_{\omega_{\Euc}}(H)$ note that $L_{\Euc}=-\Delta^2_{\Euc}$ and $\Delta^2_{\Euc}(|Z|^{4-2k})=0$, so
\begin{align*}
\|\Delta^2_{\Euc}H\|_{\rho^{\delta-4}C^{0,\alpha}_{\frac{\widehat{g}_\varepsilon}{\rho^2}}(\Omega_3)}&=
\|\Delta^2_{\Euc}(A\varepsilon^{2k-2}|Z|^{4-2k}+\widetilde{H})\|_{\rho^{\delta-4}C^{0,\alpha}_{\frac{\widehat{g}_\varepsilon}{\rho^2}}(\Omega_3)}\\
&=\|\Delta^2_{\Euc}\widetilde{H}\|_{\rho^{\delta-4}C^{0,\alpha}_{\frac{\widehat{g}_\varepsilon}{\rho^2}}(\Omega_3)}\\
&\leq Cr_\varepsilon^{4-\delta}(1+\varepsilon^{2k-2}r_\varepsilon^{1-2k})\\
&\leq cr_\varepsilon^{3-\delta}.
\end{align*}
\item On‌ $\Omega_4=\{x\in \widehat{X}:2r_\varepsilon<d(x)\}$ we have‌ ‌‌$\gamma_2(\dfrac{d}{r_\varepsilon})=1$, $\widehat{\omega}_\varepsilon=\omega_X$, $L_{\widehat{\omega}_\varepsilon}(\Gamma)=\lambda$  and
\begin{align*}
‌‌F_\varepsilon(0,0)&=S(\omega_X)-S(\omega_X)+\lambda\varepsilon^{2k-2}-\varepsilon^{2k-2}L_{\widehat{\omega}_\varepsilon}(\Gamma)-Q_{\widehat{\omega}_\varepsilon}(\nabla^2(\varepsilon^{2k-2}\Gamma))\\
&=-Q_{\widehat{\omega}_\varepsilon}(\varepsilon^{2k-2}\nabla^2\Gamma).
\end{align*}
By the approximation of $\Gamma$ and the assumption $4-2k<\delta<0$ we get $\|\Gamma\|_{\rho^{\delta}C^{4,\alpha}_{\frac{\widehat{g}_\varepsilon}{\rho^2}}(\Omega_4)}\leq  \rho^{-\delta}d^{4-2k}\leq (\sqrt{2})^{-\delta}r_\varepsilon^{4-2k-\delta}$, because $\rho\leq \sqrt{2}\varepsilon$ and $d\geq 2r_\varepsilon$. Similarly $\|\Gamma\|_{\rho^{2}C^{4,\alpha}_{\frac{\widehat{g}_\varepsilon}{\rho^2}}(\Omega_4)}\leq \rho^{-2}r_\varepsilon^{4-2k}\leq \varepsilon^{-2}r_\varepsilon^{4-2k}\leq r_\varepsilon^{2-2k-\frac{1}{k}}$ since $\rho\geq \varepsilon$.\\
Therefore Lemma \ref{bil} implies that
\begin{align*}‌
\|F_\varepsilon(0,0)\|_{\rho^{\delta-4}C^{0,\alpha}_{\frac{\widehat{g}_\varepsilon}{\rho^2}}(\Omega_4)}&=\|-Q_{\widehat{\omega}_\varepsilon}(\varepsilon^{2k-2}\nabla^2\Gamma) \|_{\rho^{\delta-4}C^{0,\alpha}_{\frac{\widehat{g}_\varepsilon}{\rho^2}}(\Omega_4)}\\
&\leq c\varepsilon^{4k-4}\|\Gamma\|_{\rho^{\delta}C^{4,\alpha}_{\frac{g_\varepsilon}{\rho^2}}(\Omega_4)}\|\Gamma\|_{\rho^{2}C^{4,\alpha}_{\frac{\widehat{g}_\varepsilon}{\rho^2}}(\Omega_4)}\\
& \leq c'\varepsilon^{4k-4}r_\varepsilon^{4-2k-\delta}r_\varepsilon^{4-2k-2-\frac{1}{k}}\\
&\leq c'\varepsilon^{4k-4}r_\varepsilon^{6-4k-\delta-\frac{1}{k}}\\
&\leq c'(r_\varepsilon^{1+\frac{1}{2k}})^{4k-4}r_\varepsilon^{6-4k-\delta-\frac{1}{k}}\\
&\leq c' r_\varepsilon^{2-\delta+2-\frac{3}{k}}\\
&\leq c' r_\varepsilon^{3-\delta}.
‌‌‌\end{align*}‌

\end{enumerate}

\begin{figure}
\centering
\begin{tikzpicture}[x=0.75pt,y=0.75pt,yscale=-1,xscale=1]

\draw [color={rgb, 255:red, 0; green, 0; blue, 8}, draw opacity=1] (74,251) -- (273.81,252.97);

\draw [draw opacity=0] (273.81,252.97) .. controls (273.81,252.75) and (273.81,252.54) .. (273.81,252.32) .. controls (273.51,215.95) and (302.76,186.23) .. (339.12,185.94) .. controls (375.42,185.65) and (405.09,214.78) .. (405.5,251.04) -- (339.65,251.79) -- cycle;
\draw [color={rgb, 255:red, 249; green, 10; blue, 10}, draw opacity=1] (273.81,252.97) .. controls (273.81,252.75) and (273.81,252.54) .. (273.81,252.32) .. controls (273.51,215.95) and (302.76,186.23) .. (339.12,185.94) .. controls (375.42,185.65) and (405.09,214.78) .. (405.5,251.04);

\draw [color={rgb, 255:red, 0; green, 0; blue, 0}, draw opacity=1] (405.5,251.04) -- (591,249);

\draw (340,186) -- (340,55.5);
\draw [shift={(340,53.5)}, rotate=90] [color={rgb, 255:red, 0; green, 0; blue, 0}, line width=0.75] (10.93,-3.29) .. controls (6.95,-1.4) and (3.31,-0.3) .. (0,0) .. controls (3.31,0.3) and (6.95,1.4) .. (10.93,3.29);

\draw (405.5,251.04) .. controls (448,237) and (570,223) .. (570,91);
\draw (273.81,252.97) .. controls (236,241) and (92,225) .. (93,98);
\draw (405.5,251.04) .. controls (433,236) and (488,218) .. (510,91);
\draw (273.81,252.97) .. controls (259,245) and (155,227) .. (155,97);

\draw (191,95) -- (295,203);
\draw (475,91) -- (383.65,203.79);

\draw (273,111) node [anchor=north west][inner sep=0.75pt] {$\si{\Omega}_{1}$};
\draw (369,110) node [anchor=north west][inner sep=0.75pt] {$\si{\Omega}_{1}$};
\draw (193,142) node [anchor=north west][inner sep=0.75pt] {$\si{\Omega}_{2}$};
\draw (444,144) node [anchor=north west][inner sep=0.75pt] {$\si{\Omega}_{2}$};
\draw (140,167) node [anchor=north west][inner sep=0.75pt] {$\si{\Omega}_{3}$};
\draw (501,166) node [anchor=north west][inner sep=0.75pt] {$\si{\Omega}_{3}$};
\draw (95,208) node [anchor=north west][inner sep=0.75pt] {$\si{\Omega}_{4}$};
\draw (546,207) node [anchor=north west][inner sep=0.75pt] {$\si{\Omega}_{4}$};
\draw (328,188.5) node [anchor=north west][inner sep=0.75pt] {$H_{1}$};
\draw (152,257) node [anchor=north west][inner sep=0.75pt] {$\textcolor[rgb]{0.05,0,1}{H_{2}}$};
\draw (486,255) node [anchor=north west][inner sep=0.75pt] {$\textcolor[rgb]{0.05,0,1}{H_{2}}$};
\draw (309,28) node [anchor=north west][inner sep=0.75pt] {$[0,+\infty)_{\varepsilon}$};
\end{tikzpicture}

\caption{\small Four different regions on $\widehat{X}$}
\end{figure}
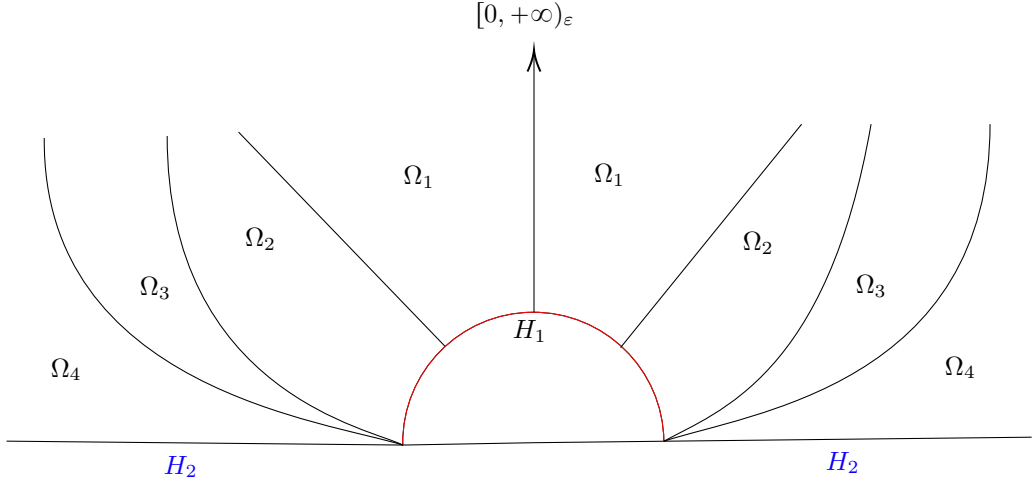
\end{proof}
Finally, to prove existence of a solution for the nonlinear equation
$$S(\widehat{\omega}_\varepsilon+\sqrt{-1}\dd u)=R,$$
for $\varepsilon>0$ small enough, we show that there exist $v_\varepsilon\in \rho^\delta C^{4,\alpha}_{\frac{\widehat{g}_\varepsilon}{\rho^2}}(\widehat{X})_0$ such that $$F_\varepsilon(v_\varepsilon,R_\varepsilon)=L_{\widehat{\omega}_\varepsilon}(v_\varepsilon)-R_\varepsilon.$$
 Define $\mathcal{N}_\varepsilon: \rho^\delta C^{4,\alpha}_{\frac{\widehat{g}_\varepsilon}{\rho^2}}(\widehat{X})_0\times \bR\to \rho^\delta C^{4,\alpha}_{\frac{\widehat{g}_\varepsilon}{\rho^2}}(\widehat{X})_0\times \bR$ by $\mathcal{N}_\varepsilon(v,R)=P_\varepsilon F_\varepsilon(v,R)$, where \linebreak $P_\varepsilon:=\widetilde{L}_\varepsilon^{-1}:\rho^{\delta-4}C^{0,\alpha}_{\frac{\widehat{g}_\varepsilon}{\rho^2}}(\widehat{X})\to \rho^\delta C^{4,\alpha}_{\frac{\widehat{g}_\varepsilon}{\rho^2}}(\widehat{X})_0\times \bR$ is as Proposition \ref{inv} on page \pageref{inv}. If we show that $\mathcal{N}_\varepsilon$ is a contraction, then by the Banach fixed point theorem, there exist unique $(v_\varepsilon,R)$, such that $\mathcal{N}_\varepsilon(v_\varepsilon,R)=(v_\varepsilon,R)$ or equivalently $F_\varepsilon(v_\varepsilon,R)=\widetilde{L}_\varepsilon(v_\varepsilon,R)
$. Since $\widetilde{L}_{\varepsilon}(v_\varepsilon,R)=L_{\widehat{\omega}_\varepsilon}(v_\varepsilon)-R$, then $F(v_\varepsilon)+R=L_{\omega_\varepsilon}(v_\varepsilon)$. Now we are going to show that $\mathcal{N}_\varepsilon$ is a contraction on a suitable domain. By Proposition \ref{lambda}, we must have that $R=R_\varepsilon$.

\begin{lemma}\label{forn}
There exist constants $c_0,\varepsilon_0>0$ such that for $0<\varepsilon<\varepsilon_0$,
$$\|\mathcal{N}_\varepsilon(v_1,R_1)-\mathcal{N}_\varepsilon(v_2,R_2)\|_{\rho^{\delta-4} C^{0,\alpha}_{\frac{\widehat{g}_\varepsilon}{\rho^2}}(\widehat{X})}\leq \dfrac{1}{2}
\|v_1-v_2\|_{\rho^{\delta} C^{0,\alpha}_{\frac{\widehat{g}_\varepsilon}{\rho^2}}(\widehat{X})},$$
for $(v_i,R)$ such that $\|v_i\|_{\rho^{2} C^{0,\alpha}_{\frac{\widehat{g}_\varepsilon}{\rho^2}}(\widehat{X})}<c_0$.
\end{lemma}
\begin{proof}
The proof is essentially the same as Lemma 23 in \cite{szekelyhidi2012blowing}. Since $P_\varepsilon$ is bounded independently of $\varepsilon$, we just need to control 
$$\|F_\varepsilon(v_1,R_1)-F_\varepsilon(v_2,R_2)\|_{\rho^{\delta-4} C^{0,\alpha}_{\frac{\widehat{g}_\varepsilon}{\rho^2}}(\widehat{X})}=
\|-Q_{\widehat{\omega}_\varepsilon}(\nabla^2 u_1)+Q_{\widehat{\omega}_\varepsilon}(\nabla^2 u_2)\|_{\rho^{\delta-4} C^{0,\alpha}_{\frac{\widehat{g}_\varepsilon}{\rho^2}}(\widehat{X})}.$$
By the mean value theorem there exist $t\in[0,1]$ such that for $X=(1-t)u_1+tu_2$,
$$S(\widehat{\omega}_\varepsilon+\sqrt{-1}\dd u_1)-S(\widehat{\omega}_\varepsilon+\sqrt{-1}\dd u_2)=L_{\widehat{\omega}_\varepsilon+\sqrt{-1}\dd X}(u_1-u_2).$$
Hence, this means that 
$$Q_{\widehat{\omega}_\varepsilon}(\nabla^2 u_1)-Q_{\widehat{\omega}_\varepsilon}(\nabla^2 u_2)=(L_{\widehat{\omega}_\varepsilon+\sqrt{-1}\dd X}-L_{\widehat{\omega}_\varepsilon})(u_1-u_2).$$

The linear operator $L_{\widehat{\omega}_\varepsilon}$ is bounded independently of $\varepsilon$, so
\begin{align*}
\|(L_{\widehat{\omega}_\varepsilon+\sqrt{-1}\dd X}-L_{\widehat{\omega}_\varepsilon})(u_1-u_2)\|_{\rho^{\delta-4} C^{0,\alpha}_{\frac{\widehat{g}_\varepsilon}{\rho^2}}(\widehat{X})}&\leq C (\|u_1\|_{\rho^2 C^{4,\alpha}_{\frac{\widehat{g}_\varepsilon}{\rho^2}}(\widehat{X})}+\|u_2\|_{\rho^2 C^{4,\alpha}_{\frac{\widehat{g}_\varepsilon}{\rho^2}}(\widehat{X})})\|u_1-u_2\|_{\rho^\delta C^{4,\alpha}_{\frac{\widehat{g}_\varepsilon}{\rho^2}}(\widehat{X})}\\
&\leq  2c'C\|u_1-u_2\|_{\rho^\delta C^{4,\alpha}_{\frac{\widehat{g}_\varepsilon}{\rho^2}}(\widehat{X})}\\
&\leq  2c'C\|v_1-v_2\|_{\rho^\delta C^{4,\alpha}_{\frac{\widehat{g}_\varepsilon}{\rho^2}}(\widehat{X})},
\end{align*}
where the constant \( c' \) can be chosen as small as we want, provided \( c_0 \) and \( \varepsilon \) are sufficiently small, since $u_i=\varepsilon^{2k-2}\gamma_2(\dfrac{d}{r_\varepsilon})\Gamma+v_i$ and when $\varepsilon\to 0$
$$\|\varepsilon^{2k-2}\gamma_2(\dfrac{d}{r_\varepsilon})\Gamma\|_{\rho^2 C^{k,\alpha}_{\frac{\widehat{g}_\varepsilon}{\rho^2}}(\widehat{X})}\leq c(\dfrac{\varepsilon}{r_\varepsilon})^{2k-2}=o(1).$$
By Properties \ref{inv}, the result follows.
\end{proof}
Now we define open set
$$\mathcal{U}_\varepsilon=\{v\in \rho^\delta C^{4,\alpha}_{\frac{\widehat{g}_\varepsilon}{\rho^2}}(\widehat{X})_0 : \|v\|_{\rho^\delta C^{4,\alpha}_{\frac{\widehat{g}_\varepsilon}{\rho^2}}(\widehat{X})}\leq {(1+2c)Cr_{\varepsilon}^{3-\delta}} \},$$
where $C$ is the independent bound of $P_\varepsilon$ and $c$ is the constant in Properties \ref{forf}.
\begin{proposition}
Suppose $\delta<0$ is sufficiently close to $4-2k$. Then for $\varepsilon>0$ sufficiently small, the map $\mathcal{N}_\varepsilon:\mathcal{U}_\varepsilon\to \mathcal{U}_\varepsilon$ is a contraction and therefore has a fixed point $v_\varepsilon$. 
\end{proposition}
\begin{proof}
Note that if $(v,R)\in \mathcal{U}_\varepsilon$, then we have
$$\|v\|_{\rho^2 C^{4,\alpha}_{\frac{\widehat{g}_\varepsilon}{\rho^2}}(\widehat{X})}\leq \varepsilon^{\delta-2}\|v\|_{\rho^\delta C^{4,\alpha}_{\frac{\widehat{g}_\varepsilon}{\rho^2}}(\widehat{X})}\leq (1+2c)C\varepsilon^{\delta-2}r_\varepsilon^{3-\delta}\leq c_0,$$
for sufficiently small $\varepsilon$, so Lemma \ref{forn} applies to $\mathcal{U}_\varepsilon$. It remains to check that $\mathcal{N}_\varepsilon(\mathcal{U}_\varepsilon)\subseteq \mathcal{U}_\varepsilon$. To do this, for any $v\in \mathcal{U}_\varepsilon$, Proposition \ref{forf} and Lemma \ref{forn} implies that:
\begin{align*}
\|\mathcal{N}_\varepsilon(v,R)\|_{\rho^\delta C^{4,\alpha}_{\frac{\widehat{g}_\varepsilon}{\rho^2}}(\widehat{X})_0}&\leq
\|\mathcal{N}_\varepsilon(v,R)-\mathcal{N}_\varepsilon(0,0)\|_{\rho^\delta C^{4,\alpha}_{\frac{\widehat{g}_\varepsilon}{\rho^2}}(\widehat{X})_0}+\|\mathcal{N}_\varepsilon(0,0)\|_{\rho^\delta C^{4,\alpha}_{\frac{\widehat{g}_\varepsilon}{\rho^2}}(\widehat{X})_0}\\
&\leq \dfrac{1}{2}\|(v,R)\|_{\rho^\delta C^{4,\alpha}_{\frac{\widehat{g}_\varepsilon}{\rho^2}}(\widehat{X})_0}+C\|F_\varepsilon(0,0)\|_{\rho^{\delta-4} C^{0,\alpha}_{\frac{\widehat{g}_\varepsilon}{\rho^2}}(\widehat{X})_0}\\
&\leq \dfrac{1}{2}((1+2c)Cr_{\varepsilon}^{3-\delta})+Cc(r_{\varepsilon}^{3-\delta})\leq (1+2c)Cr_{\varepsilon}^{3-\delta}.
\end{align*}
\end{proof}
The above proposition completes the proof of our main theorem.
\begin{theorem}\label{themain}
Suppose that $X$ is a compact cscK orbifold with no non-trivial holomorphic vector fields, and such that the set of singular points $Y$ of $X$ is of complex co-dimension $> 2$. Suppose, furthermore, that any point $p \in Y$ has a local orbifold uniformization chart of the form $\mathbb{C}^{n-k} \times \left(\mathbb{C}^k / \Gamma_{(-w_0,w)}\right)$, where $\Gamma_{(-w_0,w)}$ is a finite linear group of type $\mathcal{I}$. Then on the $(-w_0,w)$-weighted blow-up $\widehat{X}$ of $X$ along $Y$, the Kähler class $[\omega_X] - \varepsilon^2[E]$ admits a cscK metric for $\varepsilon > 0$ sufficiently small, where $E = \pi^{-1}(Y)$ is the exceptional divisor of the partial resolution $\pi: \widehat{X} \to X$.

\end{theorem}
Unless the singularity of type $\mathcal{I}$ is of the form $(-w_0,1,\ldots,1)$ for some $r \in \mathbb{N}$, $\widehat{X}$ also has a singularity of type $\mathcal{I}$ along a suborbifold of complex codimension $k$. However, as described on page \pageref{bihologrroupblow}, since the singularity is of type $\mathcal{I}$, we can find a sequence of weighted blow-ups 
$$
\widehat{X}_l \to \widehat{X}_{l-1} \to \ldots \to \widehat{X}_1 \to X,
$$
with $\widehat{X}_1 = \widehat{X}$ and $\widehat{X}_l$ smooth. Thanks to Proposition \ref{bihologrroupblow}, we can apply Theorem \ref{themain} iteratively to each $\widehat{X}_i$ to obtain on $\widehat{X}_l$ a cscK metric, which establishes Corollary \ref{coroc} in the introduction.
\bibliography{references}
\bibliographystyle{amsplain}

\end{document}